\documentclass[11pt]{elsarticle}
\usepackage{mystyleElsArticle}
\usepackage{mathrsfs}
\usepackage{booktabs}

\newcommand{\hmpo}{\mathrm{H}}
\newcommand{\hmpn}{\mathbf{H}}

\newcommand{\mindex}[1]{\boldsymbol{#1}}
\newcommand{\minda}{\mindex{\alpha}}
\newcommand{\mindb}{\mindex{\beta}}

\newcommand{\mset}[1]{\mathfrak{#1}}

\newcommand{\sJ}{\mset{J}}

\usepackage[margin=1.00 in]{geometry}    

\makeatletter
\def\ps@pprintTitle{%
  \let\@oddhead\@empty
  \let\@evenhead\@empty
  \let\@oddfoot\@empty
  \let\@evenfoot\@oddfoot
}
\makeatother

\begin{document}

\begin{frontmatter}
  \title{A model reduction method for multiscale elliptic PDEs with random coefficients using an optimization approach}

  \author[cms]{Thomas Y. Hou}
  \ead{hou@cms.caltech.edu}
  \author[hku]{Dingjiong Ma}
  \ead{martin35@hku.hk}
  \author[hku]{Zhiwen Zhang\corref{cor1}}
  \ead{zhangzw@hku.hk}

  \address[cms]{Computing and Mathematical Sciences, California Institute of Technology, Pasadena, CA 91125, USA. }
  \address[hku]{Department of Mathematics, The University of Hong Kong, Pokfulam Road, Hong Kong SAR.}
  \cortext[cor1]{Corresponding author}

\begin{abstract}
In this paper, we propose a model reduction method for solving multiscale elliptic PDEs with random coefficients in the multiquery setting using an optimization approach. The optimization approach enables us to construct a set of localized multiscale data-driven stochastic basis functions that give optimal approximation property of the solution operator. Our method consists of the offline and online stages. In the offline stage, we construct the localized multiscale data-driven stochastic basis functions by solving an optimization problem. In the online stage, using our basis functions, we can efficiently solve multiscale elliptic PDEs 
with random coefficients with relatively small computational costs. Therefore, our method is very efficient in solving target problems with many different force functions. The convergence analysis of the proposed method is also presented and has been verified by the numerical simulation.

\end{abstract}


\end{frontmatter}

\section{Introduction} \label{sec:introduction}
\noindent
Many physical and engineering applications involving uncertainty quantification (UQ) can be described by stochastic partial differential equations (SPDEs, i.e., PDEs driven by Brownian motion) or partial differential equations with random coefficients (RPDEs). In recent years, there has been an increased interest in the simulation of systems with uncertainties, and several numerical methods have been developed in the literature to solve UQ problems; see \cite{Ghanem:91,Xiu:03,babuska:04,matthies:05,WuanHou:06,Wan:06,Babuska:07,Webster:08,Xiu:09,Najm:09,sapsis:09,Zabaras:13,
Grahamquasi:2015} and references therein. These methods can be effective when the dimension of stochastic input variables is low or medium. However, their performance deteriorates dramatically when the dimension of stochastic input variables is high because of the curse of dimensionality. In addition, when SPDEs or RPDEs involving multiscale features, the problems become more challenging as it requires tremendous computational resources to resolve the small scales of the solutions.

Recently, some progress has been made in developing numerical method for multiscale PDEs with random coefficients. 
In \cite{hou2015heterogeneous}, Hou et. al. proposed a heterogeneous stochastic finite element method (HSFEM) framework to solve elliptic PDEs with random coefficient, which discretizes the equations using the heterogeneous coupling of spatial basis with local stochastic basis to exploit the local stochastic structure of the solution space. Zabaras et. al. proposed a stochastic variational multiscale method for diffusion in heterogeneous random media \cite{Zabaras:06,Zabaras:07}. They combined the generalized polynomial chaos (gPC) method with variational multiscale method to do model reduction. 
In \cite{Ghanem:08}, Ghanem et. al. considered the probabilistic equivalence and stochastic model reduction in multiscale analysis. In \cite{Kevrekidis:2003}, Kevrekidis et.al. applied the equation-free idea to study stochastic incompressible flow problems.

We also make some progress in developing numerical methods for multiscale PDEs with random coefficients by exploring the low-dimensional structure of the solutions and constructing problem-dependent basis functions to solve these RPDEs. In \cite{ChengHouYanZhang:13,ZhangCiHouMMS:15,ZhangHouLiu:15}, we proposed the data-driven stochastic methods to solve partial differential equations with high-dimensional random input and/or multiscale coefficients. We found that the data-driven stochastic basis functions can be used to solve the RPDEs with many different force functions. However, to obtain these data-driven stochastic basis functions, we need to choose a set of force functions, solve the RPDEs for each force function, and extract related basis information. Therefore, these data-driven stochastic methods cannot fully explore the information hidden in the RPDEs operator. This motivates us to develop more efficient data-driven stochastic methods to solve RPDEs.

In this paper, we propose a new approach to construct multiscale data-driven stochastic basis functions, which can be used to solve multiscale PDEs with random coefficients in the multi-query setting. Our method is inspired by the recent development in exploring intrinsic low-dimensional structures of the deterministic elliptic problems, see
\cite{HouPengchuan:2017,Peterseim:2014,Owhadi:2015,Owhadi:2017} and references therein. We shall use the following multiscale elliptic equation with random coefficient as an example to illustrate the main idea of our approach,
\begin{align}
\mathcal{L}^{\varepsilon}(x,\omega) u^{\varepsilon}(x,\omega) \equiv -\nabla\cdot(a^{\varepsilon}(x,\omega)\nabla u^{\varepsilon}(x,\omega)) = f(x), 
\quad x\in D, \omega\in\Omega. \label{MsStoEllip_Eq}
\end{align}
More details about the setting of the problem \eqref{MsStoEllip_Eq} will be discussed in Section \ref{sec:ModelReductionOC}. Our method consists of the offline and online stages. In the offline stage,
we apply an optimization approach to systematically construct localized multiscale data-driven stochastic basis functions on each patch associated with each coarse interior grid. The basis functions are energy minimizing functions on local regions of the domain. Thus, they give optimal approximation property of the stochastic multiscale solution operator.  We emphasize that the construction of the basis functions only depends on the RPDEs operator and does not depend on the forcing functions. Therefore, our method fully explores the information hidden in the RPDEs operator.

In the online stage, using our localized multiscale data-driven stochastic basis functions, we can efficiently solve the  multiscale RPDEs for a broad range of forcing functions.
Under some mild conditions, we prove that the basis functions have exponential decay property away from the center of local patch. In addition, we analyze the error between the numerical solution obtained from our method and the exact solution. Finally, we carry out several numerical experiments to show that our new method can offer considerable savings over many existing methods.

Though there are several prior works on the numerical methods for multiscale elliptic PDEs with random coefficients, the novelty of our paper is that we provide a unified approach to upscale the physical space and random space simultaneously. Moreover, the optimization approach that we use to construct the 
localized multiscale random basis functions (see Eqns. \eqref{OC_MsStocBasis_Obj}-\eqref{OC_MsStocBasis_Cons1}) 
enables us to explore the localized random structures in the solution space so that we can compute 
problems that are parametrized by many random variables.  More details will be demonstrated in Figure \ref{fig:ProfilesLocalBasis} and Section \ref{sec:NumericalExample3}. 
In addition, we prove that the multiscale random basis functions have an optimal approximation property. 
We remark that our paper shares some similarity with the idea presented in \cite{hou2015heterogeneous}, in the sense that we both investigate the localized structures in the random space. However, 
we construct the local random basis functions and multiscale basis function at the same time by solving a local optimization problem. 

The rest of the paper is organized as follows. In Section 2, we give a brief introduction of the generalized polynomial chaos method and data-driven stochastic method and operator compression. In Section 3, we first introduce the construction of localized basis function for a deterministic elliptic operator. Then, we present the derivation of our multiscale data-driven stochastic basis functions. Issues regarding the practical implementation of our method will be discussed. Some convergence analysis of our method will be discussed in section 4. In Section 5, we present numerical results to demonstrate the efficiency of our method. Concluding remarks are made in Section 6.

\section{Preliminaries} \label{sec:preliminaries}
\noindent
To make this paper self-contained, in this section we give a brief review of the generalized polynomial chaos (gPC) method and data-driven stochastic method. 

\subsection{The generalized polynomial chaos (gPC) method} \label{sec:gPC}
\noindent
In many physical and engineering problems, randomness generally comes from various independent
sources, so randomness in the PDE \eqref{MsStoEllip_Eq} is often given in terms of independent random variables.
We assume the randomness in the differential operator $\mathcal{L}^{\varepsilon}(x,\omega)$ is given in terms of $r$ independent random variables,  i.e., $\xi(\omega) = \bkr{\xi_1(\omega), \xi_2(\omega), \cdots, \xi_{r}(\omega)}$.  Without the loss of generality, we can further assume such independent random variables have the same distribution function $\rho(x)$. We get  $\mathcal{L}^{\varepsilon}(x,\omega)=\mathcal{L}^{\varepsilon}(x,\xi_{1}(\omega),...,\xi_{r}(\omega))$.  By the Doob-Dynkin lemma \cite{Oksendal:13}, the solution of Eq.\eqref{MsStoEllip_Eq} can still be represented by these random variables, i.e. $u^{\varepsilon}(x,\omega) = u^{\varepsilon}(x,\xi_{1}(\omega),...,\xi_{r}(\omega))$.

Let $\{\hmpo_i(\xi)\}_{i=1}^{\infty}$ denote the one-dimensional, $\rho(\xi)$-orthogonal polynomials, i.e.,
\begin{align*}
 \int_{\Omega}\hmpo_{i}(\xi)\hmpo_{j}(\xi)\rho(\xi)d\xi = \delta_{ij}.
\end{align*}
For some commonly used distributions, such as the Gaussian distribution and
the uniform distribution, such orthogonal polynomial sets are Hermite polynomials and  Legendre polynomials, respectively \cite{Ghanem:91,Xiu:03}. For general distributions, such polynomial set can be obtained by numerical methods \cite{Wan:06}.  Furthermore, by a tensor product representation, we can use the one-dimensional polynomial $\hmpo_i(\xi)$ to construct a set
of orthonormal basis $\hmpn_{\minda}(\xi) $'s of $L^2(\Omega)$ as follows
\begin{align}
  \hmpn_{\minda}(\xi) = \prod_{i=1}^{r}\hmpo_{\alpha_i}(\xi_i), \quad \minda \in \sJ_{r}^{\infty}, \label{tensor_orthonormal_basis}
\end{align}
where $\minda$ is a multi-index and $\sJ_{r}^{\infty}$ is a multi-index set of countable cardinality,
\begin{align*}
  \sJ_{r}^{\infty} = \bkc{\minda=\bkr{\alpha_1, \alpha_2, \cdots, \alpha_{r}} \,|\, \alpha_i \ge 0, \alpha_i \in \NN}.
\end{align*}
The zero multi-index corresponding to $\hmpn_{\vzero}(\xi)=1$, which is used to represent the mean of the solution. Clearly, the cardinality of $\sJ_{r}^{\infty}$ is infinite. For the purpose of numerical computations, we prefer a finite set of polynomials. There are many choices of truncations. One possible choice is the set of polynomials whose total orders are at most $p$, i.e.,
\begin{equation}
  \sJ_{r}^{p} = \bkc{\minda \,|\, \minda=\bkr{\alpha_1, \alpha_2, \cdots, \alpha_{r}}, \alpha_i \ge 0, \alpha_i \in \NN,\bkl{\alpha}=\sum_{i=1}^{r} \alpha_i \le p}. \label{gPC_index}
\end{equation}
The cardinality of $\sJ_{r}^{p}$ in \eqref{gPC_index} or the number of polynomial basis functions, denoted by $N_{gPC}=\bkl{\sJ_{r}^{p}}$, is equal to $\frac{(p+r)!}{p!r!}$. Another good choice is the sparse truncation method proposed in \cite{SchwabTodor:03,WuanHou:06}. We may simply write such a truncated set as $\sJ$ when no ambiguity arises.

By the Cameron-Martin theorem \cite{cameron:47}, we know the solution of Eq.\eqref{MsStoEllip_Eq} admits a polynomials chaos expansion,
\begin{align}
  u(x,\omega) = \sum_{\minda \in \sJ_{r}^{\infty}} v_{\minda}(x )\hmpn_{\minda}(\xi(\omega)) \approx
  \sum_{\minda \in \sJ_{r}^{p}} v_{\minda}(x )\hmpn_{\minda}(\xi(\omega)).   \label{Solu_gPC_expansion}
\end{align}
The expansion coefficients $u_{\minda}(x )$ can be obtained by the Galerkin projection with standard finite element basis functions. To simplify the notation, we write the polynomials chaos basis and its expansion coefficient in a vector form.
Substituting the expansion of $u(x,\omega)$ into Eq.\eqref{MsStoEllip_Eq}, multiplying both side by $\hmpn_{\mindb}(\xi(\omega)$, and taking the expectations, we obtain a coupled PDEs system for the expansion coefficient $u_{\minda}(x)$
\begin{align}
 -\nabla  \cdot (\mathfrak{T}_{\minda \mindb}^{\hmpn}(x)\nabla v_{\minda}(x)) = f(x,\theta)E[\hmpn_{\mindb}(\xi(\omega))], \quad
 \forall \mindb \in \sJ_{r}^{p}, \label{pre_gPC_precondition}
\end{align}
where the tensor $\mathfrak{T}_{\minda \mindb}^{\hmpn}(x)=E[a^{\varepsilon}(x,\xi(\omega))\hmpn_{\minda}(\xi(\omega))
\hmpn_{\mindb}(\xi(\omega))]$ and the Einstein summation is assumed. By solving Eq.\eqref{pre_gPC_precondition}, we can obtain the polynomial chaos solution.  This is the basic idea of the polynomial chaos method that has been successfully applied to many SPDE problems. However, the number of the basis  increases fast, which makes it inefficient for high-dimensional problems.

\subsection{The data-driven stochastic method} \label{sec:DSM}
\noindent
We recognized that limitations of the polynomial chaos method were caused by the predetermined (i.e., problem-independent)
basis functions in the solution expansion \eqref{Solu_gPC_expansion}. The use of predetermined basis functions has the advantage of being easy to implement, however, it has limitation in tackling more challenging high-dimensional SPDE and RPDE problems. We found that many high-dimensional SPDE and RPDE problems have certain low-dimensional structures, in the sense of Karhunen-Lo\`{e}ve expansion (KLE) \cite{Karhunen:47,Loeve:78}, which suggest the existence of reduced-order models and better formulations for efficient numerical methods. Motivated by this observation, we have made progress in developing data-driven stochastic method (DSM) to solving high-dimensional RPDE problems \cite{ChengHouYanZhang:13,ZhangHuHouLinYanCiCP:14,ZhangCiHouMMS:15,ZhangHouLiu:15}.

In the DSM, we explore the low-dimensional structures of the high-dimensional RPDE solutions and
construct a set of problem-dependent stochastic basis functions $\{A_{i}(\omega)\}_{i=1}^{m}$, which
give a compact representation for a broad range of forcing functions.
The DSM consists of offline and online stages. In the offline stage, we use the KLE of RPDE solutions to construct a set of data-driven stochastic basis functions. Specifically, we assume the force function 
$f(x)$ in Eq.\eqref{MsStoEllip_Eq} can be parameterized by $f(x)\approx \sum_{i=1}^{K}c_if_i(x)$.  
With such a parametrization of $f(x)$, we begin our construction of the stochastic basis $\{A_i(\omega)\}_{i=0}^{m_1}$ based on the KL expansion of the RPDE solution of \eqref{MsStoEllip_Eq} with $f_{1}(x)$ as a forcing function. We propose an error analysis to evaluate the completeness of the data-driven basis $\{A_i(\omega)\}_{i=0}^{m}$. Multiple trial functions
$f_{i}(x)$, $i=2,...,K$, are used to enrich the stochastic basis and improve the accuracy of our method. When this enriching process is done, we obtain the data-driven stochastic basis functions, denoted by $\{A_i(\omega)\}_{i=0}^{m}$, where $A_0(\omega)\equiv1$ is used to represent the mean of the solution.
In the online stage, we solve Eq.\eqref{MsStoEllip_Eq} for any given forcing function by projecting the stochastic solution into the data-driven stochastic basis,
\begin{align}
    u(x,\omega) \approx \sum_{i=0}^{m} u_i(x)A_i(\omega),
\end{align}
Then, we use the Galerkin projection to derive a coupled deterministic system of PDEs for $u_i(x)$ and solve this system by any standard numerical method. Compared with the gPC method, the ratio of the computational complexities between DSM (in online stage) and gPC is of an order $O((m/N_p)^2)$, where $m$ and $N_{p}$ are the numbers of the basis used in the DSM and gPC, respectively. We expect that $m$ is much smaller than $N_p$ when the effective random dimension of the solution is small. More details about the construction of data-driven stochastic basis can be found in \cite{ChengHouYanZhang:13}

The DSM is very efficient in solving RPDEs with many different force functions. However, 
it does not fully explore the information hidden the Green's function of  the RPDEs \eqref{MsStoEllip_Eq}. Since we need to choose a set of trial functions when we construct the data-driven stochastic basis functions.

\section{Derivation of the multiscale data-driven stochastic basis functions} \label{sec:ModelReductionOC}
\noindent
In this section, we shall provide the detailed derivation of the multiscale data-driven stochastic basis functions. We first introduce the idea of constructing adaptive basis functions using an optimization approach for a deterministic elliptic operator. Then, we discuss the derivation of the multiscale data-driven stochastic basis functions using the optimization approach, which provides an effective model reduction method for the stochastic multiscale elliptic PDEs \eqref{MsStoEllip_Eq}.

\subsection{Localized basis functions for elliptic operator} \label{sec:OCelliptic}
\noindent
Consider the deterministic elliptic equation with the homogeneous Dirichlet boundary conditions
\begin{align}
\mathcal{L}u(x) \equiv -\nabla\cdot(a(x)\nabla u(x)    = f(x), \quad x\in D,  \label{DeterEllip_Eq}
\end{align}
where the force function $f(x)\in L^{2}(D)$. For a self-adjoint, positive definite operator $\mathcal{L}$, the Eq.\eqref{DeterEllip_Eq} has a unique solution, denoted by $\mathcal{L}^{-1}u(x)$. In practice,
Eq.\eqref{DeterEllip_Eq} can be solved using numerical methods, such as a finite element method or finite difference method. When the operator $\mathcal{L}$ contains multiscale features, i.e,
$\mathcal{L}(\cdot)\equiv -\nabla\cdot(a^{\epsilon}(x)\nabla \cdot) $, where $0<\epsilon\ll1$ is the small-scale parameter, one can  construct multiscale basis functions to solve it efficiently \cite{HouWu:97}.
Recently, much effort has been devoted to build localized basis functions that give optimal approximation property of the solution operator. The basis functions are energy minimizing functions on local patches and provide optimal
convergence rate for elliptic problems in the energy norm, see \cite{BabuskaLipton:2011,Peterseim:2014,HouPengchuan:2017,Owhadi:2017}.

To construct such localized basis functions,  we first partition the physical domain $D$ into a set of regular coarse elements with mesh size $H\gg \epsilon$. For example, we divide $D$ into a set of
non-overlapping triangles $\mathcal{T}_{H}=\{K_e\}_{e=1}^M$ such that no vertex of one triangle lies in the interior of the edge of another triangle. In each element $K$, we define a set of nodal basis $\{\varphi_{j,K},j=1,...,d\}$ with $d$ being the number of nodes of the element. We neglect the subscript $K$ for notational simplicity. The functions $\varphi_{i}(x)$ are called measurement functions, 
which are chosen as the characteristic functions on each coarse element in \cite{HouPengchuan:2017,Owhadi:2017} and piecewise linear basis functions in \cite{Peterseim:2014}.
In \cite{LiZhangCiCP:18}, we find that using the nodal basis functions will reduce the error so we 
use the same setting in this paper. 
  
Let $\mathcal{N}$ denote the set of interior vertices of $\mathcal{T}_{H}$ and $N_{x}$ be the number of interior vertices. For every vertex $x_i\in\mathcal{N}$, let $\varphi_{i}(x)$ denote the corresponding nodal basis function, i.e., $\varphi_{i}(x_j)=\delta_{ij}$. For now, we assume that all the nodal basis functions $\varphi_{i}(x)$ are continuous across the boundaries of the elements, so that $$V^{H}=\{\varphi_{i}(x):i=1,...,d \}\subset H_{0}^{1}(D).$$
In order to obtain localized basis functions, we define local patches of coarse elements. For $r>0$, let $S_r$
be the union of the elements $K_e$ intersecting the ball $B(x_i,r)$ which is centered at $x_i\in\mathcal{N}$ and of radius $r$. Let $\psi_i^{loc}(x)$ be the minimizer of the following optimization problem:
\begin{align}
\psi_i^{loc}& =\underset{\psi \in H_{0}^{1}(D)}{\arg\min} ||\psi||_a^2    \label{OC_MsBasis_Obj}\\
 \text{s.t.}\ &\int_{S_r}\psi \varphi_{j}= \delta_{i,j},\ \forall 1\leq j \leq N_{x}, \label{OC_MsBasis_Cons1}\\
&\psi(x)= 0, \ x\in D\backslash S_r.\label{OC_MsBasis_Cons2}
\end{align}
In the objective function, $||\psi||_a$ is the energy norm associated with the differential operator $\mathcal{L}$
defined by $||\psi||_a^{2}=\int_{D}\nabla\psi^{T}(x)a(x)\nabla\psi(x)dx$.  We neglect the superscript $\epsilon$ for notational simplicity. In \eqref{OC_MsBasis_Cons2}, we have used the fact that the basis function $\psi$ defined on the whole domain has exponential decay property so that we can localize the basis function $\psi$ to its associated patches $S_r$ and set $\psi(x)=0$, $x\in D\backslash S_r$. On a uniform mesh of size $H$, if we choose the diameter $r$ of each patch to be $H\log(1/H)$, we will obtain a convergence rate of order $O(H)$ in the energy norm.
We have incorporated the boundary condition of the elliptic problem into the above optimization problem through the definition of the solution space $H_{0}^{1}(D)$ and the energy norm $||\cdot||_a$. When we obtain the localized basis functions, we  use the Galerkin method to solve Eq.\eqref{DeterEllip_Eq}.


Recall that the  number of localized basis functions is equal to the number of interior vertices. We can obtain the localized basis functions $\psi_i^{loc}$, $i=1,...,N_{x}$, by solving the optimization problem \eqref{OC_MsBasis_Obj}-\eqref{OC_MsBasis_Cons2}. In general, we cannot solve this  optimization problem analytically as it is an optimization problem in infinite dimensional space.  Thus, we have to resort to numerical methods.
In the finite element framework, we partition the physical domain $S_r$ into a set of non-overlapping fine triangles with size $h \ll \epsilon$. Then, we use standard finite element method to discretize $\psi_{i}$, $\varphi_{j}$, $1\leq i,j\leq N_{x}$. In the discrete level, the optimization problem \eqref{OC_MsBasis_Obj}-\eqref{OC_MsBasis_Cons2} is reduced to a constrained quadratic optimization problem, which can be efficiently solved using Lagrange multiplier methods.

\begin{remark}
Since the localized basis functions are discretized on the fine mesh have already resolved the small-scale  variations in the problem, they can incorporate the small-scale features of the problem into the solutions obtained on the coarse mesh. Therefore, the localized basis functions effectively compress the multiscale solution space and provide a robust upscaling method.
\end{remark}

\begin{remark}
For a self-adjoint, elliptic different operator $\mathcal{L}$ and its corresponding energy norm, the constrained quadratic optimization problem obtained from the numerical discretization of \eqref{OC_MsBasis_Obj}-\eqref{OC_MsBasis_Cons2} is a  convex optimization problem. The convexity property makes the problem easy to solve since a local minimum must be the global minimum.
\end{remark}

\subsection{Compression of the random space} \label{sec:CoarseningTensorSpace}
\noindent
We shall derive localized multiscale data-driven stochastic basis functions that can efficiently compress the solution space of multiscale elliptic PDEs with random coefficients. We consider the following problem posed on a bounded domain $D$ and subject to a Dirichlet boundary condition,
\begin{align}
-\nabla\cdot(a^{\varepsilon}(x,\omega)\nabla u^{\varepsilon}(x,\omega)) & = f(x), \quad x\in D, \omega\in\Omega,  \label{MsStoEllip_OCbasis_Eq} \\
u^{\varepsilon}(x,\omega)&= 0, \quad x\in \partial D, \label{MsStoEllip_OCbasis_BC}
\end{align}
where the physical domain $D$ is assumed to be a convex polygon domain in $R^d$, $d=2,3$, and $\Omega$ is a sample space. The forcing function $f(x)$ is assumed to be in $L^2(D)$ (not just $H^{-1}(D)$) because this is necessary for the compactness of the solution space. We also assume that the problem is uniformly elliptic almost surely, namely, there exist $a_{\min}, a_{\max}>0$, such that
\begin{align}
	P(\omega\in \Omega: a(x, \omega)\in [a_{\min},a_{\max}], \forall x \in D) = 1.
\label{asUniformlyElliptic}
\end{align}
Note that we do not make any assumption on the regularity of the coefficient $a(x)\in L^\infty(D)$, which can be arbitrarily rough for each realization $a(x,\omega)$. Furthermore, we assume $a(x,\omega)$ is parameterized in terms of $r$ independent random variables,  i.e., $a(x,\omega) =  a(x,\xi_{1}(\omega),...,\xi_{r}(\omega))$, then
according to the Doob-Dynkin lemma the solution can be represented by these random variables, i.e.,
$u(x,\omega) = u(x,\xi_{1}(\omega),...,\xi_{r}(\omega))$.

Before we proceed the construction of the multiscale data-driven stochastic basis, we introduce the concept of
compression of the random space.  To demonstrate the main idea, let us first consider the localized basis functions and the compression of the physical space.  The localized basis functions $\psi_i^{loc}(x)$, $i=1,...,N$ are associated with the interior vertices of  coarse elements with mesh size $H$. However, when we solve the problem \eqref{OC_MsBasis_Obj}-\eqref{OC_MsBasis_Cons2} numerically, each basis function $\psi_i^{loc}(x)$ is represented by a set of finite element basis functions defined on $S_r$, i.e., $\psi_{i}^{loc}(x)=\sum_{s=1}^{N_s}b_s^{i}\lambda_{s}(x)$, where $\lambda_{s}(x)$ are defined on  fine elements with mesh size $h$ and $N_s$ is the number of fine-scale finite element basis functions. The small-scale information has already been captured by the basis functions $\psi_i^{loc}(x)$.

In the random space, we also have a similar structure. Consider a probability space $\{\Omega,\mathfrak{F},P\}$ and $\xi(\omega) = \bkr{\xi_1(\omega), \xi_2(\omega), \cdots, \xi_{r}(\omega)}$ is a set of independent identically distributed (i.i.d.) random variables on $\Omega$. From this parametrization, we define a set of polynomial basis functions, $\hmpn_{\minda}(\xi(\omega))$, $\minda \in \sJ_{r}^{\infty}$. By the  Cameron-Martin theorem \cite{cameron:47}, any random variable $X(\omega)$ with finite variance can be represented as $X(\omega) = \sum_{\minda \in \sJ_{r}^{\infty}} c_{\minda}\hmpn_{\minda}(\xi(\omega)) \approx  \sum_{\minda \in \sJ_{r}^{p}} c_{\minda} \hmpn_{\minda}(\xi(\omega))$. In general, the polynomial basis functions can be applied to most random systems.  However, the number of the basis functions grows exponentially fast with respect to $r$ and $p$.

Our recent studies on the elliptic PDEs with random coefficients show that when the random coefficients have some regularity in the stochastic space and the force function $f(x)$ is in $L^2(D)$, the solutions have some sparse or low-dimensional representation. Therefore, we can compress the dimension of the solution space by constructing data-driven stochastic basis functions \cite{ChengHouYanZhang:13,ZhangHuHouLinYanCiCP:14,ZhangCiHouMMS:15,ZhangHouLiu:15}.
Compared to the polynomial basis functions, the data-driven stochastic basis functions are more efficient in solving elliptic PDEs with random coefficients.




\subsection{Construct multiscale data-driven stochastic basis functions} \label{sec:ConstructDSMbasis}
\noindent
In this section, we shall propose a new approach to compress the physical and random space simultaneously. In the physical space, we use the same notations defined in section \eqref{sec:OCelliptic}. The key is to construct a set of multiscale data-deriven stochastic basis functions.
In Eq.\eqref{OC_MsBasis_Cons2}, the nodal basis functions $\varphi_{j}(x)$,  are defined on coarse elements.
The number of coarse elements controls the degree-of-freedom that we want to compress in the physical space.

In the random space, we choose a set of low-order polynomials, e.g.,  $\{\hmpn_{\minda}(\xi(\omega)),\bkl{\alpha}=\sum_{i=1}^{r} \alpha_i \le p\}$, $p$ is small and the total number is $N_{\xi}$. For notation simplicity, we reorder this set of low-order polynomials
as $\{\hmpn_{l}(\xi(\omega)), 1\leq l \leq N_{\xi}\}$.
Let $\psi_{i,k}(x,\xi(\omega))$ be the minimizer of the following optimization problem, where indices $i$ and $k$ are associated with physical and random dimensions
\begin{align}
& \psi_{i,k}(x,\xi(\omega)) =\underset{\psi \in H_{0}^{1}(D)\otimes L^{2}(\Omega)}{\arg\min}
||\psi(x,\xi(\omega))||_a^2    \label{OC_MsStocBasis_Obj}\\
\text{s.t.}\ & \mathbb{E}\big[\int_{D}\psi(x,\xi(\omega))\varphi_{j}(x)\hmpn_{l}(\xi(\omega))
 dx\big]= \delta_{i,j}\delta_{k,l},\  \forall 1\leq j \leq N_{x},
 1\leq l \leq N_{\xi}, \label{OC_MsStocBasis_Cons1}
\end{align}
where the number of the basis functions is $N_{x}N_{\xi}$. In the objective function \eqref{OC_MsStocBasis_Obj}, $||\psi||_a$ is the energy norm associated with the differential operator $\mathcal{L}$ defined by $||\psi||_a^{2}=\mathbb{E}\big[\int_{D}\nabla\psi^{T}(x,\omega)a(x,\omega)\nabla\psi(x,\omega)dx\big]$. In the constraint \eqref{OC_MsStocBasis_Cons1}, the condition $\delta_{k,l}$ ensures that different $\psi_{i,k}$ with different $k$ are uncorrelated, which is important in the optimality approximation property stated in the Lemma \ref{LemmaOptimalApproximation} and Theorem \ref{TheoremOptimalApproximation}.  We neglect the superscript $\epsilon$ for notational simplicity. The boundary condition of the elliptic problem has already been incorporated in the above optimization problem through the definition of the solution space $H_{0}^{1}(D)\otimes L^{2}(\Omega)$ and the corresponding energy norm $||\cdot||_a$. Under the uniformly elliptic assumption, our multiscale stochastic basis functions $\psi_{i,k}(x,\xi(\omega))$ still maintain the energy minimizing and exponential decay property in physical space. The proof will be provided later. 

We solve the optimization problem \eqref{OC_MsStocBasis_Obj}-\eqref{OC_MsStocBasis_Cons1} using numerical methods. Specifically, we apply the stochastic finite element method (SFEM) \cite{Ghanem:91} to discretize $\psi_{i,k}(x,\xi(\omega))$ and represent
\begin{align}
\psi_{i,k}(x,\xi(\omega)) = \sum_{s=1}^{N_{h}}\sum_{\minda \in \sJ_{r}^{p}}
b_{s,\minda}^{i,k}\lambda_{s}(x)\hmpn_{\minda}(\xi(\omega)),  \label{RepresentOCMaStocBasis}
\end{align}
where $\lambda_{s}(x)$ are FEM nodal basis functions defined on fine elements with mesh size $h$, $N_{h}$ is the number of FEM nodal basis, and $\hmpn_{\minda}(\xi(\omega)$ are the high-order polynomials basis with cardinality $\bkl{\sJ_{r}^{p}}$. In the discrete level, the optimization problem \eqref{OC_MsStocBasis_Obj}-\eqref{OC_MsStocBasis_Cons1} is reduced to a constrained quadratic optimization problem, which can be efficiently solved using the Lagrange multiplier method. After we obtain the basis functions $\{\psi_{i,k}(x,\xi(\omega))\}$, we can use the Galerkin method to solve Eq.\eqref{MsStoEllip_OCbasis_Eq} for many different force functions.

The SFEM is an efficient method to solve elliptic PDEs with random coefficients. However, the size of the coupled linear system grows dramatically when we increase the degree of freedom in  the random space and/or the physical space, which limits its scope in tackling more challenging high-dimensional SPDE and RPDE problems. With our multiscale data-driven stochastic basis, we compress both the random space and physical space so that we can efficiently solve the Eq.\eqref{MsStoEllip_OCbasis_Eq}.

\subsection{Exponential decay of the basis functions in physical space} \label{sec:ExponentialDecay}
\noindent
In this section we shall show that the basis functions $\psi_{i,k}(x,\xi(\omega))$ obtained by the optimization problem \eqref{OC_MsStocBasis_Obj}-\eqref{OC_MsStocBasis_Cons1} have exponential decay property with respect to node $x_i$ in the physical space. Recall that any feasible solution has a gPC expansion $\psi(x,\xi(\omega)) = \sum_{\minda \in \sJ_{r}^{p}}
g_{\minda}(x)\hmpn_{\minda}(\xi(\omega)) $ with $ |\sJ_{r}^{p}| \gg N_{\xi} $.
We first fix $k$ and substitute $\psi(x,\xi(\omega))$ into the constraint \eqref{OC_MsStocBasis_Cons1} and
compute the expectation with respect to $\hmpn_{l}(\xi(\omega))$, $1\leq l \leq N_{\xi}$. We know
the basis functions $\psi_{\cdot,k}(x,\xi(\omega))$ satisfy
\begin{align}
\psi_{\cdot,k}(x,\xi(\omega))&=g_{\cdot,k}(x)\hmpn_{k}(\xi(\omega))
+g_{\cdot,N_{\xi+1}}(x)\hmpn_{N_\xi+1}(\xi(\omega))+...
+g_{\cdot,|\sJ_{r}^{p}|}(x)\hmpn_{|\sJ_{r}^{p}|}(\xi(\omega)) \nonumber \\
&= g_{\cdot,k}(x)\hmpn_{k}(\xi(\omega)) + R_{\cdot,k}(x,\xi(\omega)), \label{RepresentOCMaStocBasis_Hk}
\end{align}
where $R_{\cdot,k}(x,\xi(\omega))$ denotes the remaining terms in $\psi_{\cdot,k}(x,\xi(\omega))$ except for
$g_{\cdot,k}(x)\hmpn_{k}(\xi(\omega))$. Eq.\eqref{RepresentOCMaStocBasis_Hk} reveals the structures of the basis $\psi_{\cdot,k}(x,\xi(\omega))$ in the
random space, which consists of $\hmpn_{k}(\xi(\omega))$ and a linear combination of high-order gPC basis.

To study the property of the basis $\psi_{i,k}(x,\xi(\omega))$ in the physical space, we consider
the optimization problem realization by realization.
Let $\xi(\omega^{*})$ be a random sample, then constraint \eqref{OC_MsStocBasis_Cons1} is reduced to the following form
\begin{align}
\int_{D}\big(g_{i,k}(x)\hmpn_{k}(\xi(\omega^{*}))+R_{i,k}(x,\xi(\omega^{*}))\big)\varphi_{j}(x)
dx= \delta_{i,j},\quad \forall 1\leq j \leq N_{x}.  \label{RepresentOCMaStocBasis_Hk2}
\end{align}
We consider the physical component first since the form \eqref{RepresentOCMaStocBasis_Hk2} almost surely 
satisfies the constraint in the random space. For a deterministic elliptic PDE, it has been proven that  $\psi_{\cdot,k}(x,\xi(\omega^{*}))=g_{i,k}(x)\hmpn_{k}(\xi(\omega^{*}))+R_{i,k}(x,\xi(\omega^{*})$ will decay exponentially fast away from physical node $x_i$ \cite{Peterseim:2014,Owhadi:2017,HouPengchuan:2017}. Therefore, we know that $\psi_{\cdot,k}(x,\xi(\omega))$ has the exponential decay property in the physical space almost surely. Furthermore, we get that each gPC expansion coefficient (e.g., $g_{i,k}(x)$ or $g_{i,N_{\xi+n}}(x)$, $n\geq1$) will decay exponentially fast away from physical node $ x_i $. Finally, we reach to the following proposition.
\begin{proposition}
	Consider the construction of the multiscale data-driven stochastic basis functions in  \eqref{OC_MsStocBasis_Obj}-\eqref{OC_MsStocBasis_Cons1}, let
	$H$ be the mesh size of the  coarse elements,  $S_r$ be the union of the elements $K_e$ intersecting the ball $B(x_i,r)$ which is centered at $x_i\in\mathcal{N}$ and of radius $r>0$.
	The basis function $\psi_{i,k}(\cdot,\xi(\omega))$ satisfies the following property,
	\begin{align}
		\int_{D\cap(B(x_i,r))^{c}}
		\mathbb{E}\big[\nabla\psi_{i,k}^{T}(\cdot,\xi(\omega))a(\cdot,\omega)\nabla\psi_{i,k}(\cdot,\xi(\omega))\big]dx \leq
		e^{1-\frac{r}{lH}}\int_{D}
		\mathbb{E}\big[\nabla\psi_{i,k}^{T}(\cdot,\xi(\omega))a(\cdot,\omega)\nabla\psi_{i,k}(\cdot,\xi(\omega))\big]dx,
		\label{ExponentialDecayStochastic}
	\end{align}
	with $l=1+\frac{e}{\pi}\sqrt{\frac{\lambda_{max}}{\lambda_{max}}}(1+2^{3/2}(2/\delta)^{1+d/2})$ and
	$0<\delta<1$ is a parameter such that the coarse element contains a ball with radius $\delta H$.
\end{proposition}
The proof for the deterministic case is given in \cite{Owhadi:2017}, which is based on an iterative Caccioppoli-type argument. This proposition enables us to localize the basis functions so the corresponding stiffness matrix is sparse. Specifically, we solve the following optimization problems,
\begin{align}
& \psi_{i,k}^{loc}(x,\xi(\omega)) =\underset{\psi \in H_{0}^{1}(D)\otimes L^{2}(\Omega)}{\arg\min}
||\psi(x,\xi(\omega))||_a^2    \label{OC_MsStocLocalBasis_Obj}\\
\text{s.t.}\ & \mathbb{E}\big[\int_{S_r}\psi(x,\xi(\omega))\varphi_{j}(x)\hmpn_{l}(\xi(\omega)
 dx\big]= \delta_{i,j}\delta_{k,l},\  \forall 1\leq j \leq N_{x},
 1\leq l \leq N_{\xi}, \label{OC_MsStocLocalBasis_Cons1}\\
&\psi(x,\xi(\omega))= 0, \ x\in D\backslash S_r.\label{OC_MsStocLocalBasis_Cons2}
\end{align}
where $S_r$ is the union of the elements $K_e$ intersecting the ball $B(x_i,r)$ which is centered at $x_i\in\mathcal{N}$ and of a radius $r$. The construction \eqref{OC_MsStocLocalBasis_Obj}-\eqref{OC_MsStocLocalBasis_Cons2} enable us to obtain basis functions that capture the localized random structure in the solution space. To illustrate, we consider the problem \eqref{MsStoEllip_OCbasis_Eq}-\eqref{MsStoEllip_OCbasis_BC} on $D=[0,1]\times[0,1]$ with a coefficient that has a localized random structure, i.e., 
\begin{align}
a^{\varepsilon}(x,y,\omega) =0.1 + \frac{2+p\sin{2\pi x}/\varepsilon}{2-p\cos{2\pi y}/\varepsilon}\xi(\omega)\textbf{1}_{D_1}(x,y), \label{MsDSM_NumEX3_Coef}
\end{align}
where $p=1.5$, $\varepsilon=1/31$, $\xi(\omega)\in U[0,1]$, $\textbf{1}_{D_1}(x,y)$ is an indicator function, and
$D_1 = [1/8,3/8]\times [1/8,3/8]$. Obviously the random structure in $a^{\varepsilon}(x,y,\omega)$ is localized in a sub-domain.

In Fig.\ref{fig:ProfilesLocalBasis}, we show the profiles of multiscale data-driven stochastic basis functions associated with two different coarse vertices $x_{near}=(\frac{1}{8},\frac{1}{8})$ and $x_{far}=(\frac{3}{8},\frac{7}{8})$. For simplicity, we introduce the following notations, $\psi_{L}^{near}(x,y)$, $\psi_{H}^{near}(x,y)$, $\psi_{L}^{far}(x,y)$, and $\psi_{H}^{far}(x,y)$. Here, `near' (`far') means the basis is associated with the node $x_{near}$ ($x_{far}$), which is near (far from) the random region $D_1$. `L' means the projection of the basis function on a polynomial of order $0$, while the `H' means the projection of the basis function on a polynomial of order $5$. In our experiment, we find that $\psi_{H}^{near}(x,y)$ still has large magnitude as it is close to $D_1$, but $\psi_{H}^{far}(x,y)$ has a negligible magnitude as it is far away. Therefore, our method automatically captures localized stochastic structures in the solution space, which allows us to further reduce the computational cost. 
\begin{figure}[tbph]
	\centering
	\includegraphics[width=0.60\textwidth]{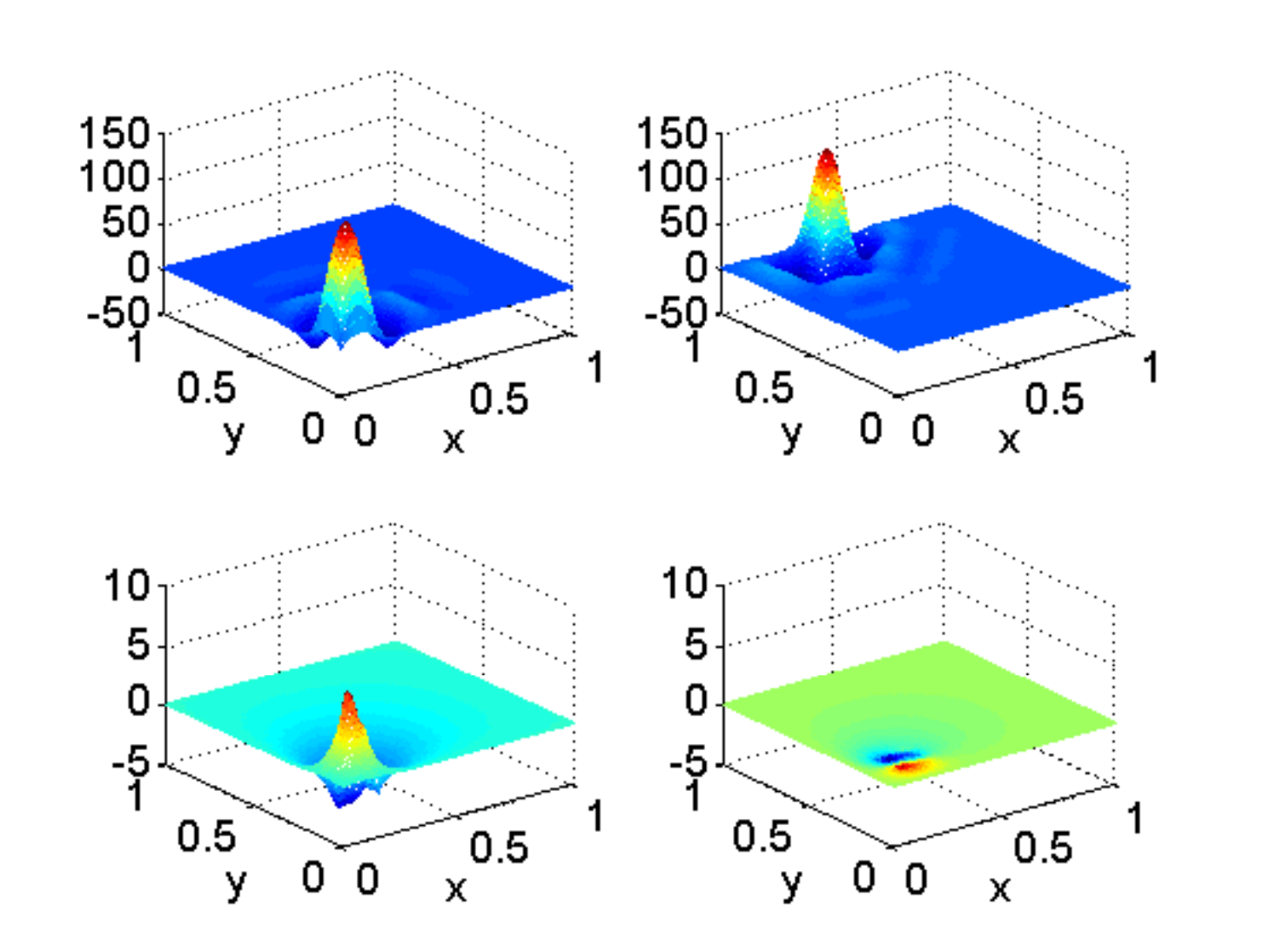}
	\caption{Profiles of the multiscale data-driven stochastic basis functions. Upper left: $\psi_{L}^{near}(x,y)$.  Lower left: $\psi_{H}^{near}(x,y)$. Upper right: $\psi_{L}^{far}(x,y)$.  Lower right: $\psi_{H}^{far}(x,y)$.
		Here, `near' (`far') means the basis is associated with the node $x_{near}$ ($x_{far}$), `L' (`H') means the projection of the basis function on a low (high) order polynomial basis.}
	\label{fig:ProfilesLocalBasis}
\end{figure}

\subsection{Optimal approximation property of the basis functions} \label{sec:VariationalProperties}
\noindent
We shall show that the basis functions $\psi_{i,k}(x,\xi(\omega))$ have some remarkable optimal approximation property that can be used for practical computation. First, we define the tensor basis $\phi_{j,l}(x,\xi(\omega))=\varphi_{j}(x)\hmpn_{l}(\xi(\omega)$, where $\varphi_{j}(x)$ and $\hmpn_{l}(\xi(\omega)$, $1\leq j \leq N_{x}, 1\leq l \leq N_{\xi}$ are defined in \eqref{OC_MsStocBasis_Cons1}. Then, we define the inner products in the tensor space as
\begin{align}
(u,v)& =\mathbb{E}\big[\int_{D}u(x,\omega)v(x,\omega)dx\big], \label{inner1}\\
(u,v)_a& =\mathbb{E}\big[\int_{D}\nabla u^{T}(x,\omega)a(x,\omega)\nabla v(x,\omega)dx\big].
\label{inner2}
\end{align}
Equipped with the above definitions, we know that \eqref{OC_MsStocBasis_Cons1} implies
$(\phi_{j,l},\psi_{i,k})=\delta_{i,j}\delta_{k,l}$. We define the solution space for Eq.\eqref{MsStoEllip_OCbasis_Eq} as
$V=H_{0}^{1}(D)\otimes L^{2}(\Omega)$ and the space spanned by the multiscale data-driven stochastic basis functions as
$V_{DSM}=span\{\psi_{i,k}(x,\xi(\omega))\}$. We also define the space
\begin{align}
V_0:=\{v\in V : (v,\phi_{j,l})=0, ~ 1\leq j \leq N_{x},
 1\leq l \leq N_{\xi}\}. \label{ResidualSpace}
\end{align}
Note that $V_0$ is the subspace of $V$. Throughout the paper we also assume that
the force functions belong to the space spanned by the basis $\varphi_i$, i.e. $ f(x)\in span\{\varphi_i\}_{i=1}^{N_x}$.
Namely, $f(x)$ can be approximated by $span\{\varphi_i\}_{i=1}^{N_x}$ up to an $O(H)$ error. Using the integration by parts and the definition of \eqref{inner2}\eqref{ResidualSpace}, we easily find that for every $v\in V_0$
\begin{align}
(\psi_{i,k},v)_a=0, \quad \quad \forall 1\leq i \leq N_{x},
 1\leq k \leq N_{\xi}. 
\label{orthogonal}
\end{align}
Based on the above definition,  we prove the following lemma
\begin{lemma}\label{LemmaOptimalApproximation}
Let $u(x,\omega)\in V$ and $u^{*}(x,\omega)=\sum_{i=1,k=1}^{N_x,N_\xi} u_{i,k}\psi_{i,k}(x,\xi(\omega))$
with $\psi_{i,k}(x,\xi(\omega))$ defined by \eqref{OC_MsStocBasis_Obj} and $u_{i,k}=\big(u(x,\omega),\phi_{i,k}(x,\xi(\omega))\big)$.
Then, we have
\begin{align}
\Big(u(x,\omega),u(x,\omega)\Big)_a=\Big(u^*(x,\omega),u^*(x,\omega)\Big)_a+
\Big((u-u^*)(x,\omega),(u-u^*)(x,\omega)\Big)_a.
\label{orthogonal_decomposition}
\end{align}
\end{lemma}
\begin{proof}
According to the definition, it is obvious that $(u-u^*)\in V_0$. Then use the orthogonal condition \eqref{orthogonal}, we know that $ \Big(u^*,u-u^*\Big)_a=0 $. Thus, $ \Big(u,u\Big)_a=\Big((u-u^*)+u^*,(u-u^*)+u^*\Big)_a=\Big(u^*,u^*\Big)_a+\Big((u-u^*),(u-u^*)\Big)_a$.
\end{proof}
Using the above lemma, we obtain the optimal approximation property of our multiscale data-driven stochastic basis functions.
\begin{theorem}\label{TheoremOptimalApproximation}
Let $u(x,\omega)$ be the exact solution to Eq.\eqref{MsStoEllip_OCbasis_Eq} and $u^*(x,\omega)$ be the numerical solution obtained using our method, i.e., $u^*(x,\omega)=\sum_{i=1,k=1}^{N_x,N_\xi} \big(u(x,\omega),\phi_{i,k}(x,\xi(\omega))\big)\psi_{i,k}(x,\xi(\omega))$. Then, we get the optimal approximation property
\begin{align}
\Big((u-u^*),(u-u^*)\Big)_a=\inf_{v\in V_{DSM}}\Big(u-v,u-v\Big). \label{OptimalApproximationProperty}
\end{align}
\end{theorem}
\begin{proof}
The proof of \eqref{OptimalApproximationProperty} is simply based on the result obtained in the Lemma \ref{LemmaOptimalApproximation} and the fact that $u-u^{*}$ is orthogonal to the space $V_{DSM}$
with respect to the inner product $(\cdot,\cdot)_a$.
\end{proof}
\section{Error Analysis}\label{sec:Error Analysis}
\noindent
We shall analyze the error between the numerical solution obtained from our method, denoted by $u^{DSM}(x,\omega)$ and the exact solution $u(x,\omega)$. For notation simplicity, we omit the superscript $\varepsilon$. We choose the mesh size $h$ of the fine grids small enough so that the error between $u(x,\omega)$ and the FEM solution $u_h(x,\omega)$ is  negligible. Thus, we simply choose $u_h(x,\omega)$ as a reference solution. Before proceeding to the main analysis, let us first introduce some notations and assumptions.
\subsection{Some notations and assumptions}
\noindent
Throughout the section, we shall use $D$ to denote the spatial domain and $\Omega$ to denote the random event space.
Then in the tensor space $H_{0}^{1}(D)\otimes L^{2}(\Omega)$, we define the norm by
\begin{align}
||u||_{H_{0}^{1}(D)\otimes L^{2}(\Omega)}=\Big(\int_{\Omega}\big(\int_{D}|\nabla_{x} u(x,\xi(\omega))|^2 dx\big)dP(\omega)\Big)^{\frac{1}{2}},  \label{normH1L2}
\end{align}
where $P(\omega)$ is the probability distribution function of random variables. Moreover, the corresponding energy norm associated with coefficient $a(x,\xi(\omega))$ can be defined by
\begin{align}
||u||_{(a,H_{0}^{1}(D)\otimes L^{2}(\Omega))}=
\Big(\int_{\Omega}\big(\int_{D}  a(x,\xi(\omega))|\nabla_{x} u(x,\xi(\omega))|^2 dx\big) dP(\omega)\Big)^{\frac{1}{2}}, \label{EnergyNormH1L2}
\end{align}
where we use a subscript `$a$' in the norm notation to indicate that the norm is associated with $a(x,\xi(\omega))$. 
In addition, we need higher regularity in the random space when we estimate the convergence rate of our method in the random space. Let $D_{\xi}^{\nu}u(x,\cdot)$ denote the $\nu$-th order mixed derivatives of $u(x,\cdot)$ with respect to the variable $\xi=(\xi_1,...,\xi_r)$ in the random space, where $\nu=(\nu_1,...,\nu_r)$ and $\nu_i$ are nonnegative
integers. Then, we define the norm and the seminorm in the random space as follows,
\begin{align}
||u(x,\cdot)||^2_{H^p(\Omega)} =  \int_{\Omega} \sum_{|\nu|\leq p }|D_\xi^{\nu}u(x,\cdot)|^2 dP(\omega),\quad |u(x,\cdot)|^2_{H^p(\Omega)} =  \int_{\Omega} \sum_{|\nu|= p }|D_{\xi}^{\nu}u(x,\cdot)|^2 dP(\omega).
\label{normHp}
\end{align}
Finally, we assume the stability of the solution with respect to the random dimension  \cite{Charrier:2012,Grahamquasi:2015}.
\begin{assumption}
If $ u(x,\omega) $ is the solution to Eq.\eqref{MsStoEllip_OCbasis_Eq} with boundary condition Eq. \eqref{MsStoEllip_OCbasis_BC}, and $u(x,\cdot)\in H^{p}(\Omega), \forall x\in D $. Then, there exists two positive constants $C_{1}$, $C_{2}$ that depend on the values of $\lambda_{min}$ and $\lambda_{max}$ and the bound of the derivative of $a(x,\xi(\omega))$ with respect to $\xi(\omega)$ so that the following inequalities hold,
\begin{align}
||u(x,\cdot)||_{H^{p}(\Omega)}            &\leq C_{1}||f(x)||_{L^2(D)} , \quad \forall x \in D, \label{Stability1} \\
||\nabla_{x} u(x,\cdot)||_{H^{p}(\Omega)} &\leq C_{2}||f(x)||_{L^2(D)} , \quad \forall x \in D, \label{Stability2}
\end{align}
where $f(x)$ is the source term.
\end{assumption}
\noindent
The above assumption is satisfied if $a(x,\xi(\omega))$ satisfies certain regularity conditions. We refer the interested reader to \cite{cohen2010convergence} for more details.

\subsection{Error analysis}
\noindent
After introducing the necessary notations and assumption, we are in the position to proceed the error analysis. We shall consider errors from the spatial space and the random space approximation. Applying the triangle inequality, we divide the error into three parts
\begin{align}\label{TriangleInequalityMainError}
||u_{h}-u^{DSM}||\leq ||u_{h}-u_{h}^{gPC}||+||u_{h}^{gPC}-u_{H}^{gPC}||+||u_{H}^{gPC}-u^{DSM}||,
\end{align}
where $u_h$ refers to the reference solution obtained using the finite element method (FEM) on fine mesh with size $h$, $u_h^{gPC}$ is the gPC solution obtained using the same fine mesh as $u_h$, $u_H^{gPC}$ is gPC solution obtained using the new multiscale basis functions in the physical space (see \eqref{OC_MsStocBasis_Obj}-\eqref{OC_MsStocBasis_Cons1}), $u^{DSM}$ is the solution obtained using the same multiscale basis as $u_H^{gPC}$ but the dimension of the random basis is further reduced. We have assumed that the error between $u_h$ and the exact solution $u$ to Eq.\eqref{MsStoEllip_OCbasis_Eq} is negligible.

We first analyse the error between $u_h(x,\omega)$ and the solution $u_h^{gPC}(x,\omega)$. To illustrate the main idea, we assume the coefficient in Eq.\eqref{MsStoEllip_OCbasis_Eq} is parameterized by
one-dimensional  random variable $\xi(\omega)=\xi_1(\omega)$ that follows uniform distribution $U[-1,1]$ and the basis functions in gPC method are Legendre polynomials. But we emphasize that the convergence estimate \eqref{error_gPC} holds for general gPC methods if we tensorize the orthogonal polynomials and use the multi-index.
\begin{lemma}\label{lemma2}
Let $u_h(x,\omega)$ be the reference solution and $u_h^{gPC}(x,\omega) $ be the gPC solution using the same mesh. Then we get the convergence estimate as follows,
\begin{align}
||u_{h}-u_{h}^{gPC}||_{(a,H_{0}^{1}(D)\otimes L^{2}(\Omega))}\leq C_{p}N^{-p}||f(x)||_{L^{2}(D)},
\label{error_gPC}
\end{align}
where $N$ is the highest order of polynomial basis in the gPC method, $p$ is an integer that quantifies
the regularity of $u_h(x,\omega)$ in the random space, and $C_{p}$ is a constant that is independent of $N$ but depends on $a_{min}$ and $a_{max}$.
\end{lemma}
\noindent
\begin{proof}
Let $ L_k(\xi(\omega)) $ be the Legendre polynomial of order $k$ and $S_{N}$ be the space spanned by Legendre polynomials of degrees at most $N$, i.e., $ S_{N}=span\{L_{k}(\xi(\omega))\}_{k=0}^{N}$. Let $P_{N}$ denote the projection operator on $S_{N}$. Specifically, we have the projection of $u_{h}(\omega)$ onto $S_{N}$ defined as $P_{N}u_{h}(\omega)=\sum\limits_{k=0}^{N}u_{k}L_{k}(\xi(\omega))$, where the coefficients $u_k=\frac{(u_{h},L_{k})}{(L_{k},L_{k})}$ and the inner product of two functions are defined as $(v,w)\equiv\int_{\Omega}v(\omega)w(\omega)dP(\omega)$. To estimate the decay rate in the projection coefficients,
we use the property that Legendre polynomials satisfy the Sturm-Liouville eigenvalue problem as follows,
\begin{align}
\mathscr{L}L_{k}(\xi(\omega))=\frac{d}{d\xi}\big((1-(\xi(\omega))^{2})\frac{d}{d\xi}
\big)L_{k}(\xi(\omega))=-k(k+1)L_{k}(\xi(\omega)).
\label{Sturm_Liouville_EigProblem}
\end{align}
Some simple calculations imply that
\begin{align}
(u_{h},L_{k}) = -\dfrac{1}{2k(k+1)}\int_{-1}^{1}u_{h}\mathscr{L}L_{k}dP(\omega) = -\dfrac{1}{k(k+1)}(\mathscr{L}u_{h},L_{k}).
\label{Sturm_Liouville_EigProblem}
\end{align}
Then, we repeat the above derivation and get $(u_{h},L_{k}) =(-\frac{1}{k(k+1)})^{l}(\mathscr{L}^{l}u_{h},L_{k})$, where
$l\geq 1$ is an integer. Finally, we obtain the error estimate of the projection approximation as
\begin{align}
||u_{h}(\omega)-P_{N}u_{h}(\omega)||^2_{L^2(\Omega)}   &=\sum\limits_{k=N+1}^{\infty}\frac{(u_{h},L_k)^2_{L^2(\Omega)}}{(L_k,L_k)_{L^2(\Omega)}}
=\sum\limits_{k=N+1}^{\infty}\frac{1}{(k(k+1))^{2l}||L_{k}||^2_{L^2(\Omega)}}(\mathscr{L}^{l}u_{h},L_{k})^2_{L^2(\Omega)}
\nonumber \\
& \leq N^{-4l}||\mathscr{L}^{l}u_{h}||^2_{L^2(\Omega)}  \leq C_pN^{-4l}||u_{h}||^2_{H^{2l}(\Omega)},
\label{ProjectionErrorEstimate}
\end{align}
where the Parseval's identity is used in the first equation. Using the gPC method to solve Eq.\eqref{MsStoEllip_OCbasis_Eq}, we know that $u_{h}^{gPC}(x,\omega)=P_{N}u_{h}(x,\omega)$. If we take the regularity index $p=2l$ and use the stability assumptions \eqref{Stability1}\eqref{Stability2} and integrate over the physical space $D$, we prove the Lemma.
\end{proof}
When the coefficient in Eq.\eqref{MsStoEllip_OCbasis_Eq} is parameterized by $r$ i.i.d. random variables
$\xi=(\xi_1,...,\xi_r)$, we use the multi-index  $ \minda=\bkr{\alpha_1, \cdots, \alpha_{r}},
0\leq \alpha_i \leq N_i, \alpha_i \in \NN $ to label the gPC basis and the multi-index $\nu=(\nu_1,...,\nu_r)$, $\nu_i \ge 0$, $\nu_i \in \NN$ to label the order of the mixed derivatives. We define the tensorized Legendre polynomials by $L_{\minda}(\xi(\omega))=\prod_{i=1}^{r}L_{\alpha_i}(\xi_i(\omega))$
and $S_{N}=span\{L_{\minda}(\xi(\omega))\}$. Let $P_{N}$ denote the projection operator on $S_{N}$, i.e.,
$P_{N}u_{h}(\omega)= u_{\minda}L_{\minda}(\xi(\omega))$, where the Einstein summation convention is used and
$N=\prod_{i=1}^{r}N_i$.
\begin{corollary}\label{corollary3}
Let $u_h(x,\omega)$ be the reference solution and $u_h^{gPC}(x,\omega)=P_{N}u_{h}(\omega)$ be the gPC solution using the same mesh. Then we get the convergence estimate as follows, 
\begin{align}
||u_h(\omega)-P_{N}u_h(\omega)||^2_{L^2(\Omega)}  \leq {(N_1^{-4\nu_1}N_2^{-4\nu_2}...N_r^{-4\nu_r})}||\mathscr{L}^{\nu}u_h||^2_{L^2(\Omega)}. \label{error_gPC_tensor}
\end{align}
If the highest order of the polynomials are the same in each random variable and let $|\nu|=\sum_{i=1}^{r}\nu_i$, 
we get
\begin{align}
||u_h(\omega)-P_{N}u_h(\omega)||^2_{L^2(\Omega)}  \leq N_{1}^{-4|\nu|}||\mathscr{L}^{\nu}u||^2_{L^2(\Omega)} \leq C_pN_{1}^{-4|\nu|}||u||^2_{H^{2|\nu|}(\Omega)}.  \label{error_gPC_tensor2}
\end{align}
\end{corollary}
\begin{remark}
One may choose the best $N$-term Galerkin approximations in the gPC method, which reduces the total number of
basis functions and maintain an optimal convergence rate \cite{cohen2010convergence}. In this paper, we do not want to
complicate the presentation by pursuing this direction.
\end{remark}
\begin{remark}
The classical orthogonal polynomials arise from a differential equation of the form
\begin{align}
Q(x)u^{\prime\prime}(x)+L(x)u^{\prime}(x)=\lambda u,\label{Classical_Orthogonal_Polynomials}
\end{align}
where $Q(x)$ and $L(x)$ are polynomials and the function $u(x)$ and the constant $\lambda$ are obtained by solving this Sturm-Liouville type eigenvalue problem. The solutions of \eqref{Classical_Orthogonal_Polynomials} have singularities unless $\lambda$ takes on specific values. Let $\{\lambda_{k},u_{k}(x)\}$, $k=0,1,...$ denote the corresponding eigenvalues and eigenfunctions. Then, $\{u_{k}(x)\}$ form a set of orthogonal polynomials and the eigenvalues satisfy $\lambda_{k}=\frac{k(k-1)}{2}Q^{\prime\prime}(x)+kL^{\prime}(x)$ \cite{Askey:1985,CanutoQuarteroni:1982}. Therefore, the error estimates proved in the Lemma \ref{lemma2} and the Corollary \ref{corollary3} hold for other type of orthogonal polynomials.
\end{remark}


We then analyse the error of the gPC solutions obtained on the fine mesh and the coarse mesh.
\begin{lemma}\label{lemma3}
Let $u_h^{gPC}(x,\omega)$ and $u_{H}^{gPC}(x,\omega)$ denote the gPC solutions obtained on the fine mesh with FEM basis functions and the coarse mesh with our multiscale basis functions, respectively. We get the following error estimate,
\begin{align}
||u_{h}^{gPC}-u_{H}^{gPC}||_{(a,H_{0}^{1}(D)\otimes L^{2}(\Omega))}\leq C_2 H||f(x)||_{L^{2}(D)}
\end{align}
where $C_2$ depends on $a_{min}$ and $a_{max}$ and $H$ is the mesh size of coarse elements.
\end{lemma}
\begin{proof}
We only need to analyse the error in the physical space as the random space is discretized using the same method.
Let $V_h$ denote the space spanned by the FE nodal basis functions on the fine mesh.
For each realization $\omega$, we know that $u^{gPC}_{h}(x,\omega)\in V_{h}\cap H_{0}^{1}(D)$ and
$u^{gPC}_{H}(x,\omega)\in V_{h}\cap H_{0}^{1}(D) $. Let $e(x,\omega)=u^{gPC}_{h}(x,\omega)-u^{gPC}_{H}(x,\omega)$ denote the error. We have $ e(x,\omega)\in V_{h}\cap H_{0}^{1}(D) $.
Using the Aubin-Nistche duality argument to the weak formulation of Eq.\eqref{MsStoEllip_OCbasis_Eq}, we get
\begin{align}
a\Big(u_{h}^{gPC}(x,\omega),e(x,\omega)\Big)=\Big(f(x),e(x,\omega)\Big), \label{Aubin_Nistche1}
\end{align}
Using the Galerkin orthogonality property $a\Big(u_{H}^{gPC}(x,\omega),e(x,\omega)\Big)=0$, we get
\begin{align}
a\Big(e(x,\omega),e(x,\omega)\Big)=\Big(f(x),e(x,\omega)\Big). \label{Aubin_Nistche2}
\end{align}
We introduce a coarse mesh interpolation operator $ \mathscr{J}_{H} $, which satisfies
$\mathscr{J}_{H}e(x,\omega)=0$. From Eq.\eqref{Aubin_Nistche2}, we get the estimate as follows,
\begin{align}
||e(x,\omega)||_{a}^2& \leq ||f(x)||_{L^2(D)}||e(x,\omega))||_{L^2(D)}=||f(x)||_{L^2(D)}||e(x,\omega))- \mathscr{J}_{H}e(x,\omega)||_{L^2(D)} \nonumber \\
&\lesssim H||f(x)||_{L^2(D)}||e(x,\omega)||_{H^1(D)}, \label{Aubin_Nistche3}
\end{align}
where we have used the holder inequality, the interpolation inequality, and the definition of the energy norm $||\cdot||_{a}$.
Note that the relationship between $||\cdot||_{a}$ and $H^1$ norm implies that
\begin{align}
||e(x,\omega)||_{a} \lesssim C_2 H||f(x)||_{L^2(D)},
\label{Aubin_Nistche4}
\end{align}
where $C_2$ depends on $a_{min}$ and $a_{max}$. Since the estimate \eqref{Aubin_Nistche4} holds for each realization $\omega$, we integrate over the stochastic space and prove the lemma.
\end{proof}
Finally, we analyse the error between the solution $u_{H}^{gPC}$ and the solution $u^{DSM}_{H}$ obtained using our method. Note that both $u_{H}^{gPC}$ and $u^{DSM}_{H}$ are represented in the same physical space. The $u_{H}^{gPC}$ is obtained using the gPC method with $N_{gPC}$ basis functions (see e.g., Eq.\eqref{gPC_index}), while the number of the stochastic basis functions for $u^{DSM}_{H}$ is $N_{\xi}$, which is further reduced.
\begin{lemma}\label{lemma4}
Let $u_H^{gPC}(x,\omega)$ and $u^{DSM}(x,\omega)$ denote the solutions that have been defined above. Then we get the following error estimate
\begin{align}
||u_{H}^{gPC}(x,\omega)-u^{DSM}(x,\omega)||_{(a,H_{0}^{1}(D)\otimes L^{2}(\Omega))}\leq 
C_3 \sum\limits_{i=N_{\xi}+1}^{N_{gPC}}\mu_{i}, \label{POD_RandomSpace}
\end{align}
where $C_3$ is a constant depending on the diameter of domain $D$, $a_{min}$ and $a_{max}$, and ${\mu_{i}}$ are eigenvalues of the covariance matrix of the solution $u_H^{gPC}(x,\omega)$.
\end{lemma}
\begin{proof}
Let $V_{DSM}=span\{\psi_{i,k}(x,\xi(\omega))\}$, $i=1,\cdots, N_x$ and $k=1,\cdots,N_{\xi}$. We know that $u^{DSM} \in V_{DSM}\subseteq H_{0}^{1}(D)\otimes L^{2}(\Omega)$. Recall that the bilinear form $a(u,v)$ are defined as
\begin{align}
a(u,v)=\int_{\Omega}\big(\int_{D}\nabla_{x} u(x,\omega)\cdot a(x,\omega)\nabla_{x} v(x,\omega) dx\big)dP(\omega)
\label{Weakform-ErrorAnalysis}
\end{align}
and the right-hand-side term is defined as
\begin{align}
(f,v)=\int_{\Omega}\int_{D} f(x) v(x,\omega) dxdP(\omega).
\label{RHS-ErrorAnalysis}
\end{align}
Note that $u_{H}^{gPC}(x,\omega)$ and $u^{DSM}(x,\omega)$ are represented using the same physical basis functions and
the gradient in the bilinear form of $a(u,v)$ is only respect to the physical space. We have
\begin{align}
a(u_{H}^{gPC}(x,\omega),v(x,\omega))& =(f(x),v(x,\omega)), \quad \forall  v(x,\omega) \in V_{DSM},  \\
a(u^{DSM}(x,\omega),v(x,\omega)& =(f(x),v(x,\omega)), \quad      \forall  v(x,\omega) \in V_{DSM}.
\end{align}
Then, we get the following orthogonality condition 
\begin{align}
a(u_{H}^{gPC}(x,\omega)-u^{DSM}(x,\omega),v(x,\omega)) =0, \quad \forall  v(x,\omega) \in V_{DSM}.
\end{align}
Applying the Aubin-Nistche duality argument again and using the fact $u^{DSM}(x,\omega)\in V_{DMS}$, we get
the following error estimate,
\begin{align}
 ||u_{H}^{gPC}-u^{DSM}||_{(a,H_{0}^{1}(D)\otimes L^{2}(\Omega))}
 &\leq \inf \limits_{\forall v\in V_{DSM}}||u_{H}^{gPC}-v||_{(a,H_{0}^{1}(D)\otimes L^{2}(\Omega))}, \nonumber \\
 &\leq C_3\inf \limits_{\forall v\in V_{DSM}}||u_{H}^{gPC}-v||_{H_{0}^{1}(D)\otimes L^{2}(\Omega)},
\label{inf_error}
\end{align}
where $C_3$ depends on $a_{min}$ and $a_{max}$. Since $u_{H}^{gPC}(x,\omega)$ and $u^{DSM}(x,\omega)$ are represented using the same physical basis functions, we only analyse the error generated from the random space. For each node $ x_{i}$ on the coarse mesh, we have
\begin{align}
u_{H}^{gPC}(x_{i},\omega)=\sum\limits_{l=1}^{N_{gPC}}u_{H,l}^{gPC}(x_{i})\hmpn_{l}(\xi(\omega),
\quad\text{and}\quad
u^{DSM}(x_{i},\omega)=\sum\limits_{k=1}^{N_{\xi}}u_{k}^{DSM}(x_{i})A_k(\xi(\omega)),
\end{align}
where each $A_{k}(\xi(\omega)) $ is a linear combination of gPC basis $\hmpn_{l}(\xi(\omega))$, i.e.,
$A_{k}(\xi(\omega))=\sum\limits_{l=1}^{N_{gPC}}A_{kl}\hmpn_{l}(\xi(\omega))$.
Applying the Gram-Schmidt algorithm to $\{A_{k}(\xi(\omega))\}_{k=1}^{N_{\xi}}$, we can construct a set of orthogonal basis functions, still denoted by  $\{A_{k}(\xi(\omega))\}_{k=1}^{N_{\xi}}$.  Let $V_{A}=span\{A_k(\xi(\omega))\}_{k=1}^{N_{\xi}}$. We have
\begin{align}
\inf \limits_{\forall v\in V_{A}}||u_{H}^{gPC}(x_i,\omega)-v(x_i,\omega)||_{L^{2}(\Omega)}
=||u_{H}^{gPC}(x_i,\omega)-\Pi_{V_{A}}(u_{H}^{gPC}(x_i,\omega))||_{L^{2}(\Omega)},
\label{inf_A2}
\end{align}
where $\Pi_{V_{A}}$ is the projection operator onto the subspace $V_{A}$. Since the number of coarse nodes is $N_x$, we
collect all the gPC coefficients together and define
\begin{equation}\label{gPC_coefficient_Y}
Y=\left(
\begin{matrix}
u_{H,1}^{gPC}(x_{1})&\cdots&u_{H,1}^{gPC}(x_{N_{x}})&\\
\vdots&\ddots&\vdots\\
u_{H,N_{gPC}}^{gPC}(x_{1})&\cdots&u_{H,N_{gPC}}^{gPC}(x_{N_{x}})&
\end{matrix}
\right).
\end{equation}
Then, we analyse the error between $u_{H}^{gPC}(x,\omega)$ and $u^{DSM}(x,\omega)$ through the singular value decomposition (SVD) of $Y$. Specifically,  we compute the eigenvalues and eigenvectors of $ YY^{T}v_{i}=\mu_{i}v_{i}$, $i=1,\cdots,N_{gPC}$, where $v_{i}$ are eigenvectors and $ \mu_{1}\geq \cdots \mu_{N_{gPC}}\geq 0 $ are corresponding eigenvalues. If we choose the first $N_{\xi}$ eigenvectors and take $V_{N_{\xi}}=span\{v_{i}\}_{i=1}^{N_{\xi}}$, we have
\begin{align}
\min\limits_{V_{N_{\xi}}}||Y-\Pi_{V_{N_{\xi}}}(Y)|| \leq \sum\limits_{i=N_{\xi}+1}^{N_{gPC}}\mu_{i},
\label{min_pod1}
\end{align}
where $\Pi_{V_{N_{\xi}}}$ is the projection operator onto the subspace $V_{N_{\xi}}$. Recall that the optimal approximation property of our basis functions that is proved in Section \ref{sec:VariationalProperties} and both $V_{A}$ and $V_{N_{\xi}}$ are $N_{\xi}$-dimensional space. We know that for each coarse grid node
$x_i$, $i=1,...N_x$, we have
 \begin{align}
\min\limits_{V_{A}}||u_{H}^{gPC}(x_i,\omega)-\Pi_{V_{A}}(u_{H}^{gPC}(x_i,\omega))||_{L^{2}(\Omega)}
\leq \min\limits_{V_{N_{\xi}}}||Y-\Pi_{V_{N_{\xi}}}(Y)|| \leq \sum\limits_{i=N_{\xi}+1}^{N_{gPC}}\mu_{i}. \label{min_pod2}
 \end{align}
Therefore by the definition of the norm $ ||\cdot||_{(a,H_{0}^{1}(D)\otimes L^{2}(\Omega))} $ and
the fact that the error is independent of the physical space approximation,
we complete our proof by using  \eqref{inf_A2} \eqref{min_pod1} and \eqref{min_pod2}
 \begin{align}
 ||u_{H}^{gPC}-u^{DSM}||_{(a,H_{0}^{1}(D)\otimes L^{2}(\Omega))}\leq
 C_3\inf \limits_{\forall v\in V_{DSM}}||u_{H}^{gPC}-v||_{H_{0}^{1}(D)\otimes  L^{2}(\Omega)}  \leq   C_3\sum\limits_{i=N_{\xi}+1}^{N_{gPC}}\mu_{i}, \label{gPC_DSM_RandomError}
 \end{align}
where $C_3$ depends on $a_{min}$, $a_{max}$, and the domain size.
\end{proof}
\begin{remark}
How to determine the random dimension of the solution (i.e., $N_{\xi}$) \emph{a priori} is very challenging since it depends on the regularity of the coefficient $a^{\varepsilon}(x,\omega)$ and the force $f(x)$. We propose an \emph{a posteriori} approach to estimate $N_{\xi}$ in this work. See \ref{sec:StochasticDimension} for more  details. The same issue was also discussed in \cite{hou2015heterogeneous}.
\end{remark}
\begin{theorem}
If $u_{h}(x,\omega) $ is the reference solution to Eqns. \eqref{MsStoEllip_OCbasis_Eq} and \eqref{MsStoEllip_OCbasis_BC} and $u^{DSM}(x,\omega)$ is the solution obtained by our method, we have the error estimate as follows,
\begin{align}
||u_{h}-u^{DSM}||_{(a,H_{0}^{1}(D)\otimes L^{2}(\Omega))}\leq 
\prod_{i=1}^{r}N_{i}^{-2\nu_i}||\mathscr{L}^{\nu}u_h||_{L^2(\Omega)}
+C_2H||f(x)||_{L^{2}(D)\otimes L^2(\Omega)}+ C_3\sum\limits_{i=N_{\xi}+1}^{N_{gPC}}\mu_{i}
\label{theorem1}
\end{align}
with $C_2$ and $C_3$ depend on $a_{min}$ and $a_{max}$ and ${\mu_{i}} $ are eigenvalues of the covariance kernel of the solution $u_{h}(x,\omega)$.
\end{theorem}
\begin{proof}
For simplicity, we use $ ||\cdot|| $ to denote the norm $ ||\cdot||_{a,H_{0}^{1}(D)\otimes L^{2}(\Omega))} $, then by the triangle inequality we can get
\begin{align}
||u_{h}-u^{DSM}||\leq  ||u_{h}-u_{h}^{gPC}||+||u_{h}^{gPC}-u_{H}^{gPC}||+||u_{H}^{gPC}-u^{DSM}||.
\end{align}
By the three lemmas that have been proved before, we can get the error estimate directly.
\end{proof}

\section{Numerical examples}\label{sec:NumericalExamples}
\noindent
In this section, we shall perform numerical experiments to test the performance and accuracy of the proposed method. We also verify our error analysis through numerical experiments.
\subsection{An example with an oscillatory coefficient}\label{sec:NumericalExample1}
\noindent
We consider the following multiscale elliptic PDE with random coefficient on $D=[0,1]\times[0,1]$:
\begin{align}
-\nabla\cdot(a^{\varepsilon}(x,y,\omega)\nabla u^{\varepsilon}(x,y,\omega))  & = f(x,y), \quad (x,y)\in D,  \omega\in\Omega, 
\label{MsDSM_NumEX1_Eq} \\
u^{\varepsilon}(x,y,\omega)&= 0, \quad \quad \quad \quad (x,y)\in \partial D.\label{MsDSM_NumEX1_BC}
\end{align}
The multiscale and random information is described by the coefficient
\begin{small}
\begin{align}
a^{\varepsilon}(x,y,\omega)=0.1
+\frac{2+p_1\sin(\frac{2\pi(x-y)}{\varepsilon_1})}{2-p_1\cos(\frac{2\pi(x-y)}{\varepsilon_1})}\xi_1(\omega)
+\frac{2+p_2\cos(\frac{2\pi y}{\varepsilon_2})}{2-p_2\sin(\frac{2\pi x}{\varepsilon_2})}\xi_2(\omega)
+\frac{2+p_3\sin(\frac{2\pi(x-0.5)}{\varepsilon_3})}{2-p_3\cos(\frac{2\pi(y-0.5)}{\varepsilon_3})}\xi_3(\omega),  \label{MsDSM_NumEX1_Coef}
\end{align}
\end{small}
where $\varepsilon_1=1/11$, $\varepsilon_2=1/14$, $\varepsilon_3=1/17$, $p_1=1.6$, $p_2=1.4$, $p_3=1.5$, and $\{\xi_{i}(\omega)\}_{i=1}^{3}$ are independent uniform random variables in $[0,1]$. We choose the parameters of
the coefficient \eqref{MsDSM_NumEX1_Coef} in such a way that it will generate oscillatory features in the solution.

\emph{Multiquery results in the online stage.}  We shall show that our multiscale data-driven stochastic basis functions  can be used to solve a multi-query problem. In our computations, we use the standard FEM to discretize the spatial dimension.  We choose a $256\times256$ fine mesh to well resolve the spatial dimension of the stochastic solution $u^{\varepsilon}(x,y,\omega)$. Since the solution $u^{\varepsilon}(x,y,\omega)$ is smooth in stochastic space, we use a sparse-grid based stochastic collocation method to discretize the stochastic space. First, we conduct a convergence study, and find that the relative errors of the mean and the standard deviation (STD) between the solutions obtained by
a seven-level sparse grid in the SFEM and higher-level sparse grids are smaller than $0.1\%$ both in the $L^2$ and the $H^1$ norms. Therefore, we choose seven-level sparse grid with 207 points in the SFEM to compute the reference solution.

To implement our method, we take the physical coarse mesh grid to be 8$ \times $8 and choose the polynomial chaos basis functions of total order 4 to approximate the  stochastic space. We remark that in our DSM method, the forcing function $f(x,y)$ should be well-resolved by the coarse mesh, otherwise the numerical error will be large. We choose $\mathfrak{F}=\{ \sin( k_{i}\pi x + \phi_{i})\cos(l_{i} \pi y + \varphi_{i})\}_{i=1}^{20}$, where $k_{i}$ and $l_{i}$ are uniformly distributed over the interval $[0,4]$, while $\phi_{i}$ and $\varphi_{i}$ are uniformly distributed over the interval $[0,1]$, as the function class of right-hand side in the preconditioning of the our method. In Fig.\ref{fig:MultiQuery2DMeanError_EX1} and Fig.\ref{fig:MultiQuery2DSTDError_EX1}, we show the relative errors of the mean and the STD of the our method in the $L^2$ norm and the $H^1$ norm, respectively. One can see that our method is efficient in solving the multi-query problem.
\begin{figure}[tbph]
  \centering
  \begin{subfigure}{0.39\textwidth}
    \centering
    \includegraphics[width=\textwidth]{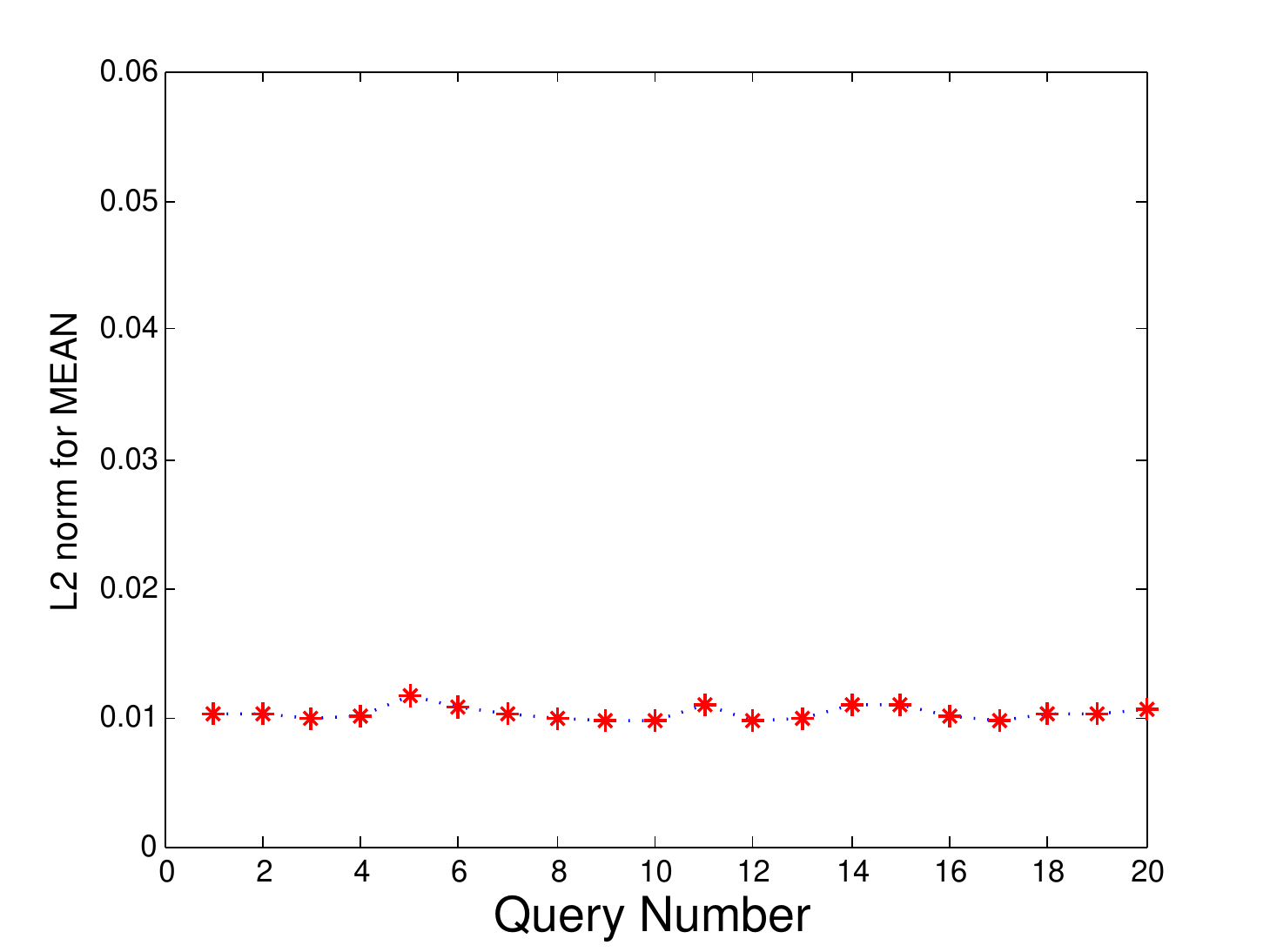}
    \label{fig:ex3_mean}
  \end{subfigure}
  \begin{subfigure}{0.39\textwidth}
    \centering
    \includegraphics[width=\textwidth]{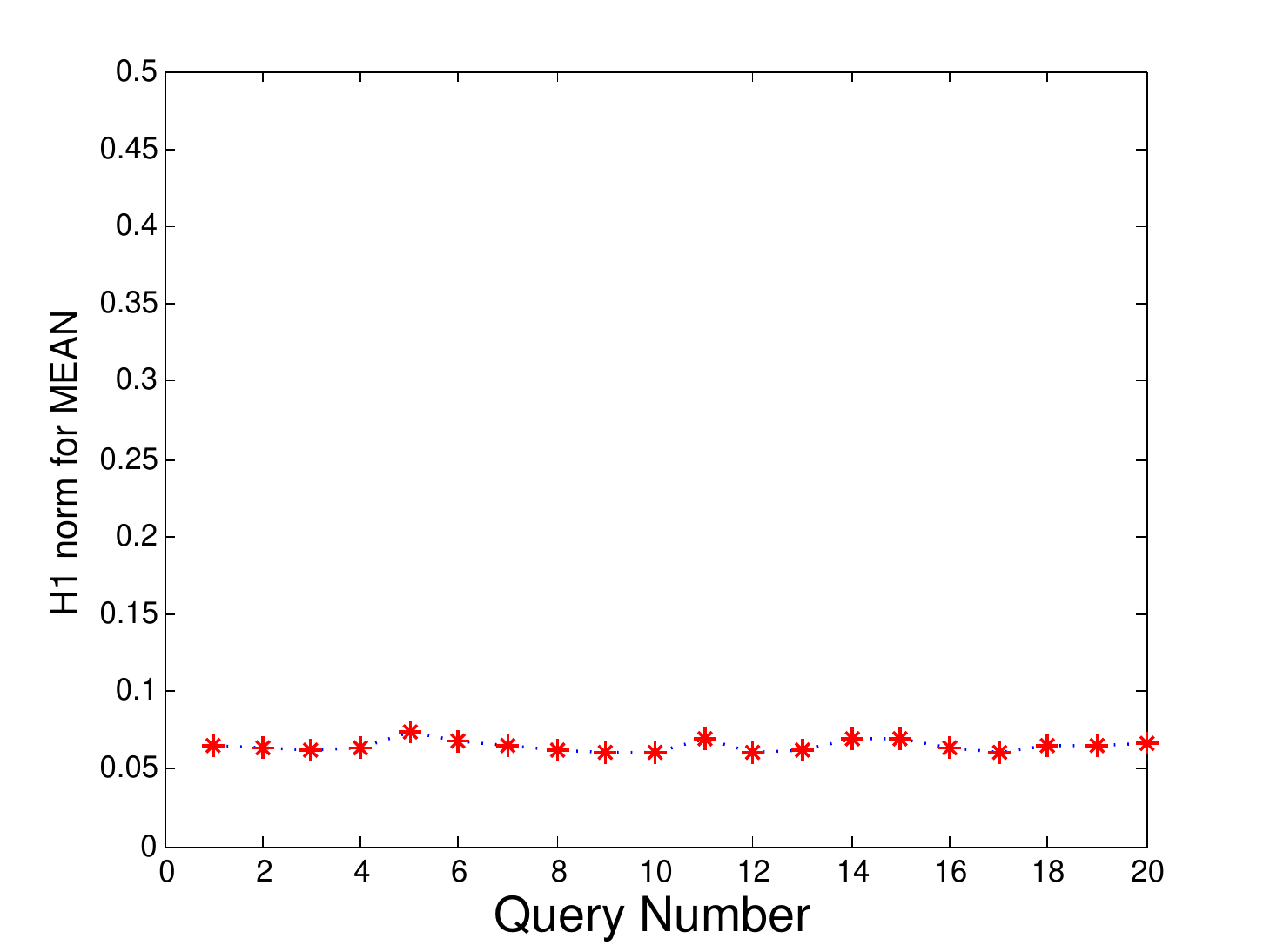}
    \label{fig:ex3_std}
  \end{subfigure}
  \caption{The mean error of our method in the $L^2$ norm and the $H^1$ norm.}
  \label{fig:MultiQuery2DMeanError_EX1}
\end{figure}
\begin{figure}[tbph]
  \centering
  \begin{subfigure}{0.39\textwidth}
    \centering
    \includegraphics[width=\textwidth]{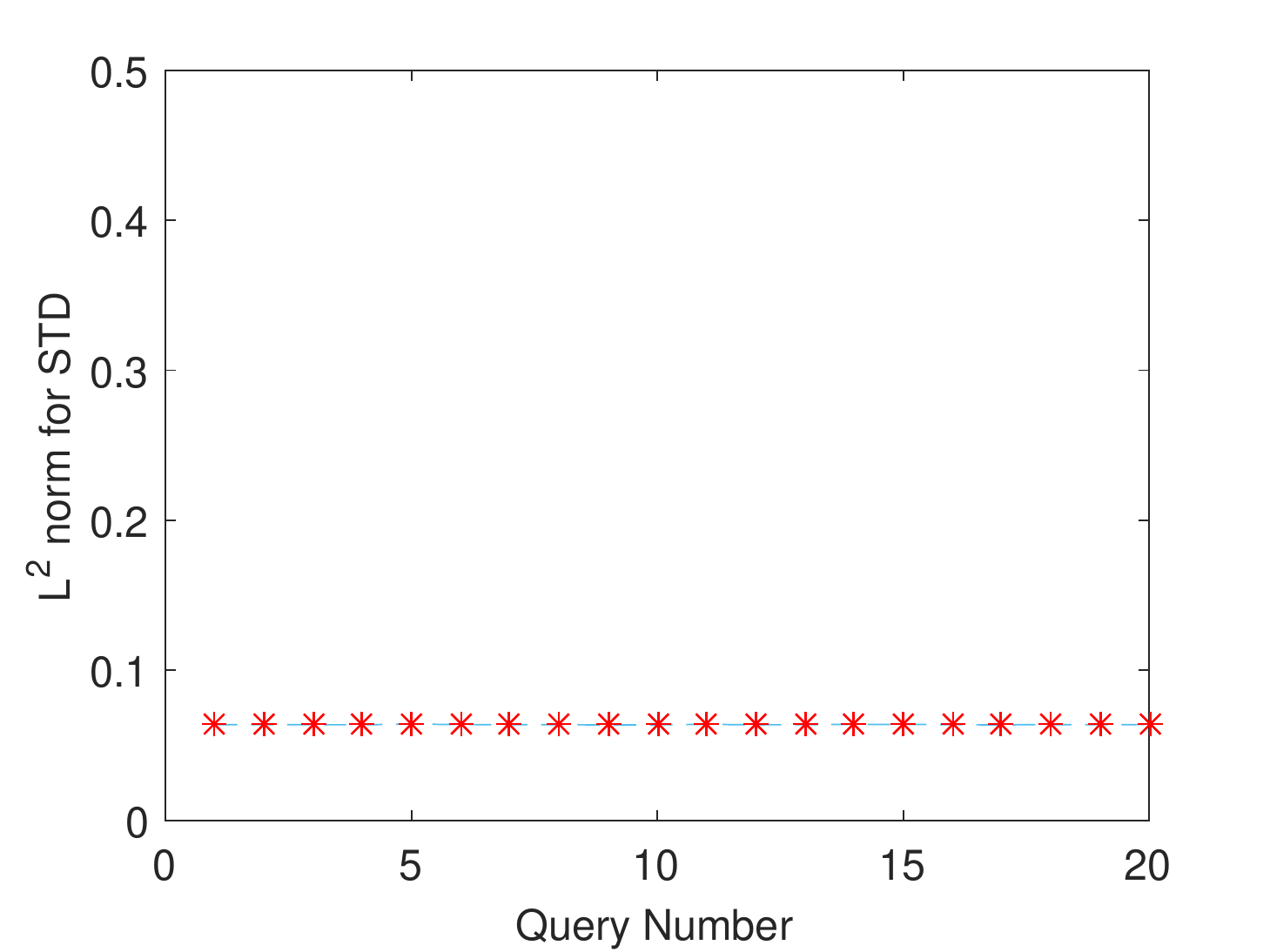}
    \label{fig:ex3_mean}
  \end{subfigure}
  \begin{subfigure}{0.39\textwidth}
    \centering
    \includegraphics[width=\textwidth]{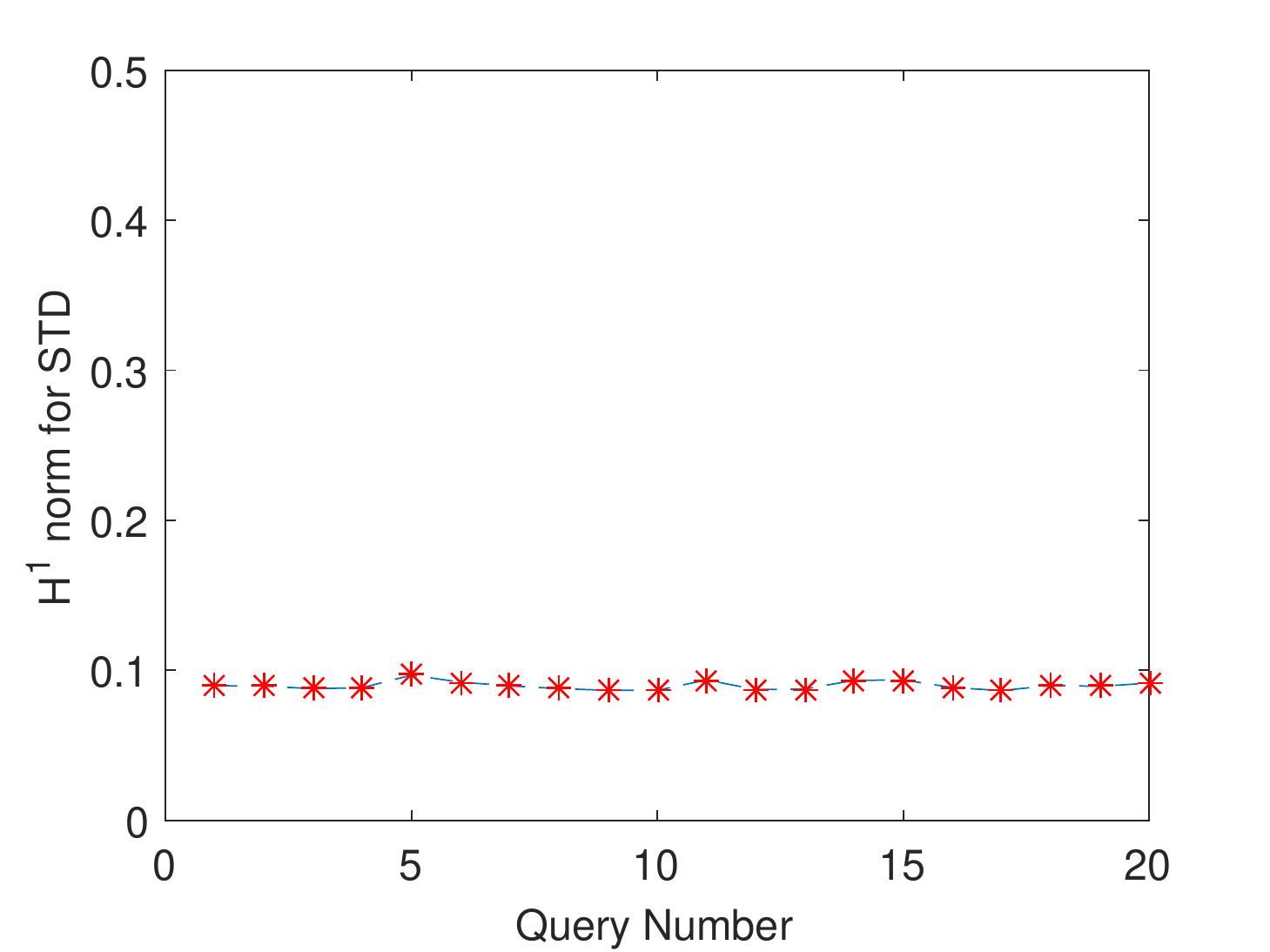}
    \label{fig:ex3_std}
  \end{subfigure}
  \caption{The STD error of our method in the $L^2$ norm and the $H^1$ norm.}
  \label{fig:MultiQuery2DSTDError_EX1}
\end{figure}

In Fig.\ref{fig:ProfilesOfOneSolution_EX1}, we show the mean and STD of the solution corresponding to $f(x,y)=\sin(2.3\pi x + 0.2)\cos(1.5\pi y - 0.3)$. We use the notation SC to denote the reference solution, which is obtained using the stochastic finite element method on the fine mesh, while the notation DSM to denote the solution obtained using our method. One can see that the mean and the STD of the DSM solution match the mean and the STD of the exact solution very well, and the multiscale solution has heterogeneous structures.
\begin{figure}[tbph]
  \centering
  \includegraphics[width=0.7\textwidth]{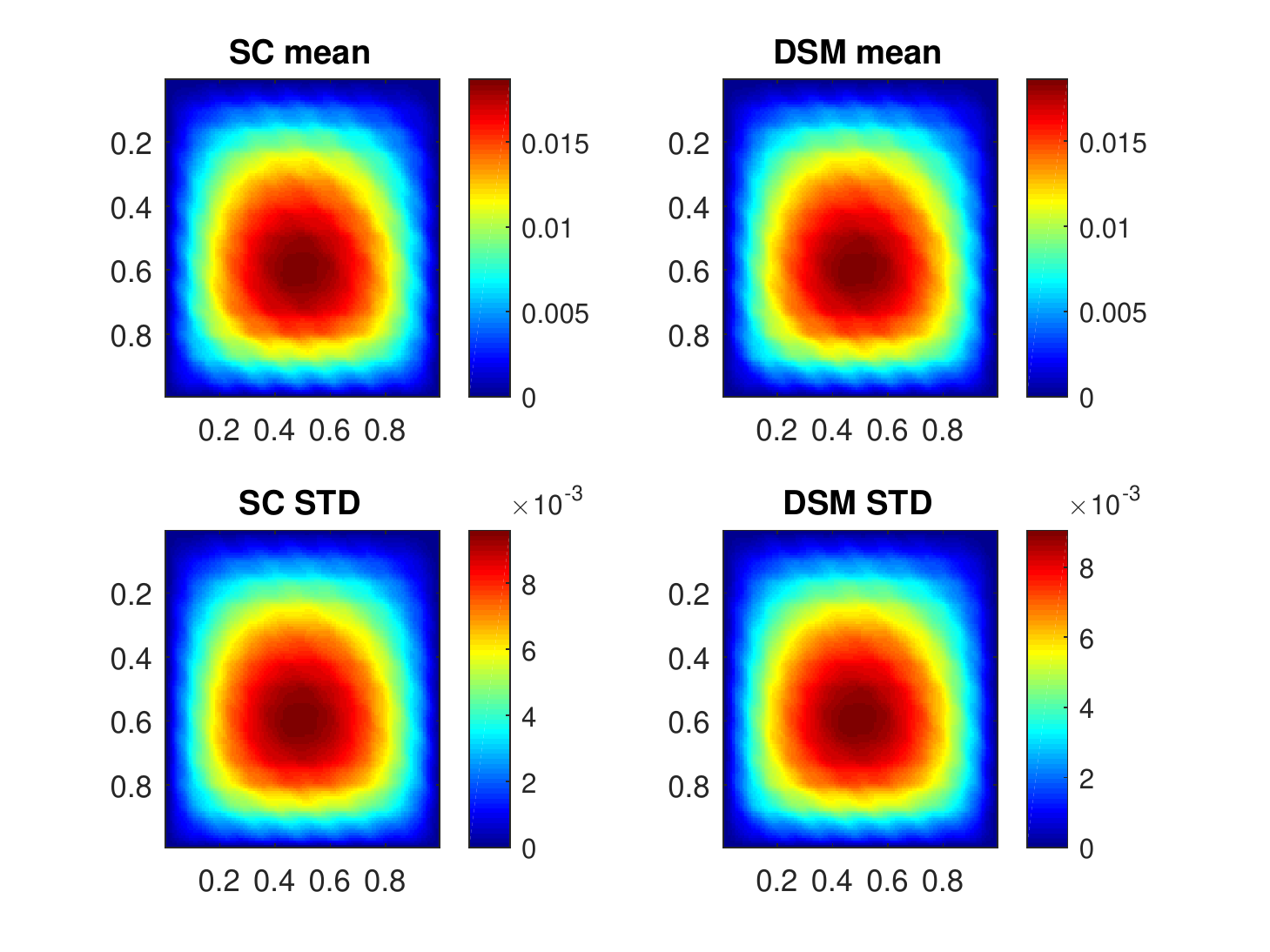}
  \caption{Profiles of the mean and the STD solution.}
  \label{fig:ProfilesOfOneSolution_EX1}
\end{figure}

In Fig.\ref{fig:NumEX1_ProfileOfBasisFunctions}, we plot the profiles of the data-driven stochastic basis functions  obtained from our method, which are associated with the grid $(0.5,0.5)$. One can find the exponential decay of the basis functions. Therefore, we can localize the computational domain of the basis functions and reduce the computational cost.
\begin{figure}[tbph]
  \centering
  \begin{subfigure}{0.39\textwidth}
    \centering
    \includegraphics[width=\textwidth]{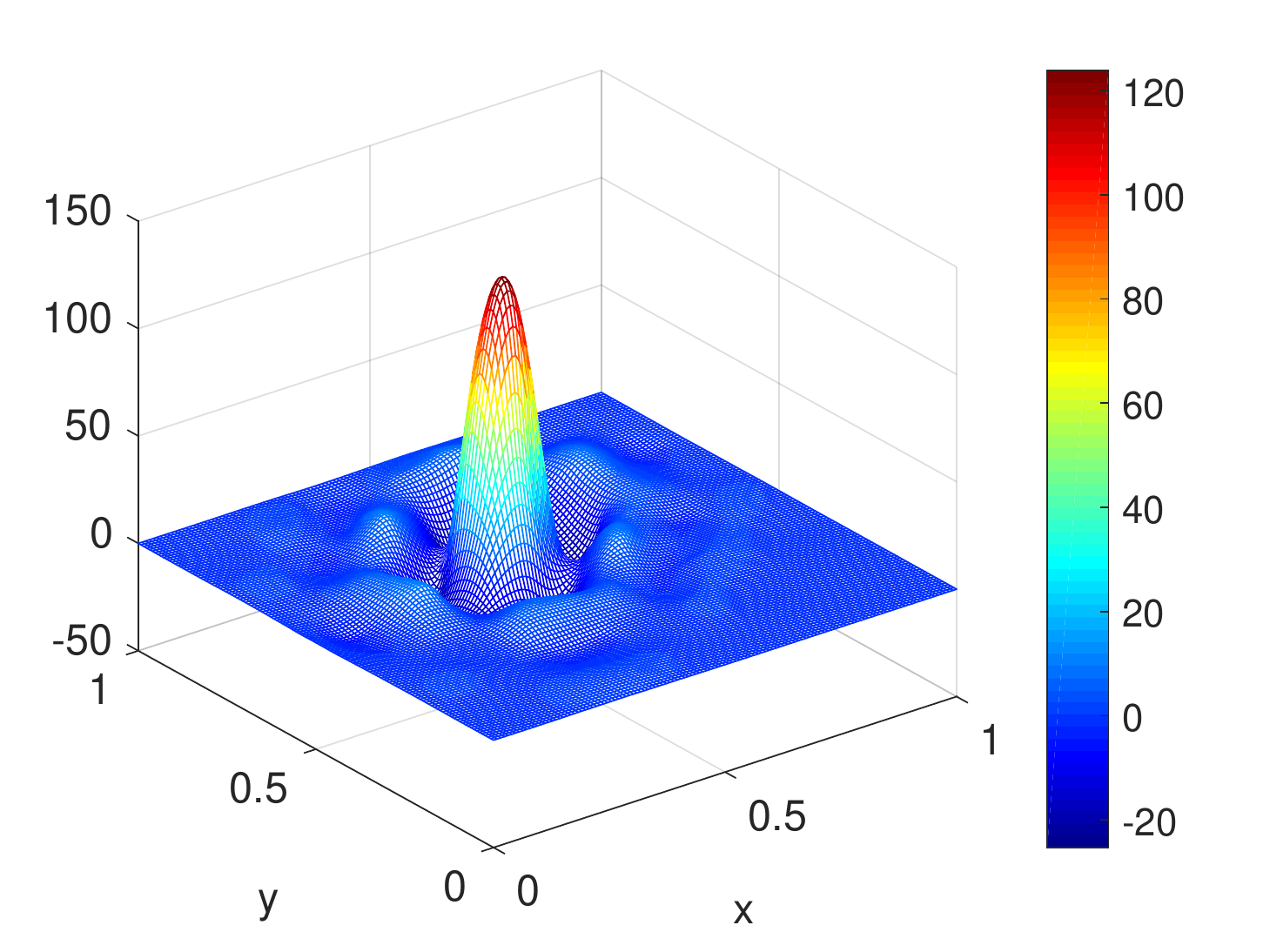}
    \label{fig:ex3_mean}
  \end{subfigure}
  \begin{subfigure}{0.39\textwidth}
    \centering
    \includegraphics[width=\textwidth]{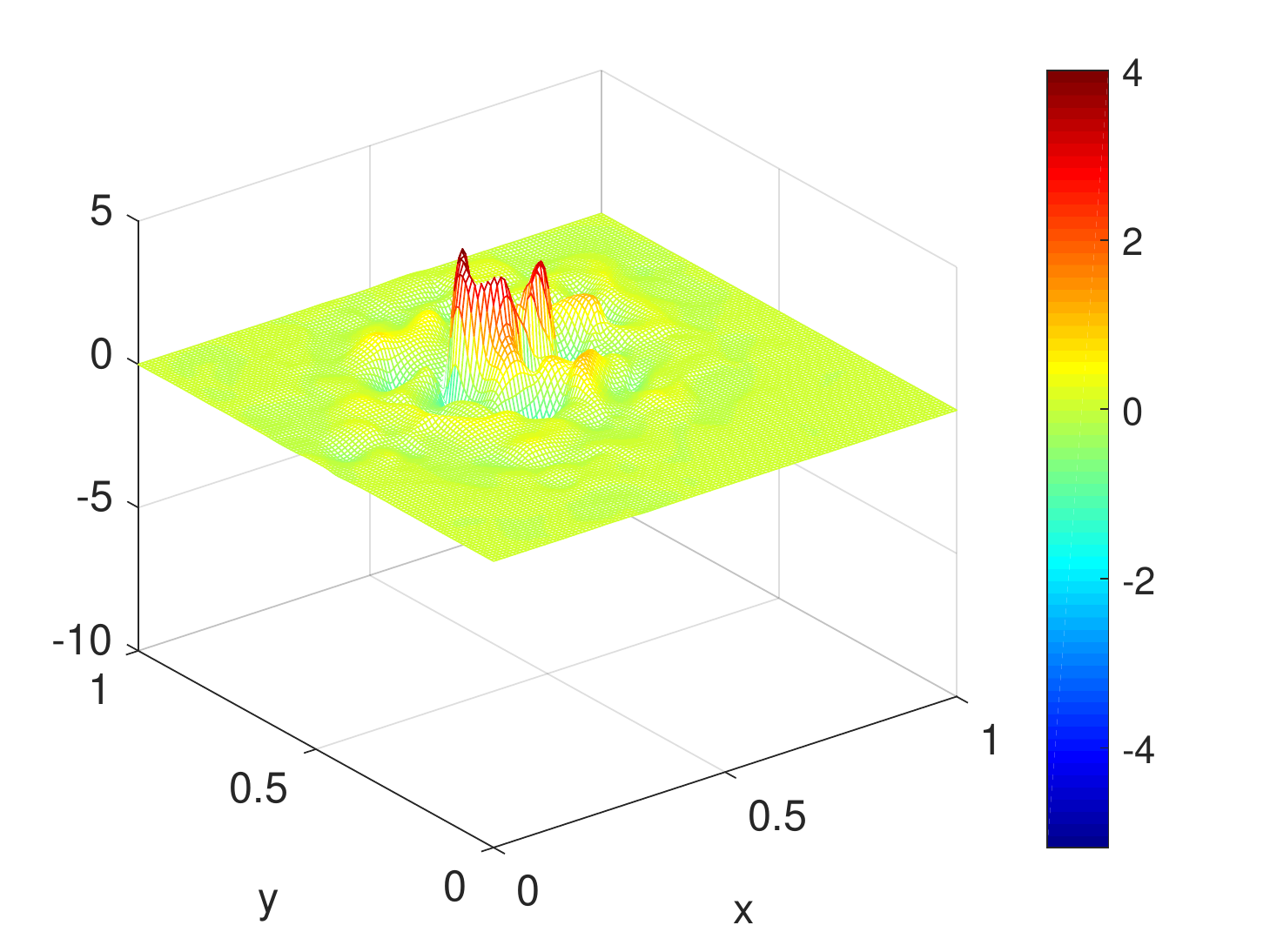}
    \label{fig:ex3_std}
  \end{subfigure}
  \caption{Profiles of the data-driven stochastic basis functions. Left is the projection on the zeroth order gPC basis. Right is the projection on one of the first order gPC basis.}
  \label{fig:NumEX1_ProfileOfBasisFunctions}
\end{figure}

\emph{Verification of the convergence rate with respect to meshsize H.}
We shall test two different coefficients. In the first case, the coefficient is parameterized by one random variable,
\begin{align}
a^{\varepsilon}(x,y,\omega)=0.1+\frac{2+p_1\sin(\frac{2\pi(x-y)}{\varepsilon_1})}
{2-p_1\cos(\frac{2\pi(x-y)}{\varepsilon_1})}\xi_1(\omega),
\label{MsDSM_NumEX1_ConvergeForH}
\end{align}
where $p_1=1.6$ and $ \varepsilon_1=\frac{1}{14}$.  The highest order of the gPC basis functions is 7. We change the coarse mesh grid from $4\times 4$ to $64\times 64$. We compare the results on different meshes and calculate the numerical error with respect to the reference solution obtained by using the fine mesh $\frac{1}{256}$. In the second case, the coefficient is given by Eq.\eqref{MsDSM_NumEX1_Coef}, which is parameterized by three random variables. The highest order of the gPC basis functions is 4. We choose the coarse mesh grids as $ 5\times 5$ , $10 \times 10$, $12 \times 12$, and $15 \times 15$. In Fig.\ref{fig:NumEX2_ConvergeForH},  we plot the convergence results with respect to meshsize $H$. For both experiments, we obtain a first order convergence for the error in the $H^1$ norm, which agrees with our error analysis.
\begin{figure}[tbph]
  \centering
  \begin{subfigure}{0.39\textwidth}
    \centering
    \includegraphics[width=\textwidth]{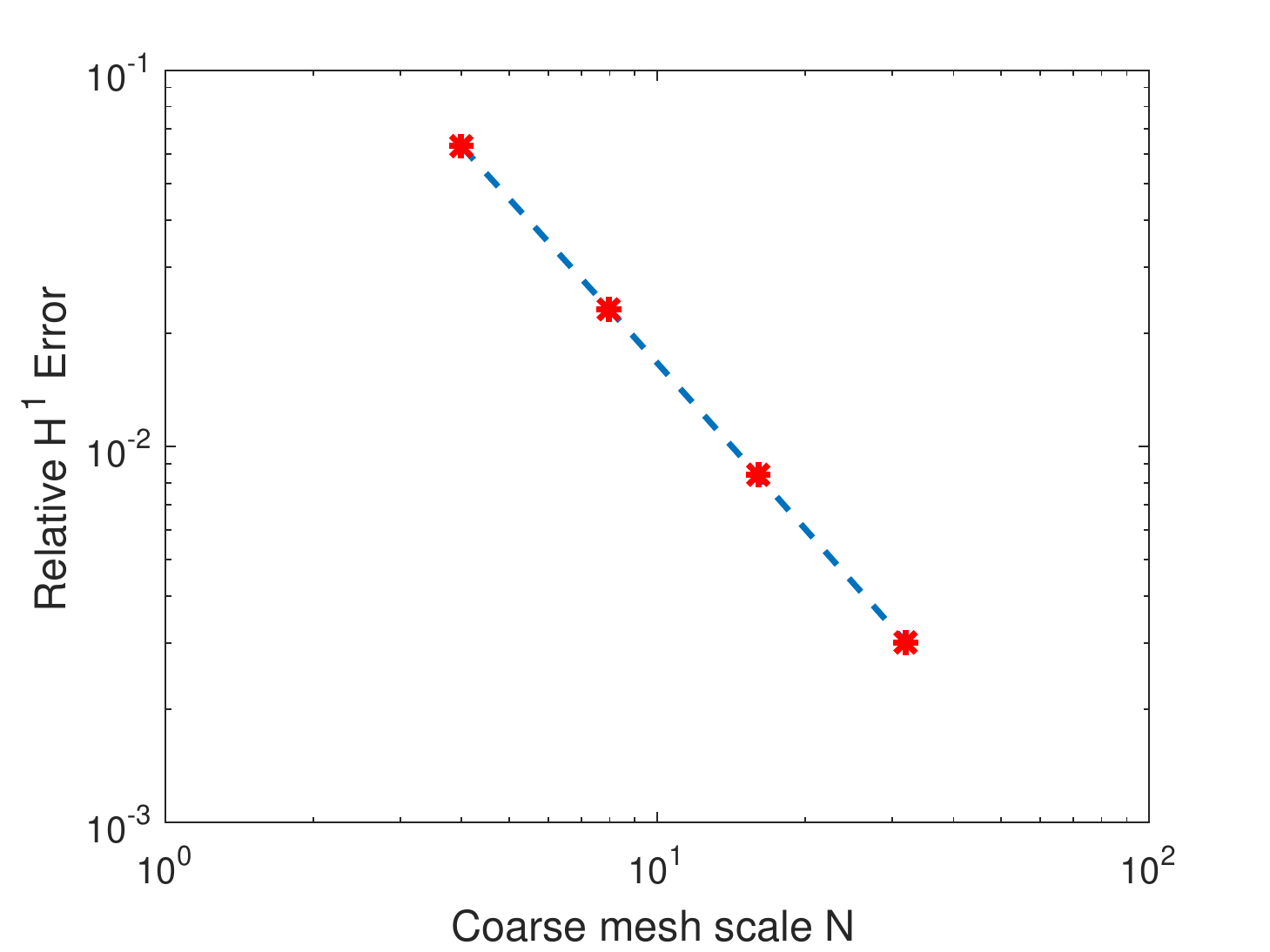}
    \label{fig:ex3_mean}
  \end{subfigure}
  \begin{subfigure}{0.39\textwidth}
    \centering
    \includegraphics[width=\textwidth]{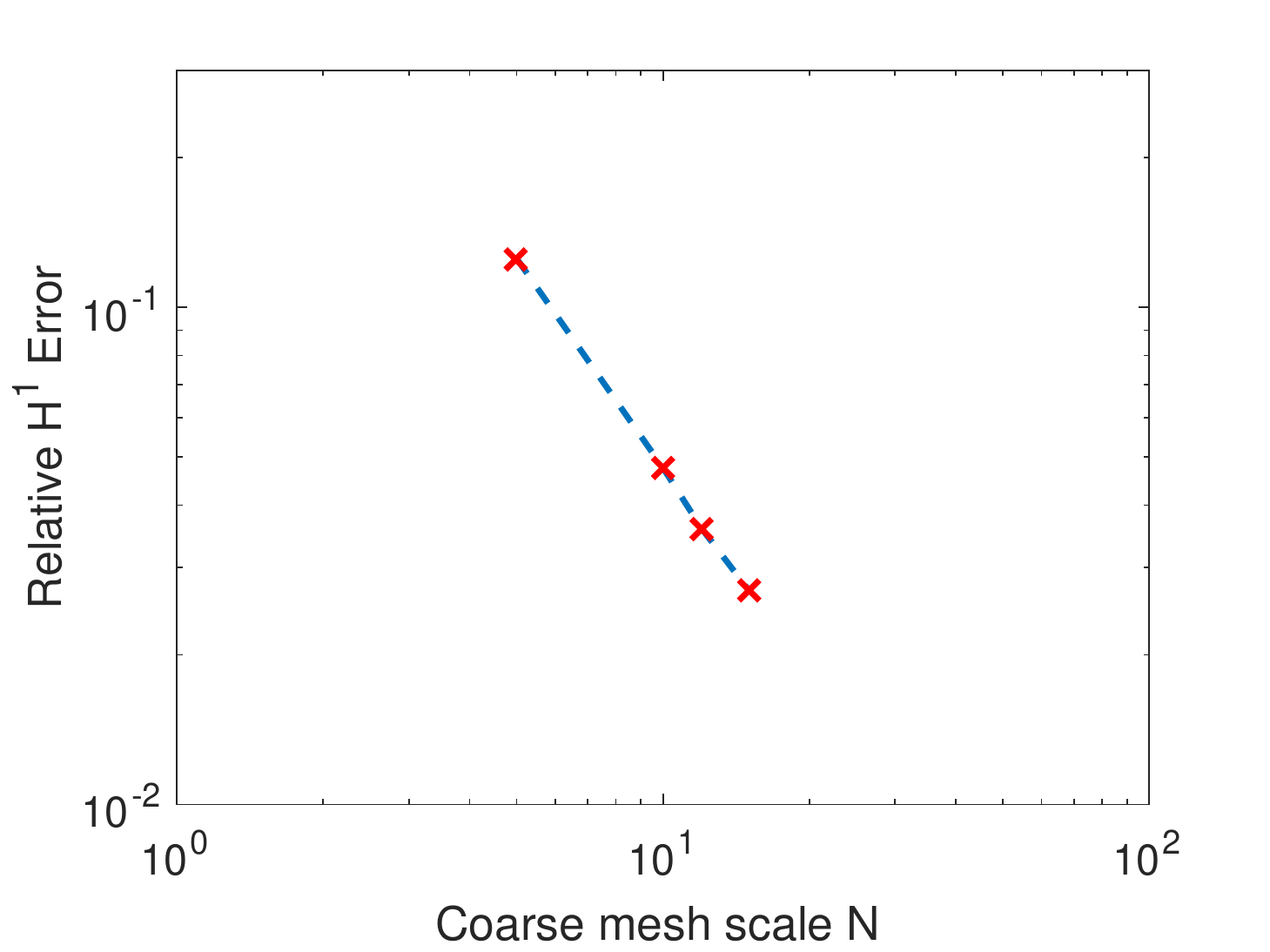}
    \label{fig:ex3_std}
  \end{subfigure}
  \caption{Convergence results with respect to mesh size. Left is for the one r.v. case. Right is for the three r.v.s case.}
  \label{fig:NumEX1_ConvergeForH}
\end{figure}
%

\emph{Verification the convergence rate with respect to the polynomial order.} We shall fix the coarse mesh size to be $H=\frac{1}{16}$. We choose the coarse mesh in such a way that the error from the physical space is small. The coefficient is given by Eq.\eqref{MsDSM_NumEX1_ConvergeForH}, which is parameterized by one random variable. Then we take the polynomial order from $1$ to $15$. In Fig.\ref{fig:NumEX1_ConvergeForgPC},  we plot convergence results with respect to different polynomial orders. One can see the exponential decay with respect to the polynomial orders increase. When the order is bigger than 6, other sources of errors become dominant. In the second case, the coefficient is given by Eq.\eqref{MsDSM_NumEX1_Coef}, which is parameterized by three random variables. We observe the qualitative decay of the error (not shown here).
\begin{figure}[tbph]
  \centering
  \begin{subfigure}{0.39\textwidth}
    \centering
    \includegraphics[width=\textwidth]{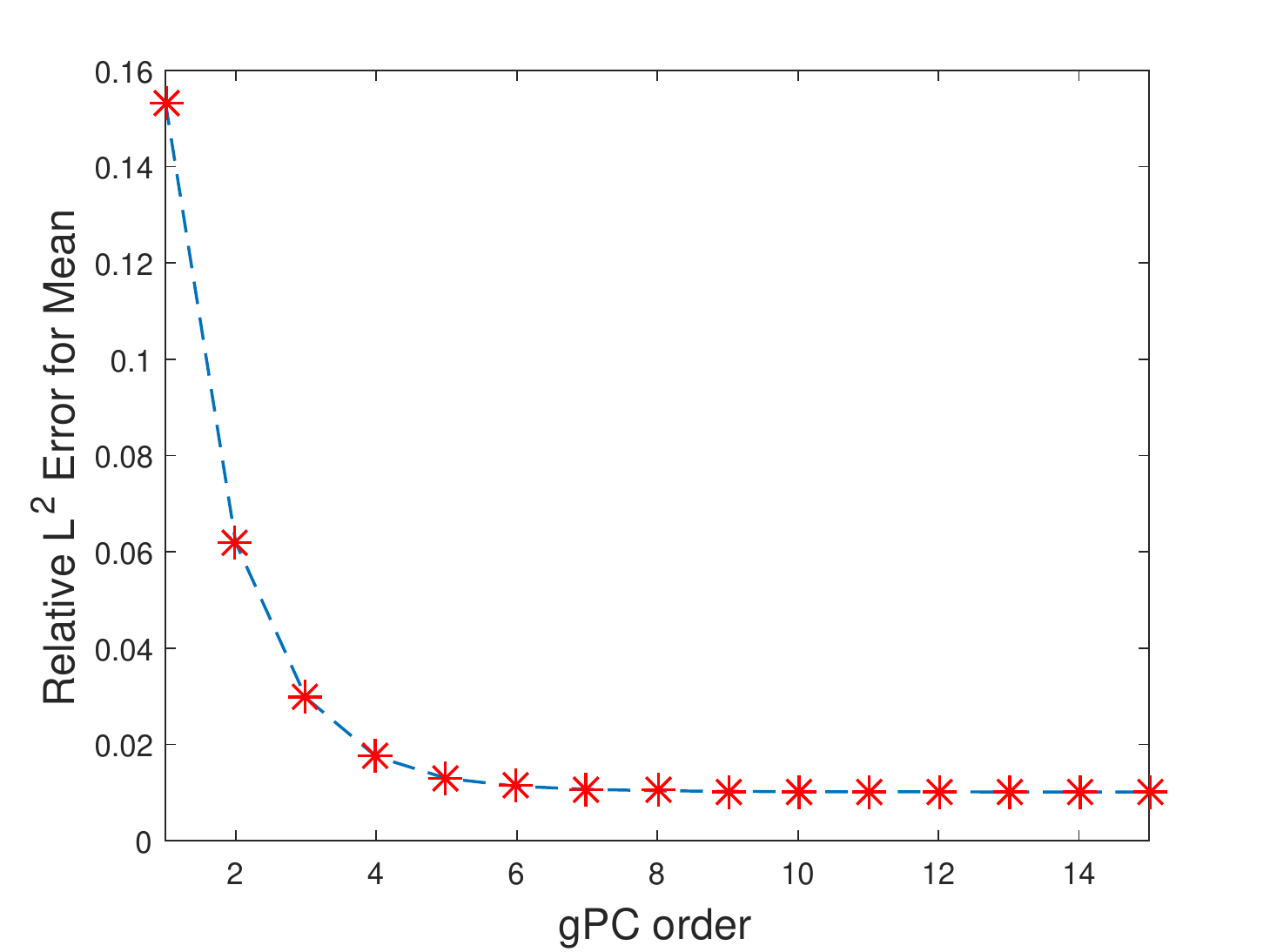}
    \label{fig:ex3_mean}
  \end{subfigure}
  \begin{subfigure}{0.39\textwidth}
    \centering
    \includegraphics[width=\textwidth]{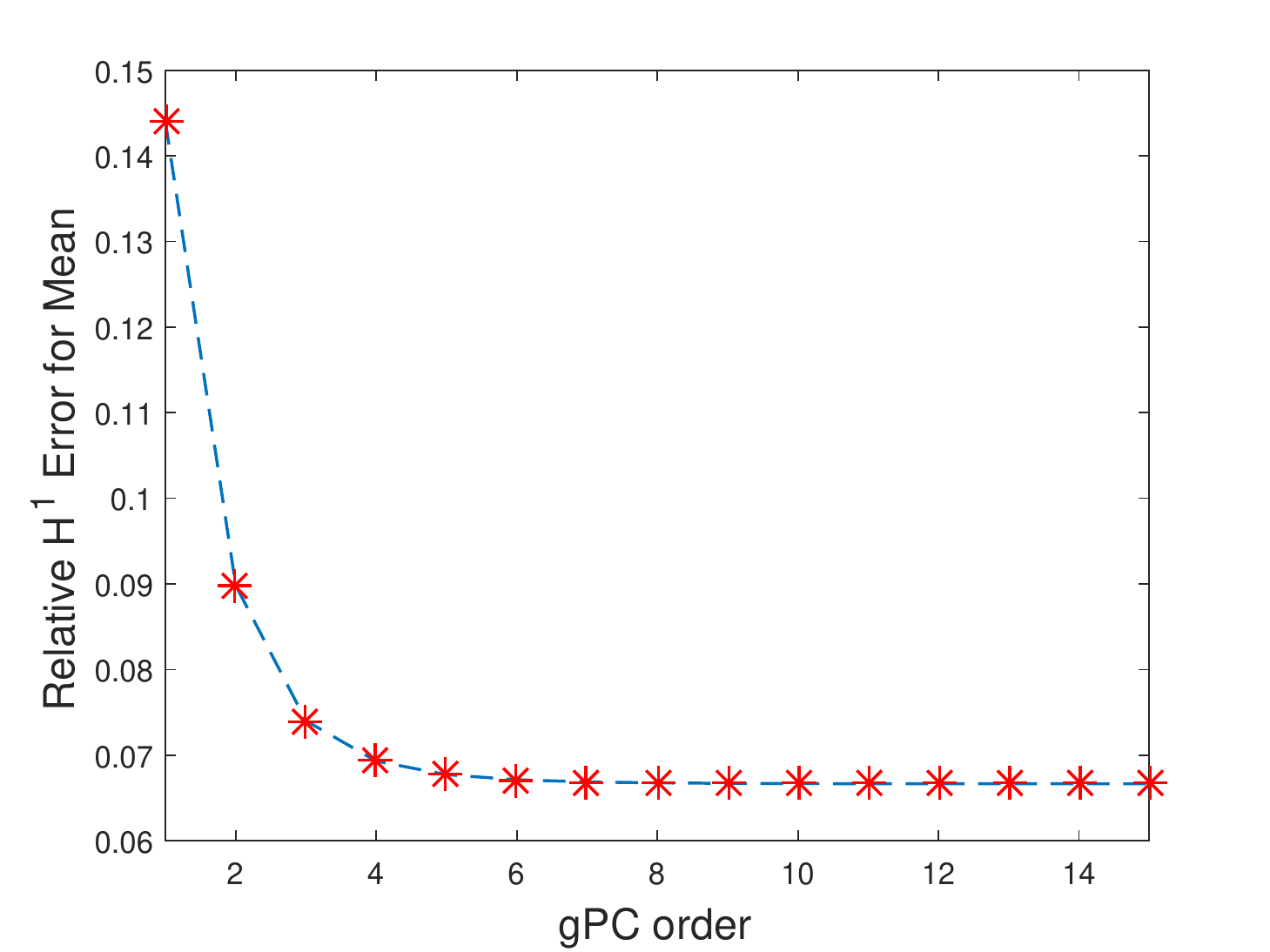}
    \label{fig:ex3_std}
  \end{subfigure}
  \caption{Convergence results with respect to the polynomial chaos order, where $H=\frac{1}{16}$.}
  \label{fig:NumEX1_ConvergeForgPC}
\end{figure}

\emph{Investigate the decay of the eigenvalue $\mu_{i}$.}
We shall show that the eigenvalues $\mu_{i}$ that appear in Eq.\eqref{POD_RandomSpace} have exponentially decay property. Therefore, we only need to use a small number of multiscale data-driven stochastic basis functions in the random space. We test the problem \eqref{MsDSM_NumEX1_Eq} with coefficient \eqref{MsDSM_NumEX1_Coef} and 
$f(x,y)=\sin(2.3\pi x + 0.2)\cos(1.5\pi y - 0.3)$. We choose the fine mesh with $256 \times 256$ grids and
the highest total order of the gPC polynomials is 4. In Fig.\ref{fig:NumEX1_EigenvalueDecay}, we plot the eigenvalues $\mu_{i}$, $i=1,...,10$ of the matrix $YY^{T}$ with $Y$ defined in \eqref{gPC_coefficient_Y}. One can see that the eigenvalues $\mu_{i}$ indeed decay exponentially fast.
\begin{figure}[tbph]
  \centering
  \includegraphics[width=0.5\textwidth]{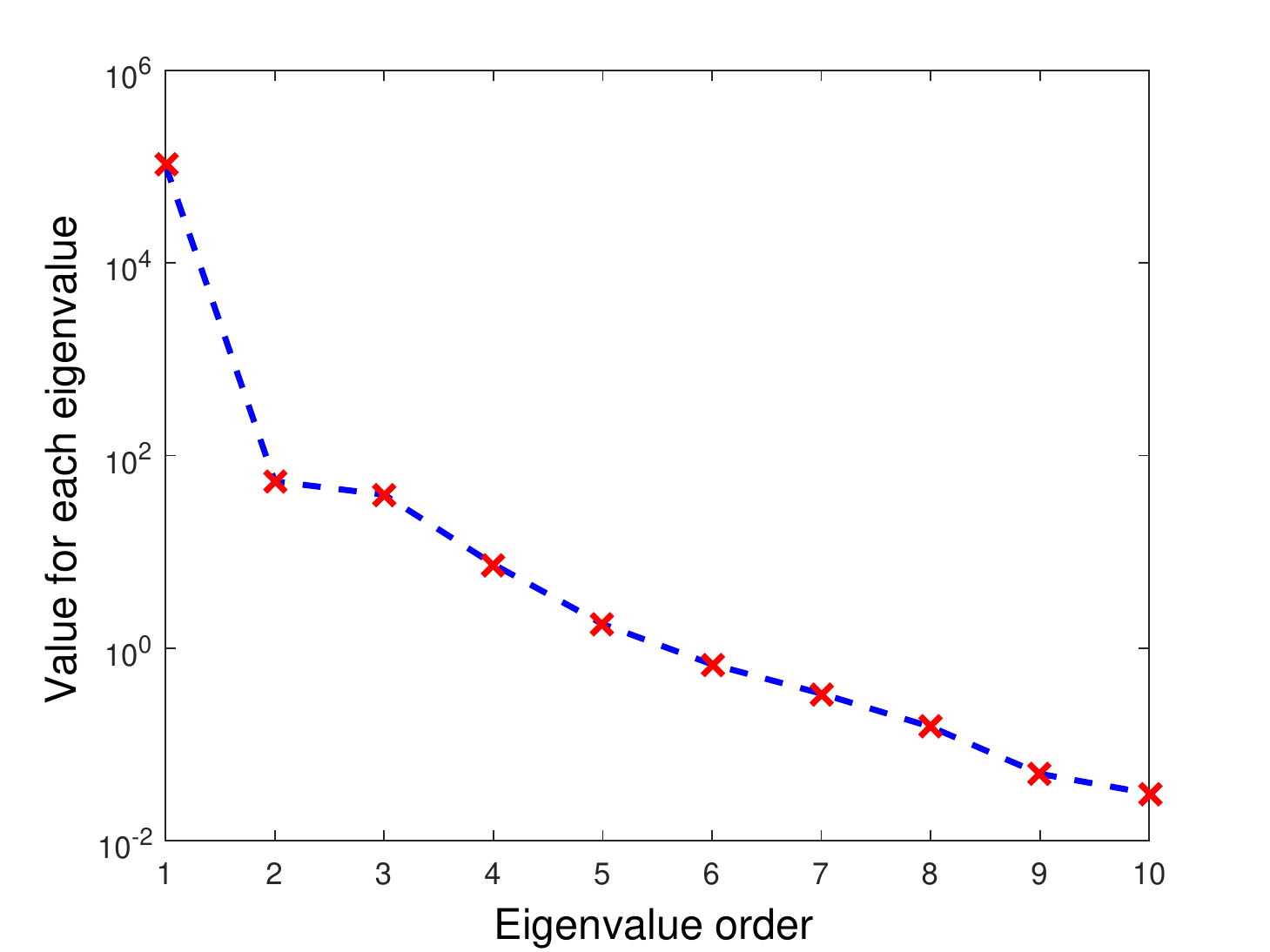}
  \caption{Decay speed of the eigenvalues.}
  \label{fig:NumEX1_EigenvalueDecay}
\end{figure}

\subsection{An example with a coefficient that does not have scale separation.}\label{sec:NumericalExample2}
\noindent
We consider the problem \eqref{MsDSM_NumEX1_Eq}-\eqref{MsDSM_NumEX1_BC} on $D=[0,1]\times[0,1]$ with a coefficient
that does not have scale separation. The coefficient $a^{\varepsilon}(x,y,\omega)$ is a random linear combination of three deterministic coefficient fields plus a constant, i.e.,
\begin{align}
a^{\varepsilon}(x,y,\omega) = \sum_{i=1}^{3}\xi_{i}(\omega)k_{i}(x,y) + 0.5 \label{MsDSM_NumEX2_Coef}
\end{align}
where $\{\xi_{i}\}_{i=1}^{3}$ are independent uniform random variables in $[0,1]$,
and $k_{i}(x,y)$, $i=1,...,3$ are some deterministic coefficients without scale separation. Specifically, $k_{i}(x,y)=|\theta_{i}(x,y)|$, where $\theta_{i}(x,y)$, $i=1,...,3$ are defined on a  $5\times 5$, $9\times 9$, and $17\times 17$ grids over the domain $D$. For each grid cell, the value of $\theta_{i}(x,y)$ is normally distributed. In Fig.\ref{fig:CoefSamples_EX2}, we show four samples of the coefficient $a^{\varepsilon}(x,y,\omega)$. One can see that the coefficient does not have any periodic structure.
\begin{figure}[tbph]
  \centering
  \begin{subfigure}{0.39\textwidth}
    \centering
    \includegraphics[width=\textwidth]{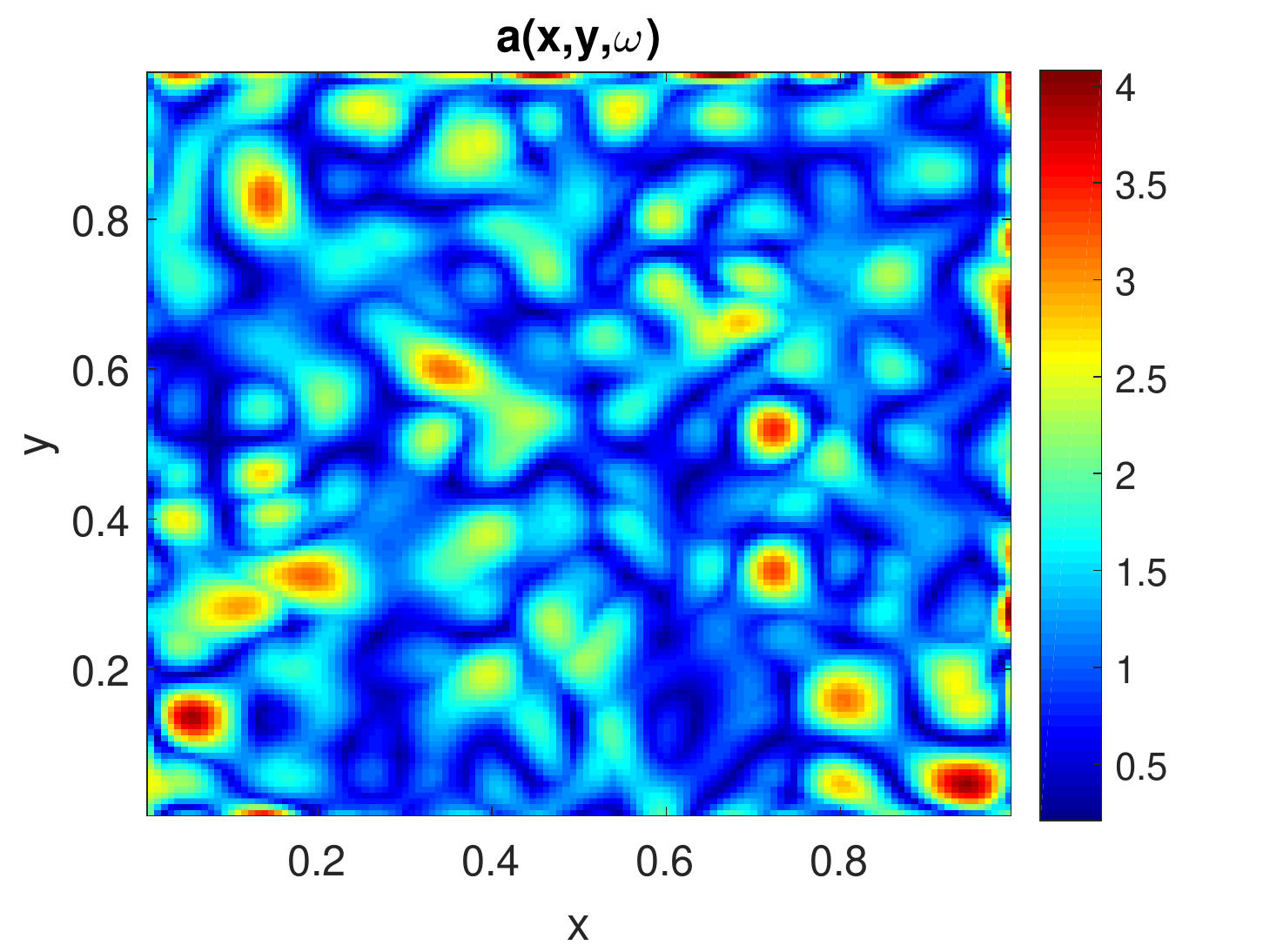}
  \end{subfigure}
  \begin{subfigure}{0.39\textwidth}
    \centering
    \includegraphics[width=\textwidth]{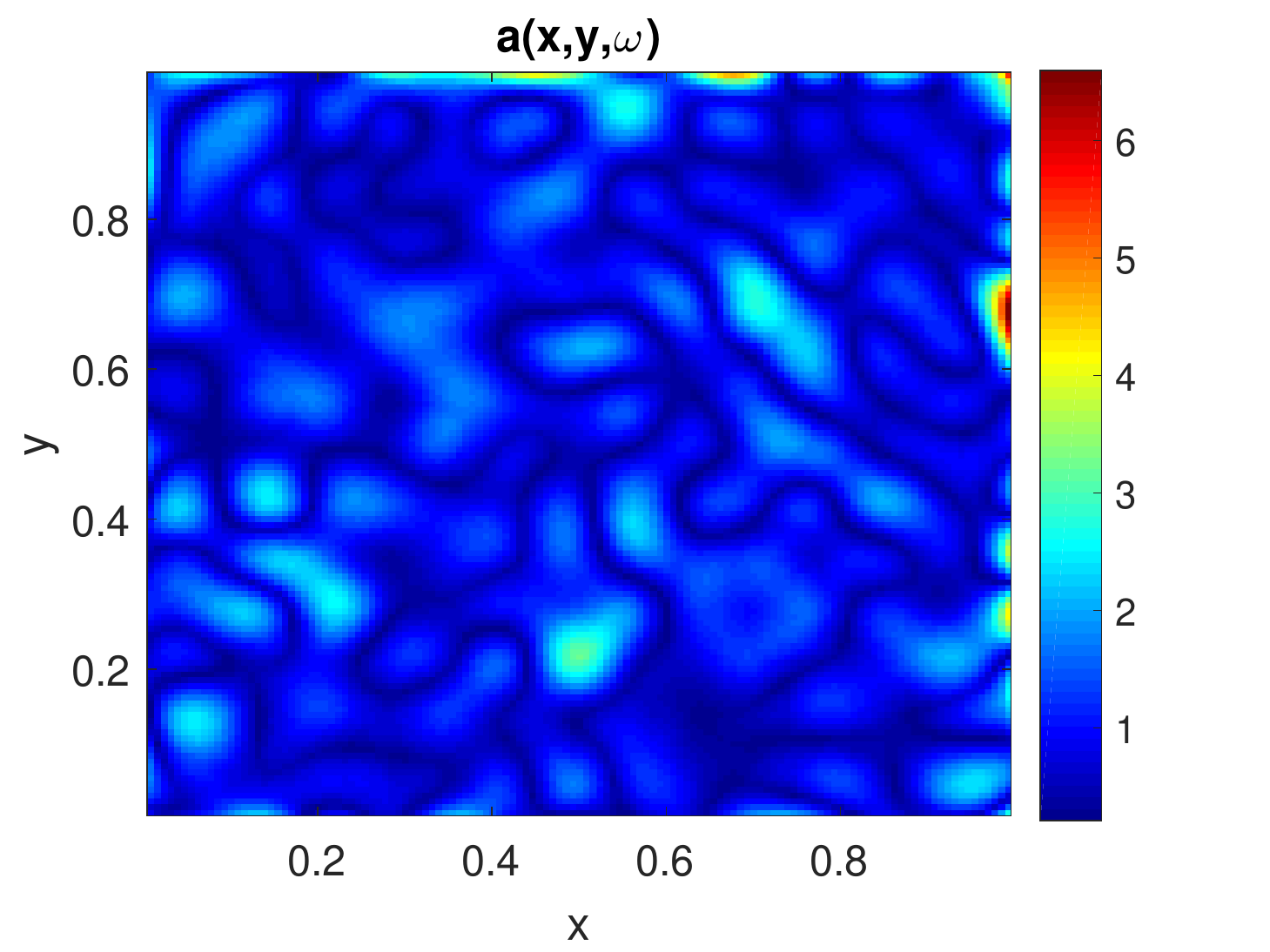}
  \end{subfigure}
  \\
  \begin{subfigure}{0.39\textwidth}
    \centering
    \includegraphics[width=\textwidth]{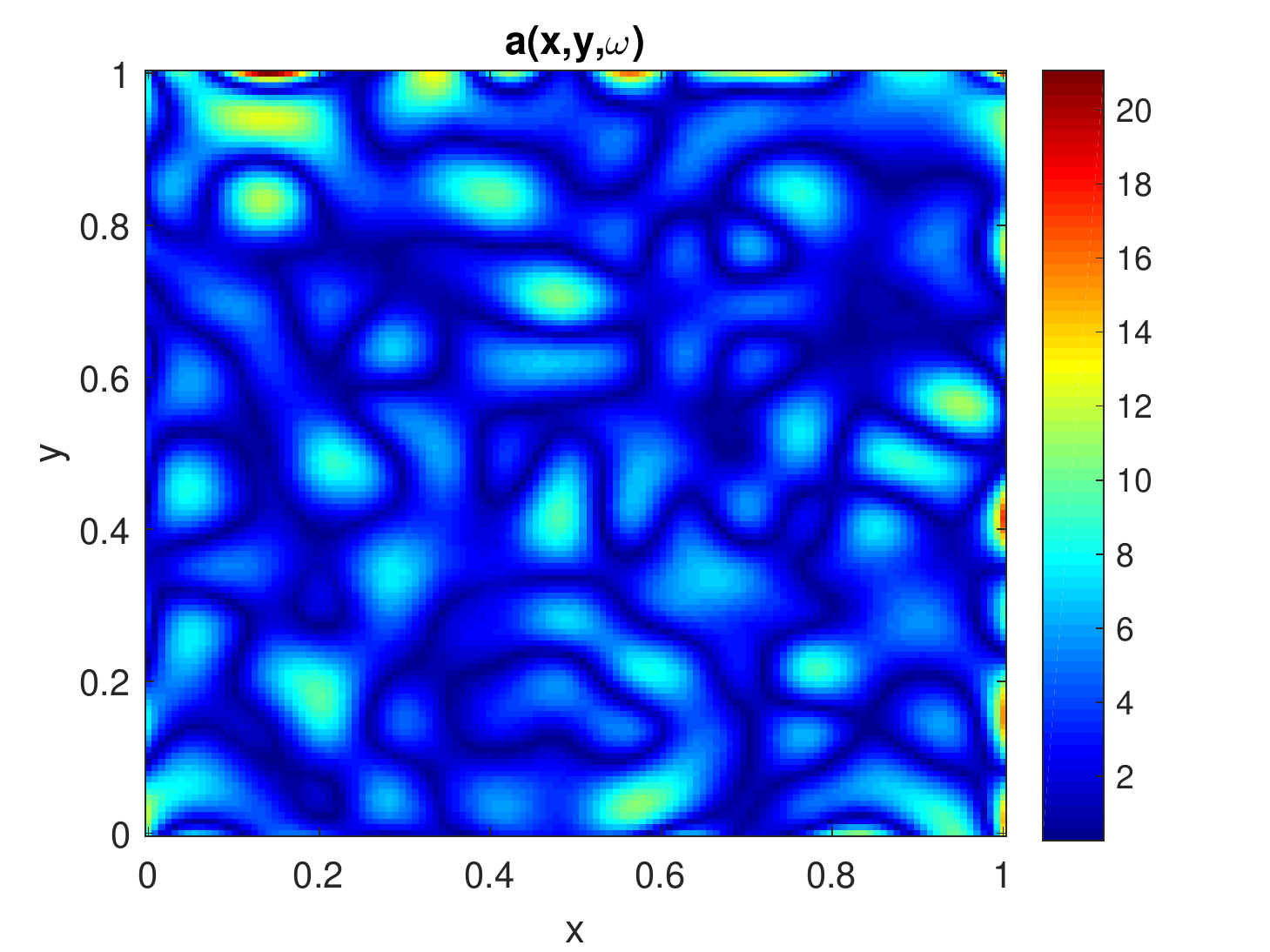}
  \end{subfigure}
  \begin{subfigure}{0.39\textwidth}
    \centering
    \includegraphics[width=\textwidth]{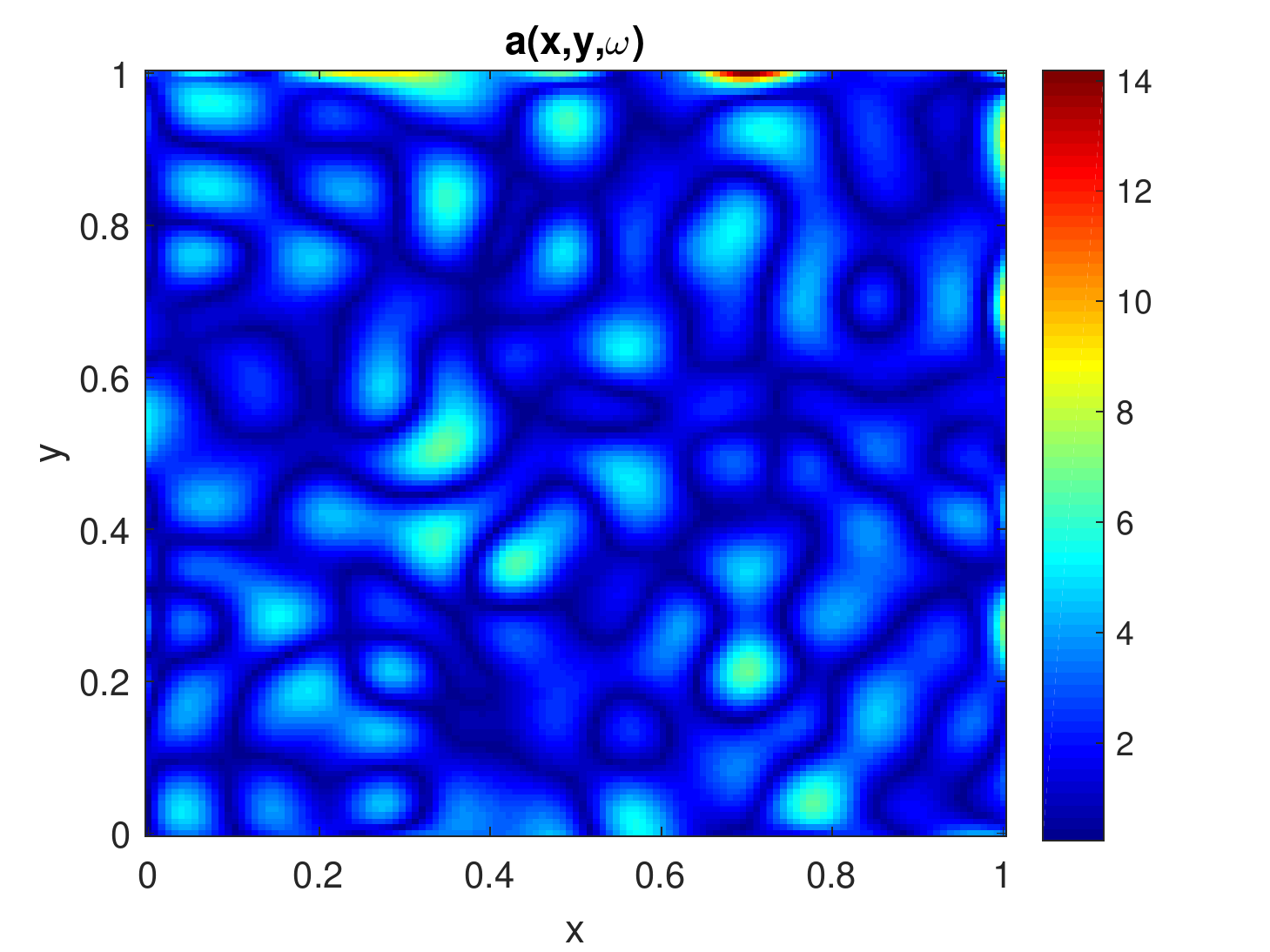}
  \end{subfigure}
  \caption{Some coefficient samples of $a(x,y,\omega)$.  }
  \label{fig:CoefSamples_EX2}
\end{figure}

\emph{Multiquery results in the online stage.}
The implementation of the SCFEM and our method are exactly the same as in the previous example. In the online stage we use them to solve the effective equation of the multiscale SPDE \eqref{MsDSM_NumEX1_Eq}. We randomly generate 20 force functions of the form $f(x,y)\in\{\sin(k_{i}\pi x + l_{i})\cos(m_{i}\pi x + n_{i}\}_{i=1}^{20}$, where $k_{i}$, $l_{i}$, $m_{i}$, and $n_{i}$ are random numbers.  In Fig.\ref{fig:MultiQuery2DErrorL2_EX2}, we show the relative error for the mean function obtained using our method in the $L^2$ norm and the $H^1$ norm, respectively. The results for the STD error is similar to those that we obtained in the previous examples (not shown here).

\emph{Verification of the convergence rate with respect to meshsize $H$.}
We shall test two different coefficients. In the first case, the coefficient is parameterized by one random variable, i.e., $a^{\varepsilon}(x,y,\omega) = \xi_{1}(\omega)k_{1}(x,y) + 0.5$.  The highest order of the gPC basis functions is 7. We change the coarse mesh grid from $4\times 4$ to $64\times 64$. We compare the results on different meshes and calculate the numerical error with respect to the reference solution obtained by the fine mesh $\frac{1}{256}$.
In the second case, the coefficient is given by Eq.\eqref{MsDSM_NumEX2_Coef}, which is parameterized by three random variables. The highest order of the gPC basis functions is 4. We choose the coarse mesh grids as $ 5\times 5$ , $10 \times 10$, $12 \times 12$, and $15 \times 15$. In Fig.\ref{fig:NumEX2_ConvergeForH},  we plot the convergence results with respect to meshsize $H$. For both experiments, we obtain first order convergence for the error in the $H^1$ norm.

\emph{Verification of the convergence rate with respect to the polynomial order.} We fix the course mesh size to be $H=\frac{1}{16}$ and the coefficient is given by $a^{\varepsilon}(x,y,\omega) = \xi_{3}(\omega)k_{3}(x,y) + 0.5$. Then we take the polynomial order from $1$ to $15$. In Fig.\ref{fig:NumEX2_ConvergeForgPC},  we plot convergence results with respect to different polynomial orders. One can find the exponential decay with respect to the polynomial orders.

\begin{figure}[tbph]
  \centering
  \begin{subfigure}{0.39\textwidth}
    \centering
    \includegraphics[width=\textwidth]{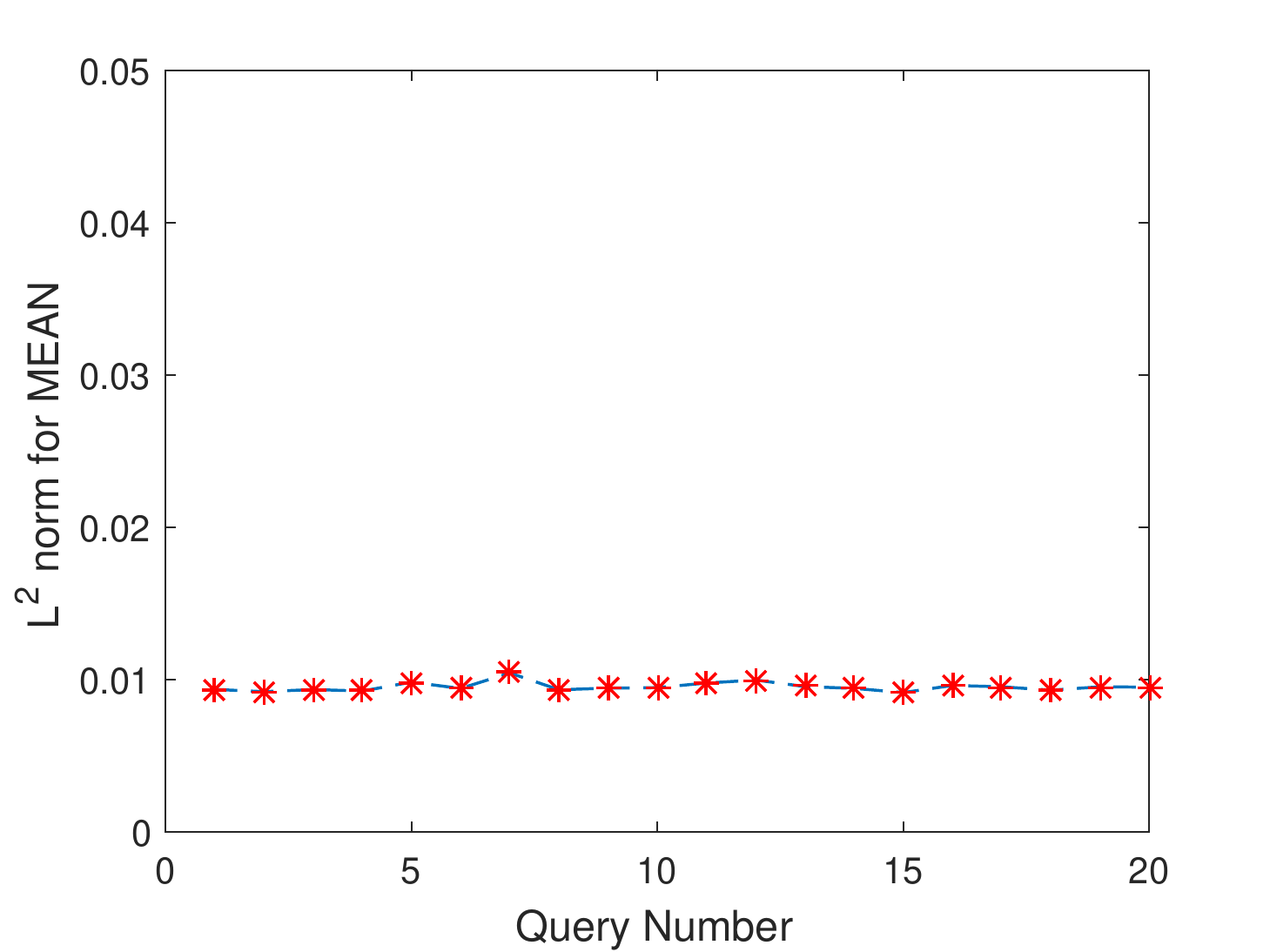}
    \label{fig:ex3_mean}
  \end{subfigure}
  \begin{subfigure}{0.39\textwidth}
    \centering
    \includegraphics[width=\textwidth]{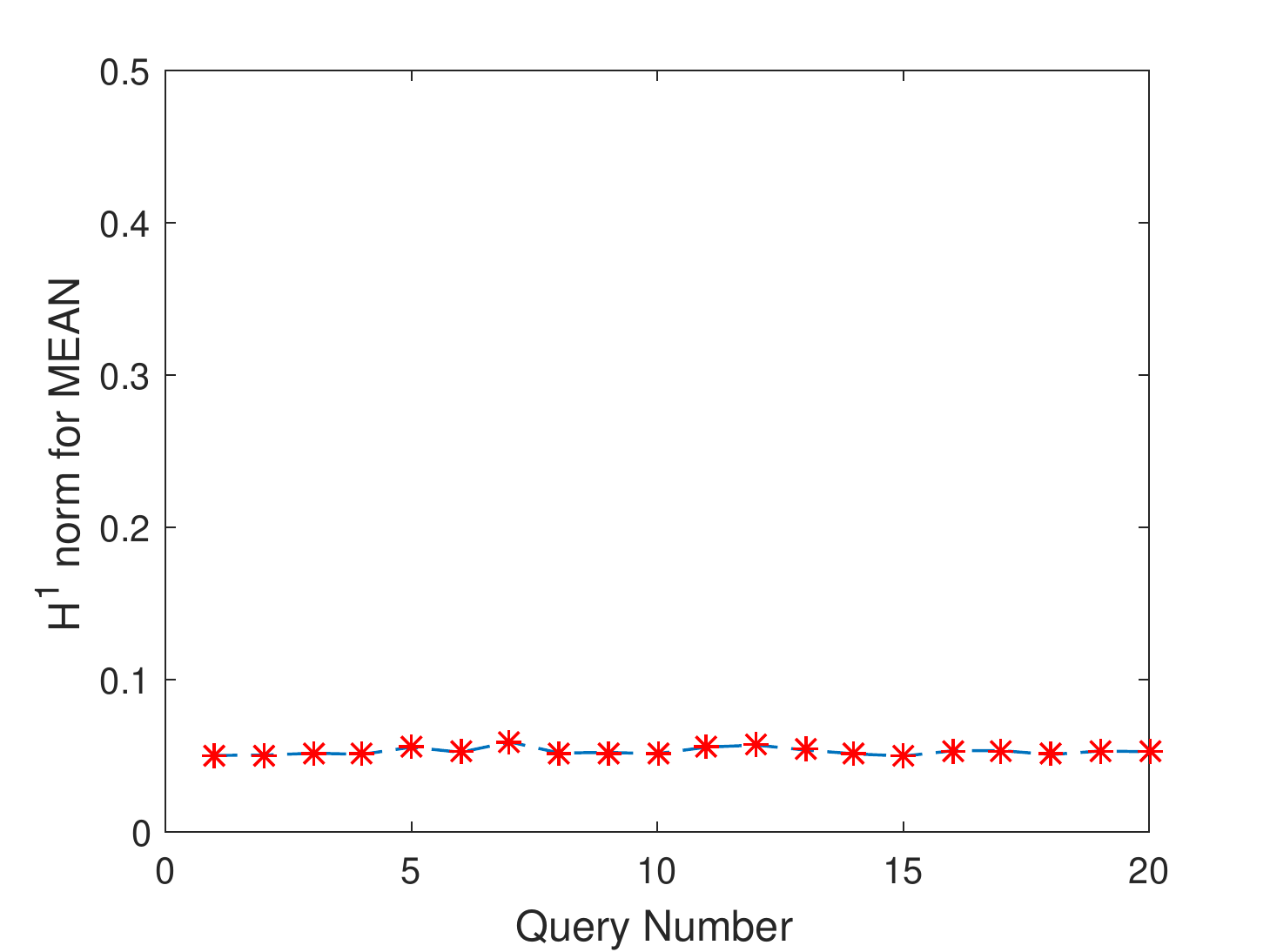}
    \label{fig:ex3_std}
  \end{subfigure}
  \caption{The mean error of our method in the $L^2$ norm and the $H^1$ norm.}
  \label{fig:MultiQuery2DErrorL2_EX2}
\end{figure}

\begin{figure}[tbph]
  \centering
  \begin{subfigure}{0.39\textwidth}
    \centering
    \includegraphics[width=\textwidth]{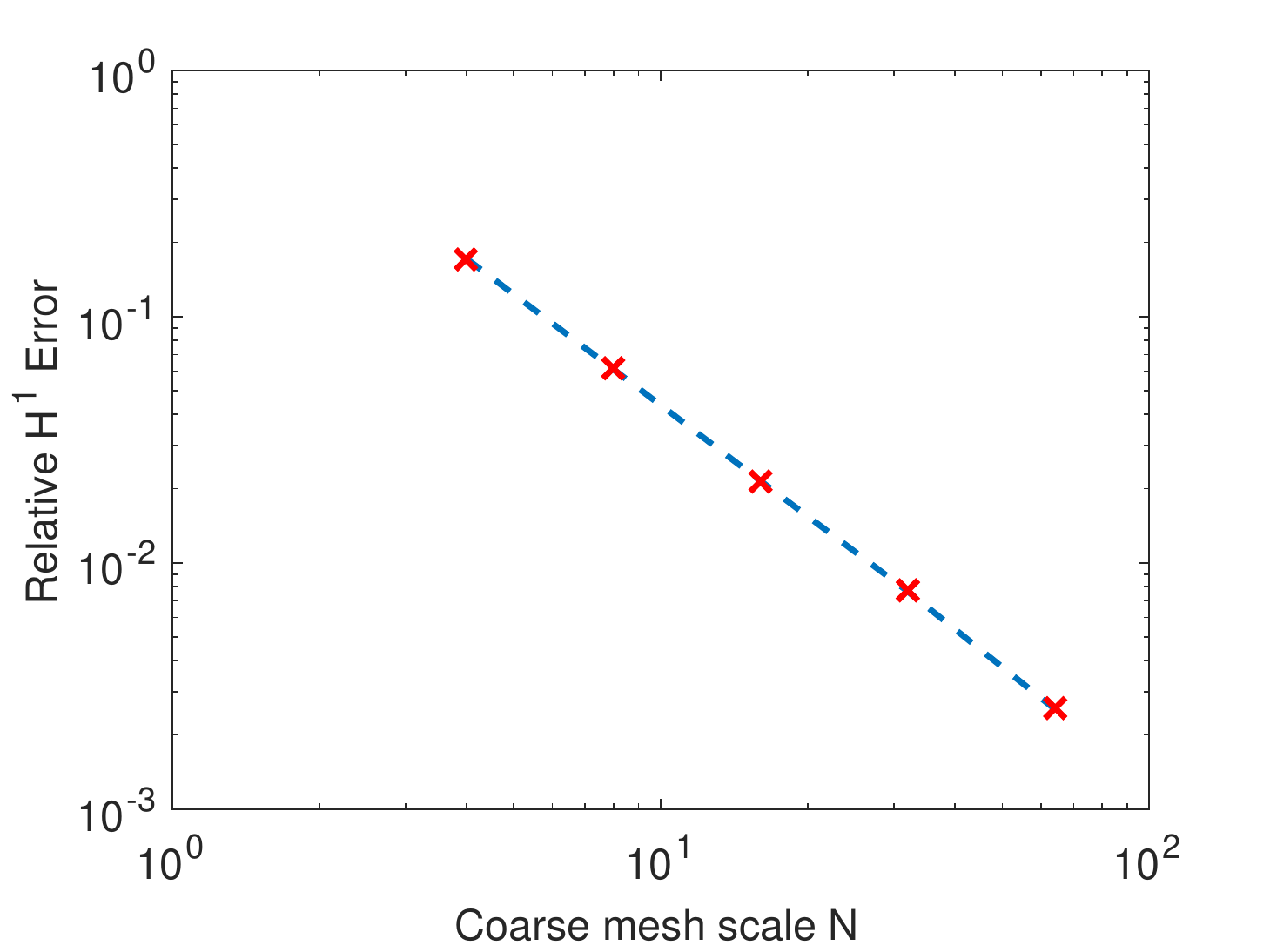}
    \label{fig:ex3_mean}
  \end{subfigure}
  \begin{subfigure}{0.39\textwidth}
    \centering
    \includegraphics[width=\textwidth]{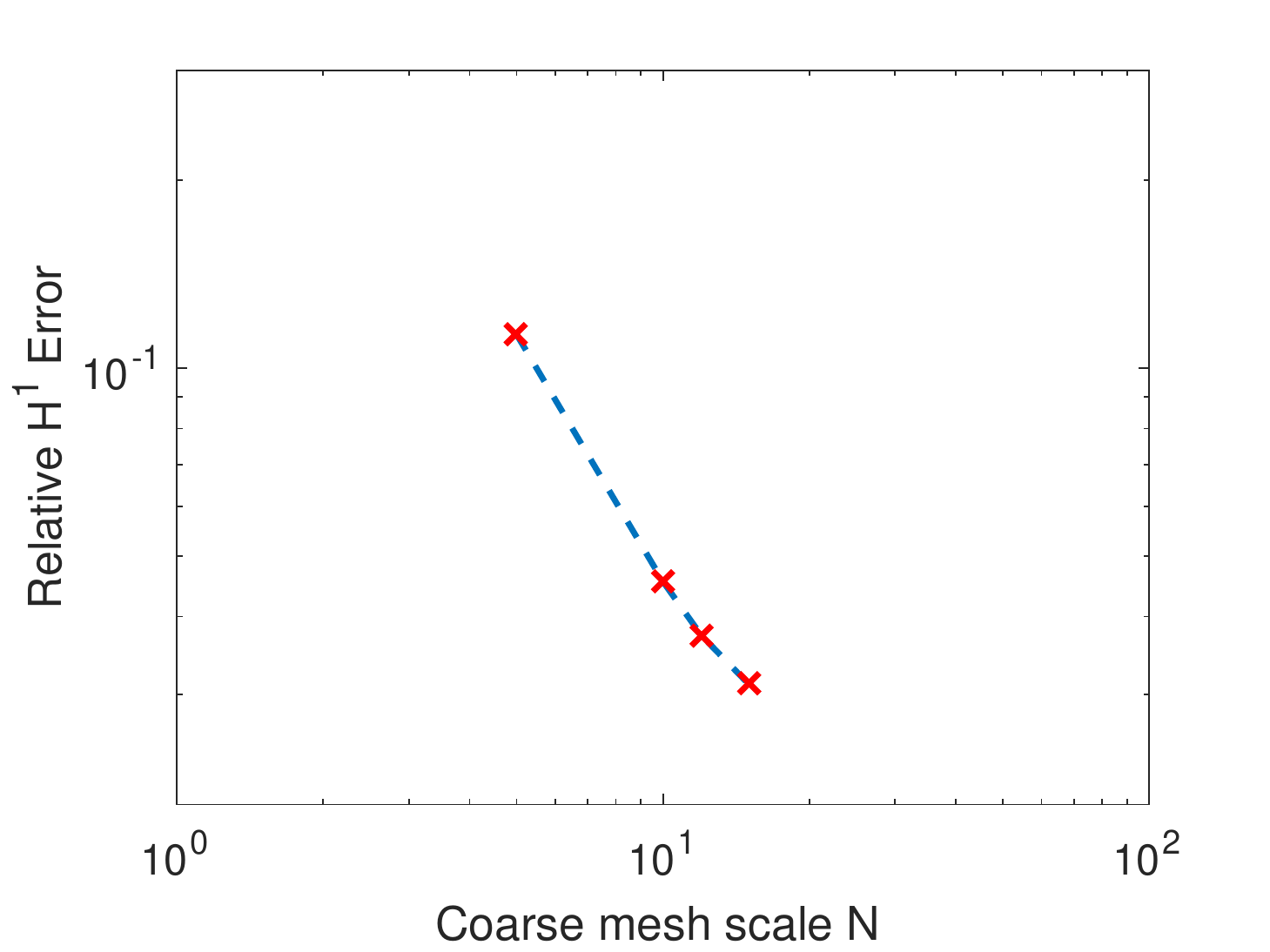}
    \label{fig:ex3_std}
  \end{subfigure}
  \caption{Convergence results with respect to mesh size. Left is for the one r.v. case. Right is for the three r.v.s case.}
  \label{fig:NumEX2_ConvergeForH}
\end{figure}

\begin{figure}[tbph]
  \centering
  \begin{subfigure}{0.39\textwidth}
    \centering
    \includegraphics[width=\textwidth]{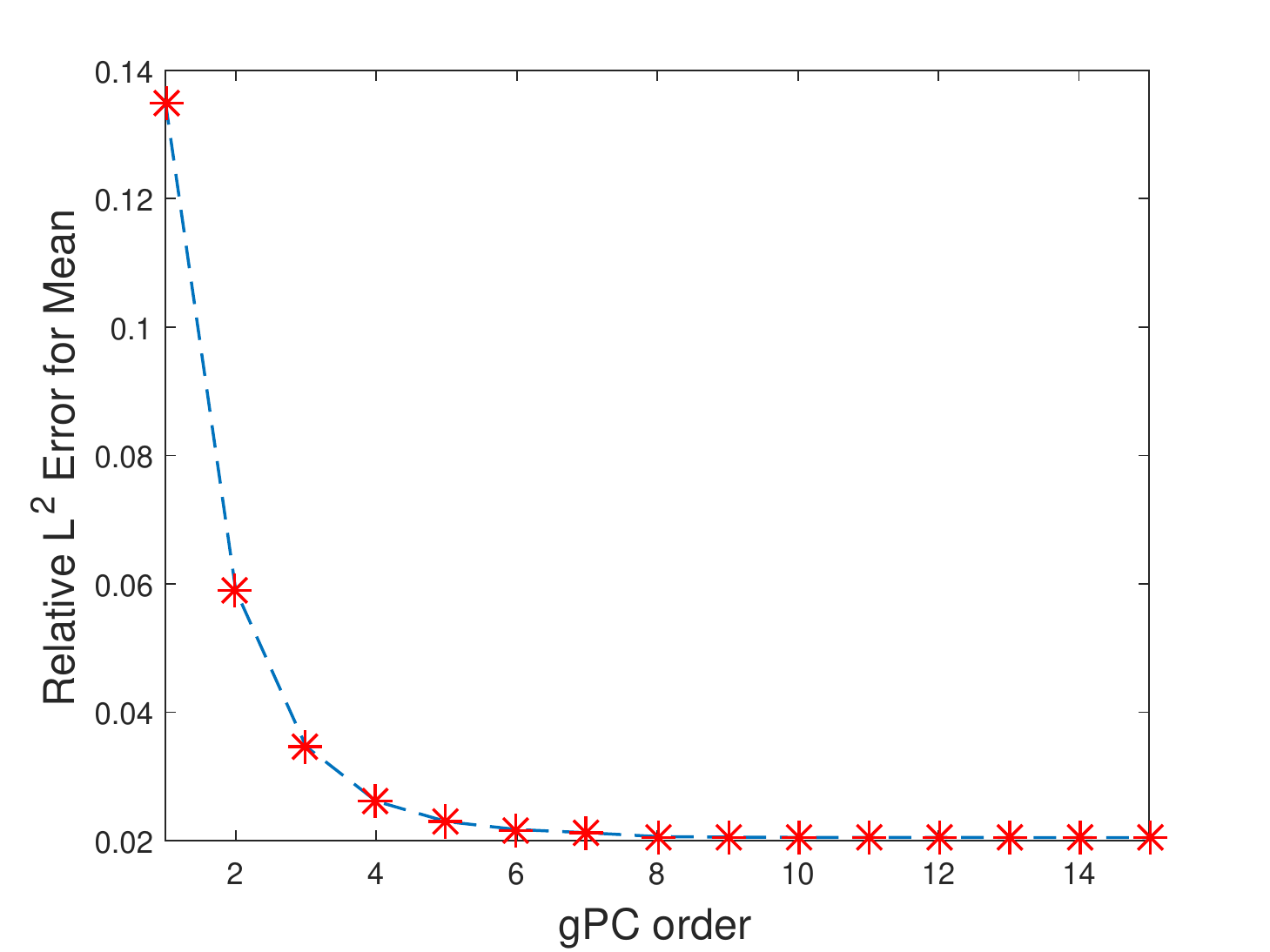}
    \label{fig:ex3_mean}
  \end{subfigure}
  \begin{subfigure}{0.39\textwidth}
    \centering
    \includegraphics[width=\textwidth]{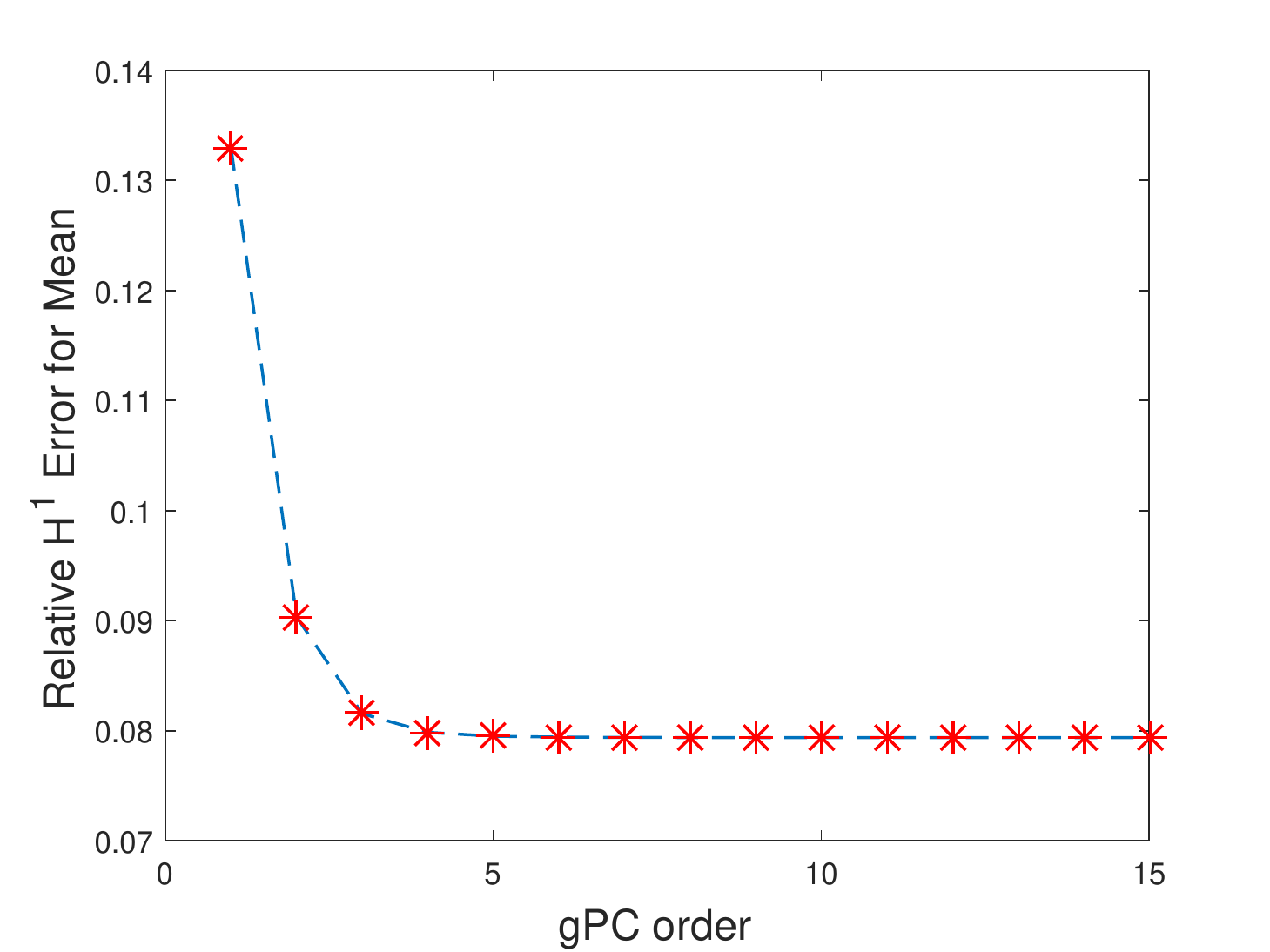}
    \label{fig:ex3_std}
  \end{subfigure}
  \caption{Convergence results with respect to the polynomial chaos order, where $H=\frac{1}{16}$.}
  \label{fig:NumEX2_ConvergeForgPC}
\end{figure}
  
\subsection{An example with a localized random coefficient}\label{sec:NumericalExample3}
\noindent
Finally we consider the problem \eqref{MsDSM_NumEX1_Eq}-\eqref{MsDSM_NumEX1_BC} on $D=[0,1]\times[0,1]$ with a
coefficient that has localized random features. The coefficient $a^{\varepsilon}(x,y,\omega)$ is
given by
\begin{small}
\begin{align}
a^{\varepsilon}(x,y,\omega) &= 0.2 + \sum_{i=1}^{3}\frac{2+p_i\sin(\frac{2\pi x}{\varepsilon_i})}{2-p_i\cos(\frac{2\pi y}{\varepsilon_i})}\xi_i(\omega)\textbf{1}_{D_1}(x,y)
+ \sum_{i=4}^{6}  \frac{2+p_i\cos(\frac{2\pi x}{\varepsilon_i})}{2-p_i\sin(\frac{2\pi y}{\varepsilon_i})}\xi_i(\omega)\textbf{1}_{D_2}(x,y) \nonumber\\ 
&+ \sum_{i=7}^{9}  \frac{2+p_i\sin(\frac{2\pi (x-y)}{\varepsilon_i})}{2-p_i\cos(\frac{2\pi (x-y)}{\varepsilon_i})}\xi_i(\omega)\textbf{1}_{D_3}(x,y) 
+ \sum_{i=10}^{12}  \frac{2+p_i\cos(\frac{2\pi (x-0.5)}{\varepsilon_i})}{2-p_i\sin(\frac{2\pi (x-y)}{\varepsilon_i})}\xi_i(\omega)\textbf{1}_{D_4}(x,y),  \label{MsDSM_NumEX3_Coef}
\end{align}
\end{small}
where $D_1 = [1/8,3/8]\times [1/8,3/8]$, $D_2 = [5/8,7/8]\times [1/8,3/8]$, $D_3 = [1/8,3/8]\times [5/8,7/8]$, $D_4 = [5/8,7/8]\times [7/8,7/8]$, $\{\xi_{i}\}_{i=1}^{12}$ are independent uniform random variables in $[0,1]$, 
$[\varepsilon_1,\varepsilon_2,...,\varepsilon_{12}]=[1/11.1,1/10.2,1/15.1,1/15.4,1/13.9,1/16.1,1/17.3,1/11.3,1/13.3,1/18.1,1/16.7,1/18]$, and \\
$ [p_1,p_2,...,p_{12}]=[1.81,1.85,1.90,1.87,1.82,1.85,1.89,1.85,1.82,1.83,1.84,1.86]$. In Fig.\ref{fig:CoefSamples_EX3}, we show two samples of the coefficient $a^{\varepsilon}(x,y,\omega)$. One can see that the localized random features of the coefficient.
\begin{figure}[tbph]
	\centering
	\begin{subfigure}{0.39\textwidth}
		\centering
		\includegraphics[width=\textwidth]{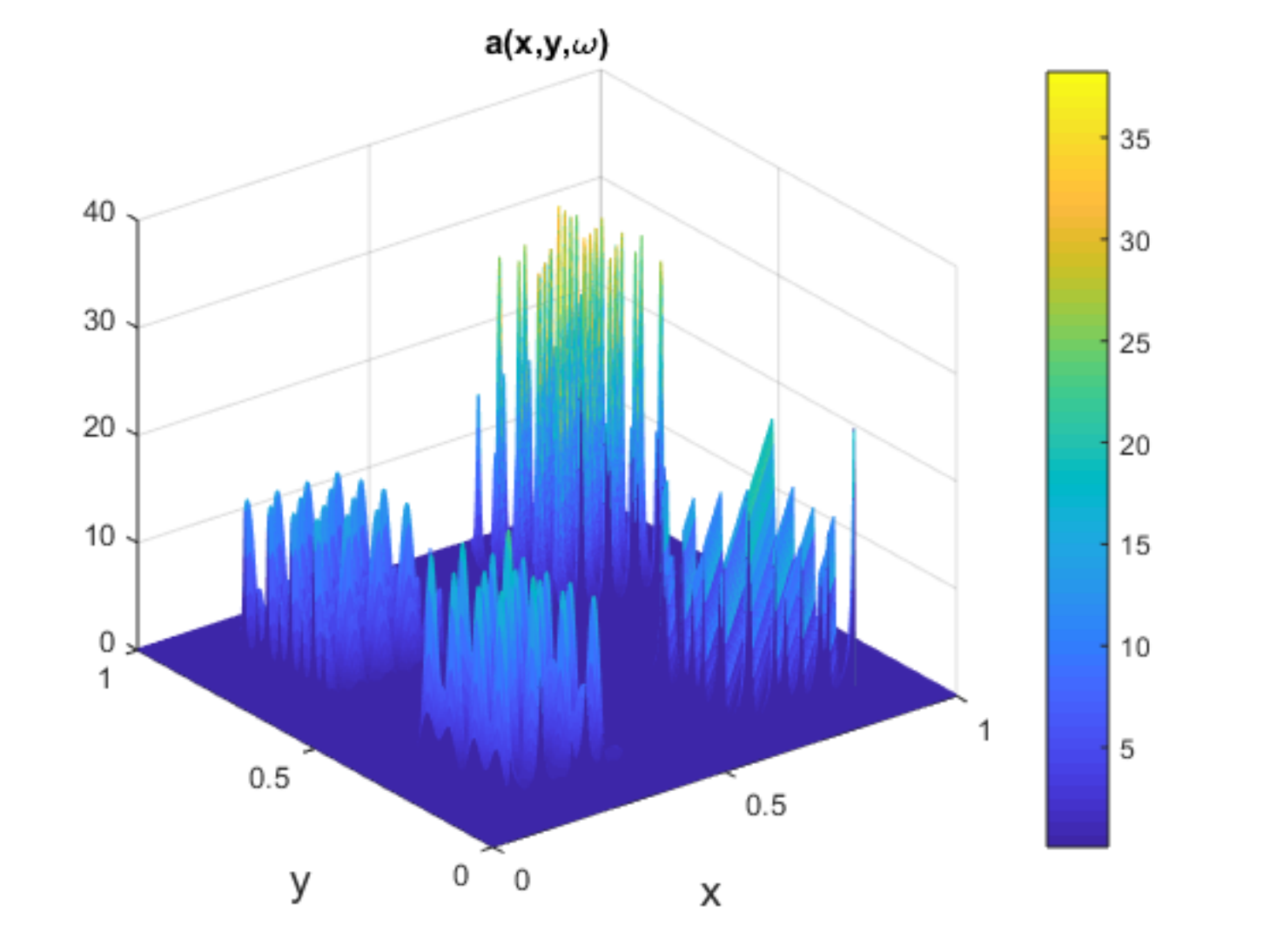}
	\end{subfigure}
	\begin{subfigure}{0.39\textwidth}
		\centering
		\includegraphics[width=\textwidth]{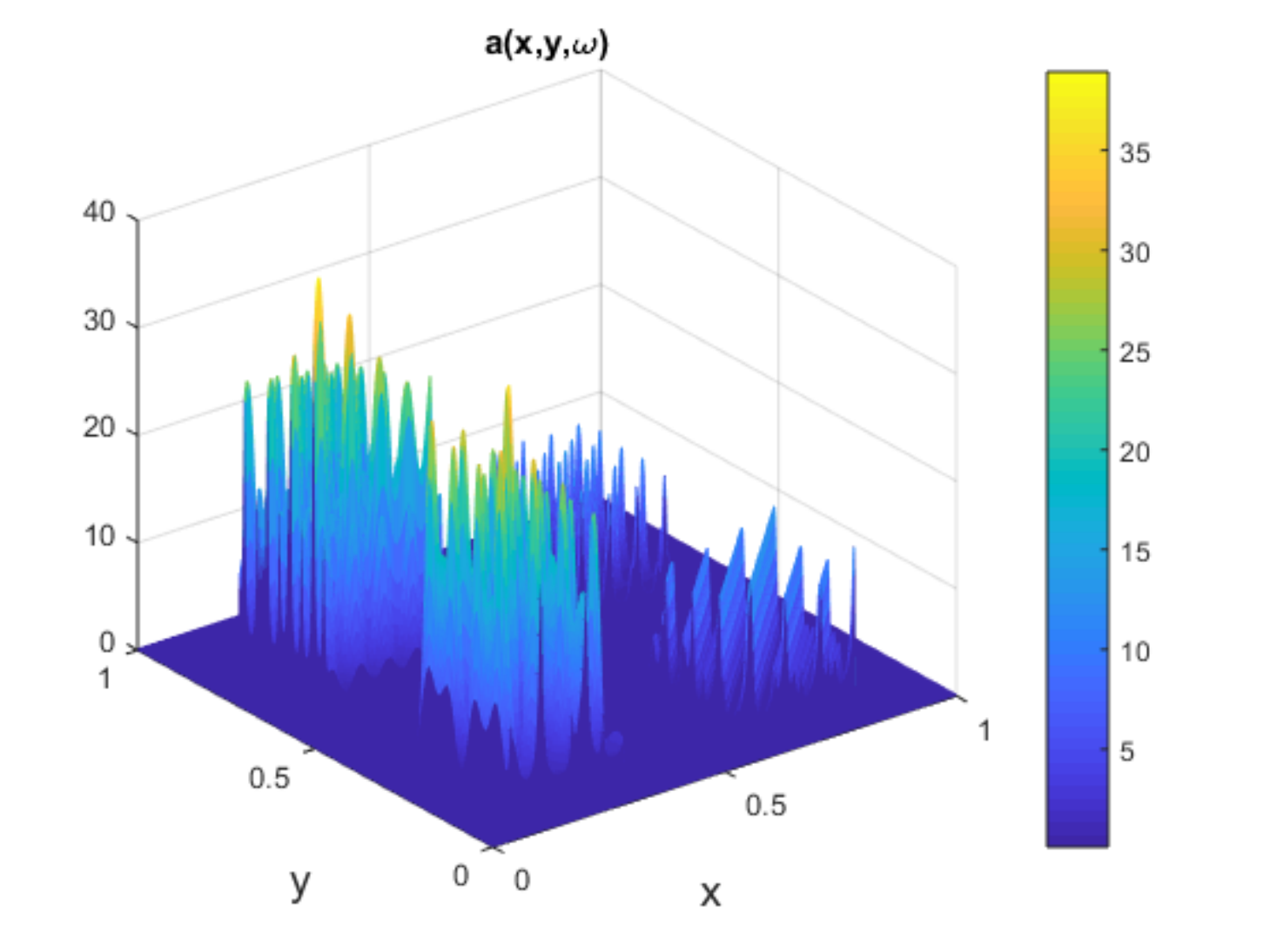}
	\end{subfigure} 	
  \caption{Some coefficient samples of $a(x,y,\omega)$.  }
  \label{fig:CoefSamples_EX3}	
\end{figure}

If we implement the gPC method in a brute force way, it is prohibitively expensive. For instance, 
choosing the total order as $3$ (see Eq.\eqref{gPC_index}) will generate $\frac{(12+3)!}{12!3!}=455$ gPC basis functions. Motivated by the observation in Fig.\ref{fig:ProfilesLocalBasis}, 
we propose an adaptive method to truncate the order of polynomials in the gPC basis depending on the physical location. For instance, when we solve the optimization problem  \eqref{OC_MsStocBasis_Obj}-\eqref{OC_MsStocBasis_Cons1} to compute $\psi_{i,k}(x,\xi(\omega))$ with the vertex $x_i\in D_1$, we define the sparse truncation order as $t=(3,3,3,1,1,1,1,1,1,0,0,0)$ so the associated sparse truncated gPC basis index is $\sJ_{12}^{3,t} = \bkc{\minda \,|\, \minda=\bkr{\alpha_1, \alpha_2, \cdots, \alpha_{12}}, 0 \leq \alpha_i \leq t_i, \alpha_i \in \NN, \bkl{\alpha}=\sum_{i=1}^{12} \alpha_i \le 3}$, which will generate 49 gPC basis functions. This sparse truncated gPC index $\sJ_{12}^{3,t}$ will be used in Eq.\eqref{RepresentOCMaStocBasis}, when we compute the basis $\psi_{i,k}(x,\xi(\omega))$.  The sparse truncation orders associated with other vertices can be defined accordingly. Using the sparse truncation techniques, we dramatically reduce the computational cost.

\emph{Multiquery results in the online stage.}
The implementation of the SCFEM and our method are exactly the same as in the previous examples. In the online stage we use them to solve the effective equation of the multiscale SPDE \eqref{MsDSM_NumEX1_Eq}. We randomly generate 20 force functions of the form $f(x,y)\in\{\sin(k_{i}\pi x + l_{i})\cos(m_{i}\pi x + n_{i}\}_{i=1}^{20}$, where $k_{i}$, $l_{i}$, $m_{i}$, and $n_{i}$ are random numbers.  In Fig.\ref{fig:MultiQuery2DErrorL2H1_EX3}, we show the relative error for the mean function obtained using our method in the $L^2$ norm and the $H^1$ norm, respectively. The results for the STD error are similar. We also obtain the same convergence rate with respect to meshsize $H$ and the polynomial order (not shown here).
\begin{figure}[tbph]
  \centering
  \begin{subfigure}{0.39\textwidth}
    \centering
    \includegraphics[width=\textwidth]{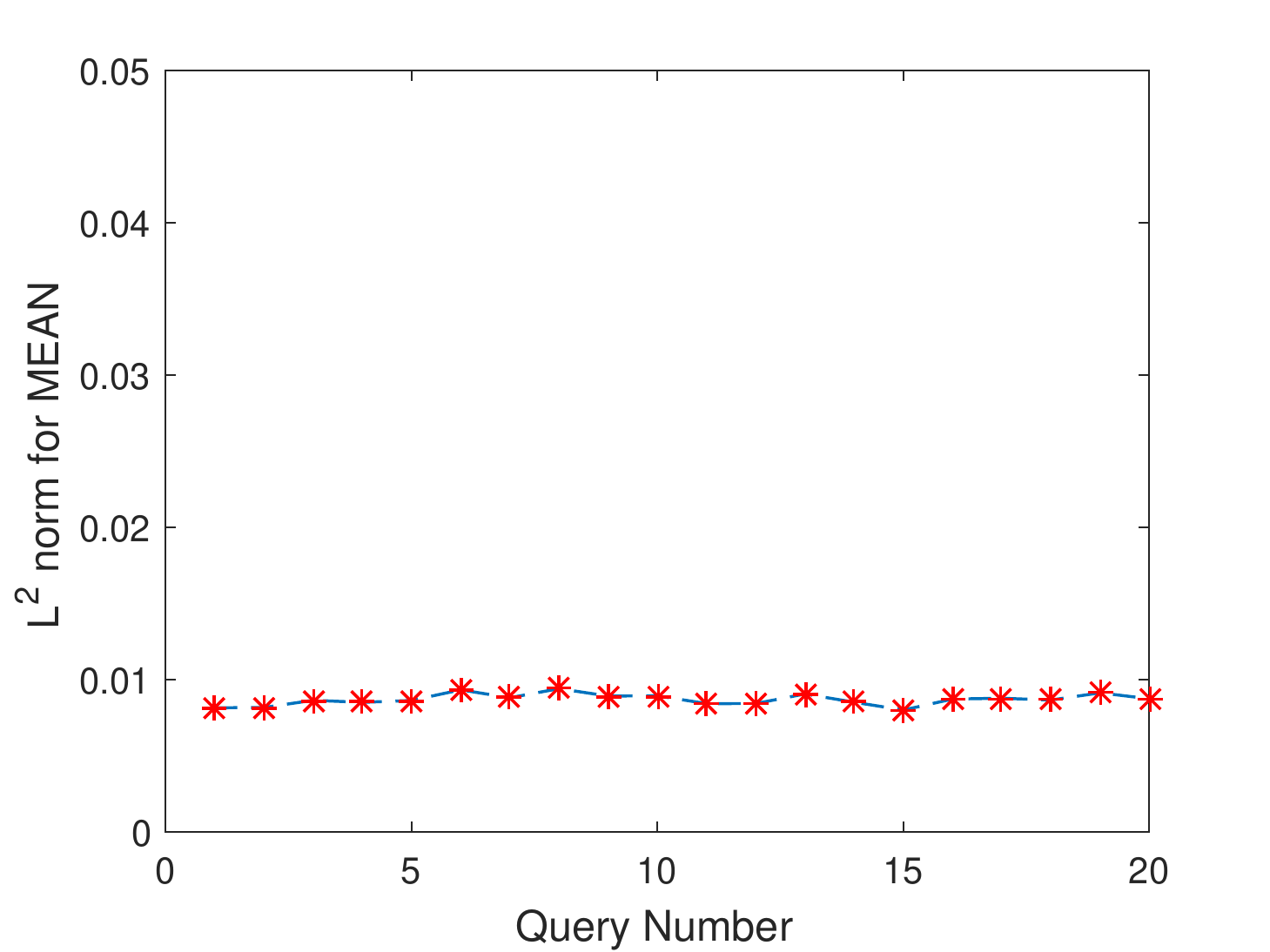}
    \label{fig:ex3_mean}
  \end{subfigure}
  \begin{subfigure}{0.39\textwidth}
    \centering
    \includegraphics[width=\textwidth]{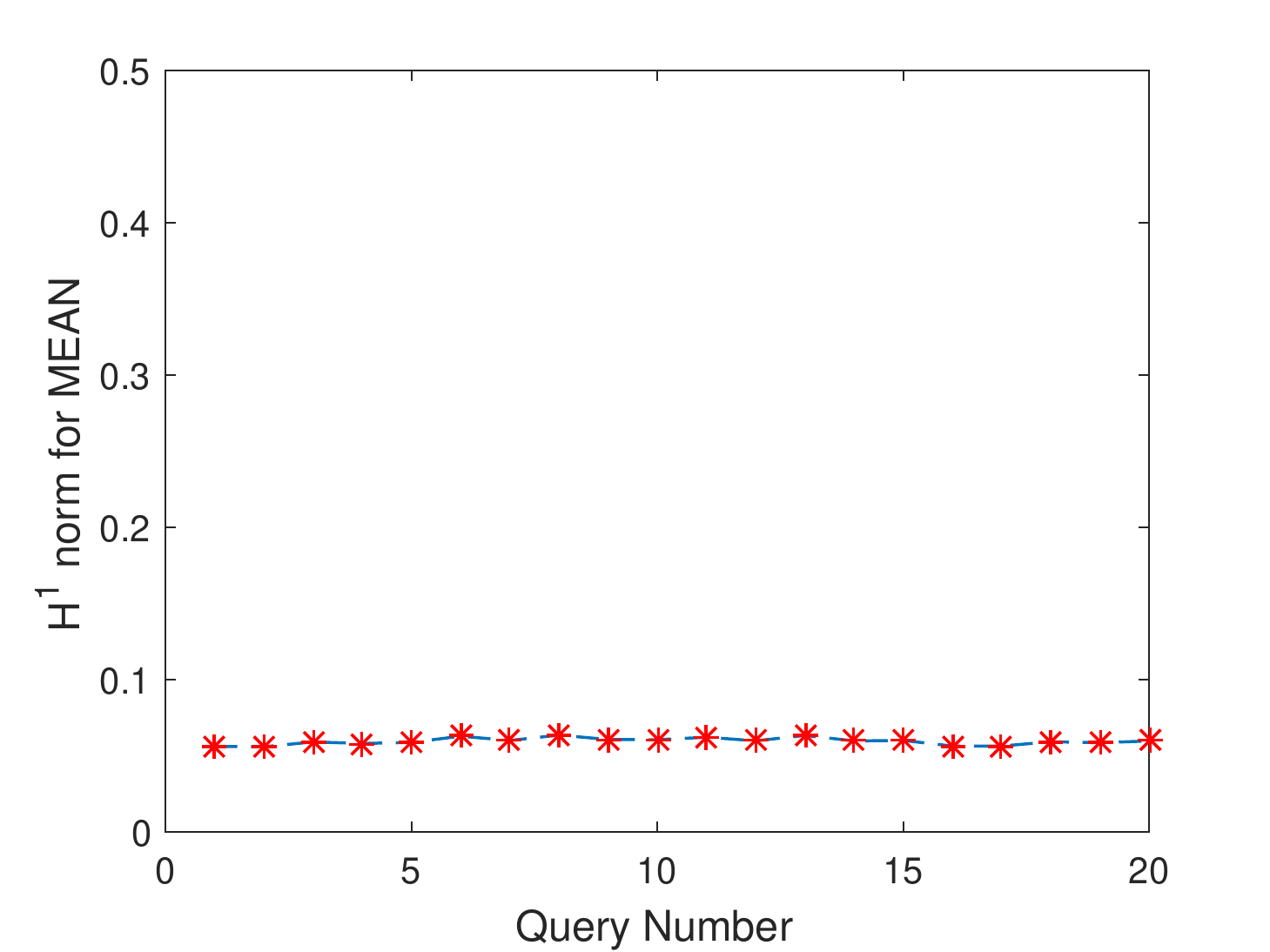}
    \label{fig:ex3_std}
  \end{subfigure}
  \caption{The mean error of our method in the $L^2$ norm and the $H^1$ norm.}
  \label{fig:MultiQuery2DErrorL2H1_EX3}
\end{figure}



\section{Conclusion} \label{sec:Conclusion}
\noindent
In this paper, we developed a novel model reduction method to construct multiscale data-driven stochastic basis
functions that can be used to solve multiscale elliptic PDEs with random coefficients in a multiquery setting. 
These problems arise from various applications, such as heterogeneous porous media flow problem in water aquifer and oil reservoirs simulations. Our method consists of the offline and online stages. In the offline stage, we construct the basis functions through solving localized optimization problems. In the online stage, we can efficiently solve the multiscale random PDEs using our multiscale data-driven stochastic basis functions. Under some mild conditions, we analysed the error between the numerical solution obtained from our method and the exact solution. We presented several numerical examples for 2D stochastic elliptic PDEs with stochastic multiscale coefficients to demonstrate the accuracy and efficiency of our proposed method.

These numerical examples indicate following advantages of our method: (1) by compressing the physical space and the stochastic space simultaneously, our method can solve the multiscale random PDEs with desirable accuracy on a coarse physical grid and using a few data-driven stochastic basis functions; (2) the data-driven stochastic basis functions can be used to solve the multiscale random PDEs with a class of deterministic force functions; (3) the optimization approach enables us to automatically explore the localized random structures in the solution space.

There are two directions we want to explore in our future work. Firstly, we intend to construct  data-driven stochastic basis functions to compute Helmholtz equation in random media. In addition, we find that the stochastic structure in coefficient $a(x,\omega)$ can influence the data-driven basis function in the stochastic components automatically. As a consequence, we shall investigate how to infer the stochastic structure of coefficient $a(x,\omega)$ through the multiscale data-driven stochastic basis functions, namely solving the multiscale random PDEs in an inverse problem setting.



\section*{Acknowledgements}
\noindent
The research of T. Hou is partially supported by the NSF Grant DMS 1613861.  The research of Z. Zhang is supported by the Hong Kong RGC grants (Projects 27300616, 17300817, and 17300318), National Natural Science Foundation of China (Project 11601457), Seed Funding Programme for Basic Research (HKU), and an RAE Improvement Fund from the Faculty of Science (HKU). We would like to thank Professor Lei Zhang for stimulating discussions.


\appendix
\section{Determine the random dimension of the solution space} \label{sec:StochasticDimension}
\noindent
How to determine the random dimension of the solution space (i.e., $N_{\xi}$ in \eqref{OC_MsStocBasis_Cons1}) \emph{a priori} is an important issue in our method. The random dimension is determined by two factors, (1) the regularity of the coefficient $a^{\varepsilon}(x,\omega)$, especially the ratio of $a_{max}/a_{min}$, and (2) the regularity of the force term $f(x)$. We propose an efficient \emph{a posteriori} approach to estimate $N_{\xi}$, which means that if we choose $N_{\xi}$ in our method, we can get the desired accuracy almost surely.

We assume the force function $f(x)$ belongs to a space $F$, which is spanned by $f_{i}(x)$, $i=1,...,N$, i.e., $F=span\{f_1(x),f_2(x),\cdots,f_N(x)\}$, $N\gg 1$. Let $u^{re}_{i}(x,\omega)=\mathcal{L}^{\varepsilon}(x,\omega)^{-1}f_{i}(x)$ denote the reference solution to the Eq.\eqref{MsStoEllip_OCbasis_Eq} with the force function $f_{i}(x)$. Similarly, let  $u_{i}^{DSM}(x,\omega)=\mathcal{L}^{\varepsilon}_{DSM}(x,\omega)^{-1}f_{i}(x)$ denote the numerical solution obtained using
$N_{\xi}$ basis functions in the stochastic dimension. In the physical dimension, the number of basis functions is equal to $N_x\approx H^{-d}$. The numerical error is $e_{i}(x,\omega)=u^{re}_{i}(x,\omega)-u_{i}^{DSM}(x,\omega)$. Given a fixed $N_{\xi}$, if all the error functions
$||e_{i}(x,\omega)||$ are less than a prescribed threshold, then we shall get our estimate for the random dimension of the solution space. However, this approach is expensive since $N$ is large.

We propose a probabilistic approach to reduce the computational cost in estimating $N_{\xi}$. Specifically, we first generate $r$ Gaussian vectors
$\mathbf{\omega}^{j}=(\omega_{1}^{j},\omega_{2}^{j},\cdots,\omega_{N}^{j})^{T}$, $1\leq j \leq r$, $r\ll N$, where $\omega_{i}^{j}\sim \mathcal{N}(0,1)$ are i.i.d. random variables. Then, we generate $r$ force functions $\hat{f}_{j}(x)=\sum_{i=1}^{N}\omega_{i}^{j} f_{i}(x)$. With this randomly generated  $\hat{f}_{j}(x)$,
we compute $e(x,\omega;\hat{f}_{j})=u^{re}(x,\omega;\hat{f}_{j})-u^{DSM}(x,\omega;\hat{f}_{j})$. Using a similar probabilistic argument (see the Lemma 3.14 and Theorem 3.15 in \cite{woolfe2008fast}), we get
\begin{lemma}\label{range_finding}
Fix a positive integer r and a real number $ \alpha>1$. We generate an independent family of standard Gaussian vectors
$\{\mathbf{\omega}^{j}=(\omega_{1}^{j},\omega_{2}^{j},\cdots,\omega_{N}^{j})^{T}: j=1,2,\cdots,r\}$ and random forces $\hat{f}_{j}(x)=\sum_{i=1}^{N}\omega_{i}^{j} f_{i}(x)$.
Then, except with probability $ \alpha^{-r} $,  we have the estimate,
\begin{align}
||u^{re}(x,\omega)-u^{DSM}(x,\omega)|| \leq \alpha \sqrt{\frac{2}{\pi}}\max \limits_{i=1,\cdots,r}
||u^{re}(x,\omega;\hat{f}_{j})-u^{DSM}(x,\omega;\hat{f}_{j})||.
\end{align}
\end{lemma}
By applying the Lemma \ref{range_finding}, we can estimate the random dimension $N_{\xi}$. Specifically, 
given a fixed $N_{\xi}$, if all $||u^{re}(x,\omega;\hat{f}_{j})-u(x,\omega;\hat{f}_{j})||$, $j=1,...,r$, 
are less than a prescribed threshold, we know  that the numerical error $||u^{re}(x,\omega)-u^{DSM}(x,\omega)||$ can be bounded by
the upper bound of the errors $||u^{re}(x,\omega;\hat{f}_{j})-u^{DSM}(x,\omega;\hat{f}_{j})||$ with the probability $ 1-\alpha^{-r}$. Otherwise, we increase $N_{\xi}$ and repeat the algorithm.

\section*{Reference}
\bibliographystyle{plain}
\bibliography{ZWpaper}

\begin{thebibliography}{10}

\bibitem{Ghanem:08}
M.~Arnst and R.~Ghanem.
\newblock Probabilistic equivalence and stochastic model reduction in
  multiscale analysis.
\newblock {\em Comput. methods Appl. Mech. Engrg}, 197(43):3584--3592, 2008.

\bibitem{Askey:1985}
R.~Askey and J.~Wilson.
\newblock {\em Some basic hypergeometric orthogonal polynomials that generalize
  Jacobi polynomials}, volume 319.
\newblock American Mathematical Soc., 1985.

\bibitem{Zabaras:06}
B.~V. Asokan and N.~Zabaras.
\newblock A stochastic variational multiscale method for diffusion in
  heterogeneous random media.
\newblock {\em Journal of Computational Physics}, 218:654--676, 2006.

\bibitem{BabuskaLipton:2011}
I.~Babuska and R.~Lipton.
\newblock Optimal local approximation spaces for generalized finite element
  methods with application to multiscale problems.
\newblock {\em SIAM Multiscale Model. Simul.}, 9(1):373--406, 2011.

\bibitem{Babuska:07}
I.~Babuska, F.~Nobile, and R.~Tempone.
\newblock A stochastic collocation method for elliptic partial differential
  equations with random input data.
\newblock {\em SIAM J. Numer. Anal.}, 45:1005--1034, 2007.

\bibitem{babuska:04}
I.~Babuska, R.~Tempone, and G.~Zouraris.
\newblock Galerkin finite element approximations of stochastic elliptic partial
  differential equations.
\newblock {\em SIAM J. Numer. Anal.}, 42:800--825, 2004.

\bibitem{cameron:47}
R.~H. Cameron and W.~T. Martin.
\newblock The orthogonal development of non-linear functionals in series of
  {F}ourier-{H}ermite functionals.
\newblock {\em Annals of Mathematics}, pages 385--392, 1947.

\bibitem{CanutoQuarteroni:1982}
C~Canuto and A~Quarteroni.
\newblock Approximation results for orthogonal polynomials in {S}obolev spaces.
\newblock {\em Mathematics of Computation}, 38(157):67--86, 1982.

\bibitem{Charrier:2012}
J.~Charrier.
\newblock Strong and weak error estimates for elliptic partial differential
  equations with random coefficients.
\newblock {\em SIAM Journal on numerical analysis}, 50(1):216--246, 2012.

\bibitem{ChengHouYanZhang:13}
M.~Cheng, T.~Y. Hou, M.~Yan, and Z.~Zhang.
\newblock A data-driven stochastic method for elliptic {PDE}s with random
  coefficients.
\newblock {\em SIAM J. UQ}, 1:452--493, 2013.

\bibitem{cohen2010convergence}
A.~Cohen, R.~Devore, and C.~Schwab.
\newblock Convergence {R}ates of {B}est {N}-term {G}alerkin {A}pproximations
  for a {C}lass of elliptic s{P}{D}{E}s.
\newblock {\em Found. Comput. Math.}, 10(6):615--646, 2010.

\bibitem{Zabaras:07}
B.~Ganapathysubramanian and N.~Zabaras.
\newblock Modelling diffusion in random heterogeneous media: {D}ata-driven
  models, stochastic collocation and the variational multi-scale method.
\newblock {\em Journal of Computational Physics}, 226:326--353, 2007.

\bibitem{Ghanem:91}
R.~Ghanem and P.~Spanos.
\newblock {\em Stochastic finite elements: a spectral approach.}
\newblock Springer-Verlag, New York, 1991.

\bibitem{Grahamquasi:2015}
I.~G. Graham, F.~Y. Kuo, J.~A. Nichols, R.~Scheichl, C.~Schwab, and I.~H.
  Sloan.
\newblock Quasi-{M}onte {C}arlo finite element methods for elliptic {P}{D}{E}s
  with lognormal random coefficients.
\newblock {\em Numerische Mathematik}, 131(2):329--368, 2015.

\bibitem{hou2015heterogeneous}
T.~Hou and P.~Liu.
\newblock A heterogeneous stochastic {FEM} framework for elliptic {PDE}s.
\newblock {\em Journal of Computational Physics}, 281:942--969, 2015.

\bibitem{WuanHou:06}
T.~Y. Hou, W.~Luo, B.~Rozovskii, and H.~M. Zhou.
\newblock Wiener chaos expansions and numerical solutions of randomly forced
  equations of fluid mechanics.
\newblock {\em J. Comput. Phys.}, 216:687--706, 2006.

\bibitem{HouWu:97}
T.~Y. Hou and X.~Wu.
\newblock A multiscale finite element method for elliptic problems in composite
  materials and porous media.
\newblock {\em J. Comput. Phys.}, 134:169--189, 1997.

\bibitem{HouPengchuan:2017}
T.~Y. Hou and P.~Zhang.
\newblock Sparse operator compression of elliptic operators part 1: second
  order elliptic operators.
\newblock {\em Research in Mathematical Sciences}, 1:4--24, 2017.

\bibitem{ZhangHouLiu:15}
T.Y. Hou, P.~Liu, and Z.~Zhang.
\newblock A localized data-driven stochastic method for elliptic {PDE}s with
  random coefficients.
\newblock {\em Bull. Inst. Math. Acad. Sin. (N.S.)}, 1:179--216, 2016.

\bibitem{Karhunen:47}
K.~Karhunen.
\newblock Uber lineare methoden in der {W}ahrscheinlichkeitsrechnung.
\newblock {\em Ann. Acad. Sci. Fennicae. Ser. A. I. Math.-Phys.}, 37:1--79,
  1947.

\bibitem{Kevrekidis:2003}
I.~G. Kevrekidis, C.~W. Gear, J.~M. Hyman, P.~G. Kevrekidid, O.~Runborg, and
  C.~Theodoropoulos.
\newblock Equation-free, coarse-grained multiscale computation: {E}nabling
  mocroscopic simulators to perform system-level analysis.
\newblock {\em Communications in Mathematical Sciences}, 1(4):715--762, 2003.

\bibitem{LiZhangCiCP:18}
S.~Li and Z.~Zhang.
\newblock Computing eigenvalues and eigenfunctions of schr{\"o}dinger equations
  using a model reduction approach.
\newblock {\em Commun. Comput. Phys.}, 24:1073--1100, 2018.

\bibitem{Loeve:78}
M.~Lo\`{e}ve.
\newblock {\em Probability theory. {V}ol. II, 4th ed. {GTM}. 46.}
\newblock Springer-Verlag, ISBN 0-387-90262-7, 1978.

\bibitem{Peterseim:2014}
A.~Malqvist and D.~Peterseim.
\newblock Localization of elliptic multiscale problems.
\newblock {\em Mathematics of Computation}, 83(290):2583--2603, 2014.

\bibitem{matthies:05}
H.~G. Matthies and A.~Keese.
\newblock Galerkin methods for linear and nonlinear elliptic stochastic partial
  differential equations.
\newblock {\em Comput. Method Appl. Mech. Eng.}, 194:1295--1331, 2005.

\bibitem{Najm:09}
H.~N. Najm.
\newblock Uncertainty quantification and polynomial chaos techniques in
  computational fluid dynamics.
\newblock {\em Annual Review of Fluid Mechanics}, 41:35--52, 2009.

\bibitem{Webster:08}
F.~Nobile, R.~Tempone, and C.~Webster.
\newblock A sparse grid stochastic collocation method for partial differential
  equations with random input data.
\newblock {\em SIAM J. Numer. Anal.}, 46:2309--2345, 2008.

\bibitem{Oksendal:13}
B.~Oksendal.
\newblock {\em Stochastic {D}ifferential {E}quations: an introduction with
  applications.}
\newblock Springer Science and Business Media, 2013.

\bibitem{Owhadi:2015}
H.~Owhadi.
\newblock Bayesian numerical homogenization.
\newblock {\em SIAM Multiscale Model. Simul.}, 13(3):812--828, 2015.

\bibitem{Owhadi:2017}
H.~Owhadi.
\newblock Multigrid with rough coefficients and {M}ultiresolution operator
  decomposition from {H}ierarchical {I}nformation {G}ames.
\newblock {\em SIAM Review}, 59(1):99--149, 2017.

\bibitem{sapsis:09}
T.~Sapsis and P.~Lermusiaux.
\newblock Dynamically orthogonal field equations for continuous stochastic
  dynamical systems.
\newblock {\em Physica D: Nonlinear Phenomena}, 238:2347--2360, 2009.

\bibitem{SchwabTodor:03}
C.~Schwab and R.~A. Todor.
\newblock Sparse finite elements for elliptic problems with stochastic loading.
\newblock {\em Numerische Mathematik}, 95:707--734, 2003.

\bibitem{Zabaras:13}
J.~Wan and N.~Zabaras.
\newblock A probabilistic graphical model approach to stochastic multiscale
  partial differential equations.
\newblock {\em Journal of Computational Physics}, 250:477--510, 2013.

\bibitem{Wan:06}
X.~L. Wan and G.~Karniadakis.
\newblock Multi-element generalized polynomial chaos for arbitrary probability
  measures.
\newblock {\em SIAM J. Sci. Comp.}, 28:901--928, 2006.

\bibitem{woolfe2008fast}
F.~Woolfe, E.~Liberty, V.~Rokhlin, and M.~Tygert.
\newblock A fast randomized algorithm for the approximation of matrices.
\newblock {\em Applied and Computational Harmonic Analysis}, 25(3):335--366,
  2008.

\bibitem{Xiu:09}
D.~Xiu.
\newblock Fast numerical methods for stochastic computations: a review.
\newblock {\em Commun. Comput. Phys.}, 5:242--272, 2009.

\bibitem{Xiu:03}
D.~Xiu and G.~Karniadakis.
\newblock Modeling uncertainty in flow simulations via generalized polynomial
  chaos.
\newblock {\em J. Comput. Phys.}, 187:137--167, 2003.

\bibitem{ZhangCiHouMMS:15}
Z.~Zhang, M.~Ci, and T.~Y. Hou.
\newblock A multiscale data-driven stochastic method for elliptic {PDE}s with
  random coefficients.
\newblock {\em SIAM Multiscale Model. Simul.}, 13:173--204, 2015.

\bibitem{ZhangHuHouLinYanCiCP:14}
Z.~Zhang, X.~Hu, T.Y. Hou, G.~Lin, and M.~Yan.
\newblock An adaptive {ANOVA}-based data-driven stochastic method for elliptic
  {PDE}s with random coefficients.
\newblock {\em Commun. Comput. Phys.}, 16:571--598, 2014.

\end{thebibliography}


\end{document}